\documentclass[a4paper, 11pt]{amsart}
\usepackage{amssymb,amsmath,comment,mathrsfs,amsfonts,mathtools}
\usepackage{color}
\usepackage{enumitem}
\usepackage[top=30truemm,bottom=30truemm,left=25truemm,right=25truemm]{geometry}%
\usepackage[normalem]{ulem}

\newcommand{\notmid}{\mathrel{\ooalign{$\mkern-5mu\not$\crcr$|$}}}
\newcommand{\A}{\mathbb A}

\newcommand{\Z}{\mathbb Z}
\newcommand{\Q}{\mathbb Q}
\newcommand{\C}{\mathbb C}
\newcommand{\F}{\mathbb F}
\newcommand{\PS}{\mathbb P}
\newcommand{\Gal}{\mathrm{Gal}}
\newcommand{\Gm}{{\mathbb G}_m}
\newcommand{\Hom}{\mathrm{Hom}}
\newcommand{\Tr}{\mathrm{Tr}}

\theoremstyle{plain}
\newtheorem{thm}{Theorem}[section]
\newtheorem*{thm*}{Theorem}

\newtheorem{prop}[thm]{Proposition}
\newtheorem{lem}[thm]{Lemma}
\newtheorem{cor}[thm]{Corollary}

\theoremstyle{definition}
\newtheorem{defn}[thm]{Definition}
\newtheorem{remark}[thm]{Remark}
\newtheorem{example}[thm]{Example}

\newtheorem*{acknowledgment}{Acknowledgment}

\renewcommand{\labelenumi}{\textup{(\arabic{enumi})}}

\makeatletter
\renewcommand*\l@subsection{\@tocline{2}{0pt}{1.8em}{3.0em}{}}
\makeatother

\numberwithin{equation}{section}

\allowdisplaybreaks[4]

\begin{document}
\title[Multiplicative $n$-ic forms on algebraic varieties]
{Multiplicative $n$-ic forms on algebraic varieties arising from Thaine's generalized Jacobi sums}

\author[A. Hoshi]{Akinari Hoshi}
\address[A. Hoshi]{Department of Mathematics, Niigata University, Niigata 950-2181, Japan}
\email{hoshi@math.sc.niigata-u.ac.jp}

\author[K. Kanai]{Kazuki Kanai}
\address[K. Kanai]{Department of Computer and Information Sciences, Ibaraki University, 4-12-1, Ibaraki, 316-8511, Japan}
\email{kazuki.kanai.du62@vc.ibaraki.ac.jp}

\footnote[0]{\textit{$2020$ Mathematics Subject Classification}. 
11D57, 11R18, 11T22, 11T24, 14M10.}
\footnote[0]{\textit{Keywords and phrases}. 
Jacobi sums, cyclotomy, Gaussian periods,
multiplication matrices, multiplicative quadratic forms,
forms permitting composition, algebraic tori.}
\footnote[0]{This work was partially supported by JSPS KAKENHI Grant Numbers 19K03418, 20H00115, 24K00519, 26K22251.}

\begin{abstract}
We study generalized Jacobi sums, cyclotomic numbers, and $d$-compositions in Thaine's framework, and prove new multiplicative formulae extending Davenport and Hasse's lifting theorem from the classical prime-power setting to products of prime powers.
As applications, motivated by Pfister's theory, we construct regular multiplicative $n$-ic forms ($n \ge 2$) on complete intersections $V\subset\A_K^{l-1}$ of $n$-ics.
For each $n \ge 3$, this yields $n$-ic forms that are multiplicative on $V$, but not on $\A_K^{l-1}$.
Moreover, under $(-1)$-composition, a dense open subset $W \subset V$ has the structure of an algebraic $K$-torus related to a relative norm.
\end{abstract}

\maketitle
\tableofcontents

\section{Introduction}\label{S1}

This paper has two closely related themes: multiplicative formulae for generalized Jacobi sums and multiplicative $n$-ic ($n\geq 2$; quadratic for $n=2$) forms on algebraic varieties. 
These themes are connected through explicit composition laws arising from cyclotomic data.

On the arithmetic side, we study generalized Jacobi sums, generalized multiplication matrices, and $d$-compositions introduced by Thaine \cite{Tha04}, building on the study of $(-1)$-composition and its lifting properties by Hoshi and Kanai \cite{HK22}. 
We establish multiplicative formulae for generalized Jacobi sums corresponding to the $d$-composition of generalized multiplication matrices. 
For prime powers $p^r$ and $q^s$, these formulae combine the classical data into generalized data attached to $p^rq^s$, while repeated $(-1)$-composition for a fixed prime power recovers the lifting considered in \cite{HK22} and the classical Davenport--Hasse lifting phenomenon. 
The resulting formulae also yield compositions of solutions of the Diophantine systems studied by Katre and Rajwade \cite{KR85a} and van Wamelen \cite{Wam02}.

On the geometric side, continuing the study of multiplicative forms on algebraic varieties initiated by Hoshi \cite{Hos03}, we first construct multiplicative quadratic forms on complete intersections of quadrics over a field $K$. Motivated by the relative-norm interpretation of the quadratic case, we then construct regular multiplicative $n$-ic forms ($n\geq 2$) on complete intersections $V\subset\A_K^{l-1}$ of $n$-ics. 
For each $n\geq 3$, we thus obtain regular $n$-ic forms over $K$ that are multiplicative on $V$, but not on $\A_K^{l-1}$.
For each $d\in(\Z/l\Z)^\times$, Thaine's $d$-composition induces a corresponding bilinear operation on $V$ via the power basis of the cyclotomic field $K(\zeta_l)$. 
Moreover, the dense open subset $W\subset V$ consisting of invertible elements for $(-1)$-composition is an algebraic $K$-torus related to a relative norm, and the multiplicative form is compatible with its group law.

Cyclotomy studies cyclotomic numbers, or equivalently the corresponding Jacobi sums, associated with a fixed finite field and a fixed order. 
Much of the classical work has focused on determining these quantities for individual finite fields and orders. Storer \cite{Sto67, Sto69} also investigated cyclotomic numbers for direct sums of finite fields, going beyond the setting of a single prime-power field. In contrast to the traditional determination problem, our focus is on multiplicative relations between cyclotomic-number matrices attached to prime powers through Thaine's $d$-composition.

\subsection{Arithmetic results: generalized Jacobi sums and $d$-compositions}\label{S1-1}

Let $D$ be an integrally closed  domain with $\operatorname{char}(D)=0$, and let $K$ be its quotient field. 
Let $e \geq 2$ be an integer and 
$P(X)=\prod_{i=0}^{e-1}(X-\theta_i), 
P^{\prime}(X)=\prod_{i=0}^{e-1}(X-\theta^{\prime}_i) \in D[X]$ 
be cyclic polynomials of degree $e$, 
i.e. irreducible polynomials with cyclic Galois groups over $K$. 
Then $L=K(\theta_i)$ and $L^{\prime}=K(\theta^{\prime}_i)$ 
are cyclic extensions of $K$ of degree $e$ with 
$\Gal(L/K)=\langle\tau\rangle$, 
$\tau(\theta_i)=\theta_{i+1}$ 
and 
$\Gal(L^{\prime}/K)=\langle\tau^{\prime}\rangle$, 
$\tau^{\prime}(\theta^{\prime}_i)=\theta^{\prime}_{i+1}$ 
where subscripts are taken modulo $e$. 

We assume that $L \cap L^{\prime}= K$. 
We also 
assume that 
$\{\theta_0,\ldots,\theta_{e-1}\}$ 
and 
$\{\theta^{\prime}_0,\ldots,\theta^{\prime}_{e-1}\}$ 
are linearly independent over $K$. 
Then $[L L^{\prime}:K]=e^2$ with 
$G=\Gal(L L^{\prime}/K)\simeq \langle \tau \rangle \times \langle \tau^{\prime} \rangle$ 
and $\{\theta_i\theta^{\prime}_j\mid 0\leq i,j\leq e-1\}$ 
becomes a normal basis of $L L^{\prime}$ over $K$. 
We make the identifications $\tau = (\tau,1)$ 
and $\tau^{\prime} = (1, \tau^{\prime})$.  
For $1\leq d\leq e-1$, 
we have intermediate fields 
$L_d=(L L^{\prime})^{\langle \tau {\tau^{\prime}}^{d} \rangle}=K(\theta_{d,i})$ which are cyclic extensions of $K$ of degree $e$ 
with a normal basis $\{\theta_{d,0},\ldots,\theta_{d,e-1}\}$ 
where $\theta_{d,i} = \Tr_{L L^{\prime}/L_d}(\theta_0\theta^{\prime}_i)= \sum^{e-1}_{s=0}\theta_s \theta^{\prime}_{ds+i}$. 

Let $M_e(K)$ be the algebra of $e\times e$ matrices over $K$. 
Define the matrix $A=[a_{i,j}]_{0\leq i,j \leq e-1}\in M_e(K)$ 
by $\theta_0\theta_i=\sum^{e-1}_{j=0}a_{i,j}\theta_j$,
and call it 
{\it the multiplication matrix of $\theta_0,\ldots,\theta_{e-1}$}.
Note that the $\theta_i$'s are eigenvalues of $A$ and hence 
$P(X)={\rm Char}_X(A)$, the characteristic polynomial of $A$. 

Thaine \cite[Section 2]{Tha04} defined 
the $d$-composition $A\overset{d}{\ast}B$ 
of the matrices $A$ and $B$ as follows. 
For $A=[a_{i,j}]_{0\leq i,j\leq e-1}$, 
$B=[b_{i,j}]_{0\leq i,j\leq e-1}\in M_e(K)$ 
and $d\in(\Z/e\Z)\setminus\{0\}$, 
{\it the $d$-composition $A \overset{d}{\ast} B$ of $A$ and $B$} is defined by 
\[
A \overset{d}{\ast} B:= \left[\sum^{e-1}_{s=0}\sum^{e-1}_{t=0}a_{s,t}b_{ds+i,dt+j}\right]_{0\leq i,j\leq e-1}.
\]
Thaine \cite[Proposition~6]{Tha04} (see also \cite[Example 4]{Tha04}) 
proved that the multiplication matrix $A_d$ 
of $\theta_{d,0},\ldots,\theta_{d,e-1}$ is given by 
\begin{align*}
A_d=A\overset{d}{\ast}A^{\prime}
\end{align*}
where $A=[a_{i,j}]_{0\leq i,j\leq e-1}$ and 
$A^{\prime}=[a^{\prime}_{i,j}]_{0\leq i,j\leq e-1}$ are 
the multiplication matrices of 
$\theta_{0},\ldots,\theta_{e-1}$ 
and $\theta^{\prime}_{0},\ldots,\theta^{\prime}_{e-1}$ respectively.  
It turns out that 
the multiplication matrix $A_d$ corresponds to the intermediate field 
$L_d=(L L^{\prime})^{\langle \tau {\tau^{\prime}}^{d} \rangle}=K(\theta_{d,i}) 
\subset LL^{\prime}$ as explained above. 

Hoshi and Kanai \cite[Section 3]{HK22} studied the $d$-compositions 
$\overset{d}{\ast}$.
In general, the $d$-composition $\overset{d}{\ast}$ 
is neither associative nor commutative, 
although the $(-1)$-composition is associative and commutative. 
Thus we can define {\it the $n$-fold product $A^{(n)}$ of $A$} as 
\begin{align*}
A^{(n)} := A \overset{-1}{\ast} \cdots \overset{-1}{\ast} A\quad(\text{$n$-fold}).
\end{align*}
By using Davenport and Hasse's lifting theorem for Gaussian periods, 
Hoshi and Kanai \cite[Theorem~1.3 and Theorem~1.4]{HK22} proved that
\begin{align} \label{eq:HKmainThm}
C_{p^{nr}} = (-1)^{n-1}C_{p^r}^{(n)}
\end{align}
where $C_{p^r}$ is the multiplication matrix of Gaussian periods $\eta_{p^r}(0),\ldots,\eta_{p^r}(e-1)$ (see Section \ref{S2} for Gaussian periods $\eta_{p^r}(i)$). 
This is the matrix analogue of Davenport and Hasse's lifting theorem.

Let 
$\zeta_e=e^{2\pi i/e}$ be a primitive $e$-th root of unity. 
Thaine \cite{Tha04} introduced generalized Jacobi sums
$E_{{\bf p}}(a,b)\in K(\zeta_e)$
and generalized multiplication matrices $H_{{\bf p}}=[h_{i,j}]_{0\leq i,j \leq e-1}
\in M_e(K)$ for ${\bf p} \in K^{\times}$, 
by abstracting the properties of classical Jacobi sums and multiplication matrices
(see Definition \ref{def:T9and10} for generalized Jacobi sums $E_{{\bf p}}(a,b)$ 
and Definition \ref{def:T9and10_2} for generalized multiplication matrices $H_{{\bf p}}$ for ${\bf p} \in K^{\times}$). 

Thaine's formalism behaves well with respect to products: generalized Jacobi sums and generalized multiplication matrices attached to ${\bf p}_1$ and ${\bf p}_2$ canonically yield those attached to ${\bf p}_1{\bf p}_2$.
In particular, in the prime-power case, classical Jacobi sums for $p^r$ and $q^s$ combine to yield a generalized
Jacobi sum attached to $p^rq^s$.

For integers $i,j$, we define $\delta_{i,j}=1$ if
$i\equiv j\pmod e$, and $\delta_{i,j}=0$ otherwise.
The first main theorem of this paper gives generalized Jacobi sums $(E_{{\bf p}_1}\overset{d}{\ast}E_{{\bf p}_2})(a,b)$ for the product ${\bf p}_1{\bf p}_2 $, together with generalized multiplication matrices for ${\bf p}_1{\bf p}_2$ corresponding to ${\rm Cyc}_{{\bf p}_1}\overset{d}{\ast}{\rm Cyc}_{{\bf p}_2}$ where ${\bf p}_1, {\bf p}_2 \in K^\times$:

\begin{thm}[Multiplicative formulae for generalized Jacobi sums]\label{mainth1}
Let $e \geq 2$ be an integer and $d \in (\Z/e\Z)^{\times}$. 
Let $K$ be a field with 
$\operatorname{char}(K)=0$ and $\Gal(K(\zeta_e)/K)\simeq (\Z/e\Z)^{\times}$, 
and 
let $\sigma_{-d}: K(\zeta_e) \to K(\zeta_e)$, $\zeta_e \mapsto \zeta_e^{-d}$ 
be a $K$-automorphism of $K(\zeta_e)$. 
Let 
$E_{{\bf p}_u}(a,b)$ $( u=1,2)$ be generalized Jacobi sums 
for ${\bf p}_u \in K^\times$ with respect to
an integer $v_u$ satisfying $ev_u/2 \in \Z$.  
Then we have:
\begin{enumerate}
\item The elements
\begin{align*}
(E_{{\bf p}_1}\overset{d}{\ast}E_{{\bf p}_2})(a,b):= -\sigma_{-d}(E_{{\bf p}_1}(a,b))E_{{\bf p}_2}(a,b)+2({\bf p}_1-1)({\bf p}_2-1)\delta_{0,a}\delta_{0,b}
\end{align*}
 are generalized Jacobi sums for ${\bf p}_1{\bf p}_2 \in K^\times$ with respect to $v_1+v_2$;
\item for $f_{{\bf p}_u} = ({\bf p}_u -1)/e\ (u=1,2)$ and
\begin{align*}
{\rm Cyc}_{{\bf p}_u}&:=H_{{\bf p}_u} + [f_{{\bf p}_u} \delta_{i,\frac{ev_u}{2}}]_{0\leq i,j\leq e-1} 
\end{align*}
where the $(i,j)$ entry $H_{{\bf p}_u}[i,j]$ of the matrix $H_{{\bf p}_u}$ 
is given by 
\begin{align*}
H_{{\bf p}_u}[i,j]&=-f_{{\bf p}_u} \delta_{i,\frac{ev_u}{2}}+\frac{1}{e^2}\sum^{e-1}_{a=0} \sum^{e-1}_{b=0} (-1)^{v_ua} \zeta_e^{-(ai+bj)}E_{{\bf p}_u}(a,b),
\end{align*}
we have
\begin{align*}
(E_{{\bf p}_1}\overset{d}{\ast}E_{{\bf p}_2})(a,b)=(-1)^{(v_1+ v_2)a+1}\sum_{i=0}^{e-1}\sum_{j=0}^{e-1}({\rm Cyc}_{{\bf p}_1}\overset{d}{\ast}{\rm Cyc}_{{\bf p}_2})[i,j]\,\zeta_e^{ai+bj}+2({\bf p}_1-1)({\bf p}_2-1)\delta_{0,a}\delta_{0,b}.
\end{align*}
\end{enumerate}

In particular, for a prime power ${\bf p}_1=p^r$ $($resp.  ${\bf p}_2=q^s$$)$ 
and Jacobi sums
$E_{{\bf p}_1}(a,b)=J^\ast_{p^r}(a,b)$ 
$($resp. $E_{{\bf p}_2}(a,b)=J^\ast_{q^s}(a,b)$$)$ for $\F_{p^r}$ 
$($resp. $\F_{q^s}$$)$, we have:
\begin{enumerate}
\item The elements
\begin{align*}
(J^{*}_{p^r}\overset{d}{\ast}J^{*}_{q^s})(a,b):= -\sigma_{-d}(J^\ast_{p^r}(a,b))J^\ast_{q^s}(a,b)+2(p^r-1)(q^s-1)\delta_{0,a}\delta_{0,b}
\end{align*}
 are generalized Jacobi sums for $p^rq^s$;
\item Define $f_r:=(p^r-1)/e$ and $f_s:=(q^s-1)/e$. Let $v_1=f_r$ if $p\ne2$ and $v_1=0$ otherwise,
and let $v_2=f_s$ if $q\ne2$ and $v_2=0$ otherwise.
Then we have
\[
(J^{*}_{p^r}\overset{d}{\ast}J^{*}_{q^s})(a,b)=(-1)^{\lambda a +1} \sum_{i=0}^{e-1}\sum_{j=0}^{e-1}({\rm Cyc}_{p^r}\overset{d}{\ast}{\rm Cyc}_{q^s})[i,j]\,\zeta_e^{ai+bj}+2(p^r-1)(q^s-1)\delta_{0,a}\delta_{0,b}
\]
where $f:=(p^rq^s-1)/e =f_rf_se+f_r+f_s$ and
\[
\lambda = 
\begin{cases}
f & (p\neq 2\ \text{and}\ q\neq2),\\ 
0 & (p=2\ \text{or}\ q=2)
\end{cases}
\]

i.e. the $d$-compositions of cyclotomic numbers ${\rm Cyc}_{p^r}\overset{d}{\ast}{\rm Cyc}_{q^s}$ 
correspond to generalized Jacobi sums $(J^{*}_{p^r}\overset{d}{\ast}J^{*}_{q^s})(a,b)$ up to sign 
$($cf. the equation \eqref{eq:Jr-Cycr} in Section $\ref{S2}$$)$.
\end{enumerate}
\end{thm}

Theorem~\ref{mainth1} shows that classical Jacobi sums are intrinsically attached to prime powers, whereas Thaine's generalized Jacobi sums are multiplicative with respect to products. 
Thus the $d$-composition combines the classical Jacobi sums for $p^r$ and $q^s$ into a generalized Jacobi sum attached to $p^rq^s$, extending the Diophantine systems considered by Katre and Rajwade \cite{KR85a} and van Wamelen \cite{Wam02} beyond the classical prime-power case.

\begin{remark}\label{rem:Sto}
Storer \cite{Sto67}, \cite{Sto69} studied cyclotomic numbers 
with respect to the direct sum $\F_{p^r} \oplus \F_{q^s}$ combinatorially. 
Theorem~\ref{mainth1} gives a generalization of his results 
using the $d$-composition $\overset{d}{\ast}$.
\end{remark}

Applying Theorem~\ref{mainth1} to the case where
${\bf p}_1={\bf p}_2=p^r$ is a prime power
and $d=-1$, 
we obtain the following equivalence between 
(1) Davenport and Hasse's lifting theorem for Jacobi sums and 
(2) the lift of multiplication matrices of Gaussian periods.
\begin{thm}\label{thm:corMT1}
Let $e\ge 2$ and $n\ge 1$ be integers, and let $p^r$ be a prime power
with $p^r\equiv 1\pmod e$.
Let $\chi$ be a multiplicative character of order $e$ on
$\F_{p^r}^{\times}$, and define
$\chi'=\chi\circ N_{\F_{p^{nr}}/\F_{p^r}}$.
Choose generators $\gamma$ of $\F_{p^r}^{\times}$ and
$\gamma'$ of $\F_{p^{nr}}^{\times}$ such that
$\chi(\gamma)=\chi'(\gamma')=\zeta_e$.
Let $C_{p^r}$ and $C_{p^{nr}}$ be the corresponding multiplication
matrices of Gaussian periods of degree $e$.
Then the following conditions are equivalent:
\begin{enumerate}
\item $J_{p^{nr}}^{\ast}((\chi^{\prime})^a, (\chi^{\prime})^b)=(-1)^{n-1}J_{p^r}^{\ast}(a, b)^n$ for all $a,b \in \Z$ such that $e\notmid a,\  e\notmid b$ and $e\notmid a+b$;
\item $C_{p^{nr}} = (-1)^{n-1}C_{p^r}^{(n)}$
\end{enumerate}
where $C_{p^r}^{(n)}$ is the $n$-fold product of $C_{p^r}$ with respect to the $(-1)$-composition $\overset{-1}{\ast}$. 
\end{thm}

Based on Theorem~\ref{thm:corMT1}, Theorem~\ref{mainth1} (1) for generalized Jacobi sums may be regarded, in the sense of Thaine~\cite{Tha04},
as a generalization of Davenport and Hasse's lifting theorem for (classical) Jacobi sums.

\subsection{Geometric results: multiplicative forms on algebraic varieties}\label{S1-2}

In the geometric direction, we use the multiplicative formulae for generalized Jacobi sums to construct multiplicative forms on algebraic $K$-varieties, first in the quadratic case. 
This extends the study of multiplicative quadratic forms on algebraic $K$-varieties in Hoshi~\cite{Hos03}.

Applying Theorem~\ref{mainth1} with $e=l$ for an odd prime $l$ (so that $p^r=lf_r+1$ when ${\bf p}=p^r$), we obtain Theorem~\ref{mainth2}, which constructs multiplicative quadratic forms of dimension $l-1$ on intersections of $\frac{l-1}{2}-1$ quadrics over $K$ and provides the quadratic prototype for the higher-degree case treated in Theorem~\ref{mainth3}.

\begin{thm}[Multiplicative quadratic forms arising from generalized Jacobi sums]\label{mainth2}
Let $l$ be an odd prime. 
Let $K$ be a field with $\operatorname{char}(K)\neq 2$, $l$ and 
$\Gal(K(\zeta_l)/K)\simeq (\Z/l\Z)^{\times}$. 
Let $V \subset K^{l-1}$ be
an 
algebraic $K$-variety defined by
\begin{align*}
f_j({\bf x})=0\quad (j=1,\dots, \tfrac{l-1}{2}-1)
\end{align*}
where
\[
f_j({\bf x})=f_j(x_1,\ldots,x_{l-1})
=\sum_{i=1}^{l-1}x_ix_{i+j}-\sum_{i=1}^{l-1}x_ix_{i+(j+1)}
\]
and each subscript of the $x_{i}$'s should be taken modulo $l$ with the convention $x_0=0$.
Let $q({\bf x})=q(x_1,\ldots,x_{l-1})$ be a regular quadratic form defined by
\[
q({\bf x})=f_0({\bf x})
=\sum_{i=1}^{l-1}x_i^2-\sum_{i=1}^{l-1}x_ix_{i+1}.
\]
Then $q({\bf x})$ is multiplicative on $V$, 
i.e. there exists a bilinear map 
$\varphi :K^{l-1}\times K^{l-1} \rightarrow K^{l-1}$ such that
\[
\varphi(V \times V) \subset V \quad \mbox{and}\quad q({\bf x})q({\bf y})=q(\varphi({\bf x,y}))\mbox{ for any }{\bf x,y} \in V. 
\]
Moreover, for $d \in (\Z/l\Z)^{\times}$, $\varphi_d ({\bf x},{\bf y})=(z_{1,d},\ldots,z_{l-1,d}) \in V$ 
can be given explicitly by
\[
z_{k,d}=-\left(\sum_{i=1}^{l-1}x_{-d^{-1}i}y_{k-i} - \sum_{i=1}^{l-1}x_{-d^{-1}i}y_{-i}\right)\quad (k=1,\ldots,l-1)
\]
where each subscript of the $x_i$'s and the $y_j$'s should be taken modulo $l$ with the convention $x_0=y_0=0$.
\end{thm}

It is also known that the cyclotomic numbers ${\rm Cyc}_{p^r}(a,b)$ are determined by the same system of Diophantine equations as the Jacobi sums for $p^r$. 
In particular, when $p^r \equiv 1 \pmod 3$, the corresponding Diophantine equation reduces to the classical binary quadratic form $4p^r=x_1^2+27x_2^2$ appearing in Gauss's \emph{Disquisitiones Arithmeticae} (see Section~\ref{S4-3}). 
In \cite[Section 7]{HK22}, by computing the $n$-fold $(-1)$-product $C_{p^r}^{(n)}$ of the multiplication matrix $C_{p^r}$ via Thaine's $d$-composition, we obtained explicit lifts of solutions of the corresponding Diophantine systems. 
Thus, for a fixed prime power $p^r \equiv 1 \pmod e$, three lift phenomena run in parallel: Davenport and Hasse's lifting for Jacobi sums, the matrix lifting formula \eqref{eq:HKmainThm} for multiplication matrices of Gaussian periods (and hence for cyclotomic numbers), and explicit lifts of solutions of the associated Diophantine systems. 
We also extend this parallelism from the prime-power case to the more general setting of $d$-compositions.

Theorem~\ref{mainth2} serves as the quadratic model for the higher-degree case.
The passage to higher degree is motivated by the observation that a relative norm relation in a cyclotomic extension yields a multiplicative form on a complete intersection.
To state the higher-degree result, recall that an $n$-ic form
$h(x_1,\dots,x_r)$ over $K$ is called {\it regular} if the associated symmetric $n$-linear form
$\theta$ is nondegenerate in the sense that
\[
\theta(v_1,\dots,v_n)=0 \quad \text{for all } v_2,\dots,v_n \in K^r
\]
implies $v_1=0$; see Shapiro \cite[Chapter~16]{Sha00}.
We also assume that $\operatorname{char}(K) \notmid n!$ where $n!$ denotes the factorial of  an integer $n$. 
This hypothesis is used in passing from a homogeneous $n$-ic form to its associated symmetric $n$-linear form, and in the proof of regularity.

In the quadratic case $n=2$, the classical relation $J^{*}_{p^r}(a,b)\overline{J^{*}_{p^r}(a,b)} = p^r$ may be interpreted as a relative norm relation from the cyclotomic field $L=\Q(\zeta_l)$ to its maximal real subfield $M=\Q(\zeta_l+\zeta_l^{-1})$ with $[L:M]=2$. 
Motivated by this picture, let $l$ be an odd prime and $n \ge 2$ with $n \mid (l-1)$, and let $M$ be the unique intermediate field of $K(\zeta_l)/K$ such that $[K(\zeta_l):M]=n$. 
For an element
\[
H=\sum_{k=0}^{l-1} a_k \zeta_l^k \qquad (a_k \in K)
\]
satisfying
\[
N_{K(\zeta_l)/M}(H)={\bf p} \in K^\times,
\]
we construct multiplicative $n$-ic forms on algebraic $K$-varieties realized as complete intersections
of $m-1$ hypersurfaces of degree $n$:

\begin{thm}[Multiplicative $n$-ic forms arising from generalized Jacobi sums]\label{mainth3}
Let $n\geq2$ be an integer and $l$ be an odd prime with $l \equiv 1\ ({\rm mod}\ n)$. 
Write $l=mn+1$.
Let $K$ be a field with $\operatorname{char}(K) \notmid l(n!)$ and 
$\Gal(K(\zeta_l)/K)\simeq (\Z/l\Z)^{\times}=\langle \gamma \rangle$. 
Define $\beta := \gamma^m$ and $M=L^{\langle \beta \rangle}$ where $L=K(\zeta_l)$, 
i.e. $\Gal(L/M)\simeq \Z/n\Z = \langle \beta \rangle \leq  (\Z/l\Z)^{\times}$.
Let $V \subset K^{l-1}$ be
an algebraic $K$-variety defined by
\begin{align*}
h_j({\bf x})=0\quad (j = 1,\ldots,m-1)
\end{align*}
where
\begin{align*}
h_j({\bf x})&=h_j(x_1,\ldots,x_{l-1})\\
&=\left(\sum_{i_0 + \beta i_1 + \cdots +\beta^{n-1}i_{n-1} \equiv \gamma^j \ (l)}x_{i_0}x_{i_1}\cdots x_{i_{n-1}}\right) - \left(\sum_{i_0 + \beta i_1 + \cdots +\beta^{n-1}i_{n-1} \equiv \gamma^{j+1} \ (l)}x_{i_0}x_{i_1}\cdots x_{i_{n-1}}\right)
\end{align*}
and each subscript of the $x_{i}$'s should be taken modulo $l$ with the convention $x_0=0$.
Let $h({\bf x})=h(x_1,\ldots,x_{l-1})$ be a regular $n$-ic form defined by
\begin{align*}
h({\bf x})
&=\left(\sum_{i_0 + \beta i_1 + \cdots +\beta^{n-1}i_{n-1}\equiv 0 \ (l)}x_{i_0}x_{i_1}\cdots x_{i_{n-1}}\right)
-\left(\sum_{i_0 + \beta i_1 + \cdots +\beta^{n-1}i_{n-1}\equiv \gamma \ (l)}x_{i_0}x_{i_1}\cdots x_{i_{n-1}}\right).
\end{align*}
Then $h({\bf x})$ is multiplicative on $V$.
Moreover, for $d \in (\Z/l\Z)^{\times}$, $\varphi_d ({\bf x},{\bf y})=(z_{1,d},\ldots,z_{l-1,d}) \in V$ 
can be given explicitly by
\[
z_{k,d}=(-1)^{n-1} \left(\sum_{i=1}^{l-1}x_{-d^{-1}i}y_{k-i} - \sum_{i=1}^{l-1}x_{-d^{-1}i}y_{-i}\right)\quad (k=1,\ldots,l-1)
\]
where each subscript of the $x_i$'s and the $y_j$'s should be taken modulo $l$ with the convention $x_0=y_0=0$.
\end{thm}

For each fixed $n\geq 3$, the number of variables in a regular $n$-ic form that is multiplicative on the whole affine space is bounded (see Section~5).
Combining this fact with Theorem~\ref{mainth3}, we obtain the following corollary:

\begin{cor}
\label{cor:no-composition}
In the setting of Theorem~\ref{mainth3},
assume that $n=3$ or that $n\geq 4$ and $l-1>n^2$.
Then $h$ is multiplicative on $V$, but not on $\A_K^{l-1}$.

If, in addition, $\operatorname{char}(K)=0$, then for each fixed $n\geq 3$, 
such regular $n$-ic forms exist over $K$ in $l-1$ variables for arbitrarily large primes $l\equiv 1\pmod n$.
\end{cor}

As shown later in Theorem~\ref{thm:CI},
the varieties $V$ are in fact complete intersections
of $m-1$ hypersurfaces of degree $n$.
The case $l=7$, $n=3$ provides a regular six-variable
cubic form with this property on a cubic hypersurface;
see Example~\ref{ex:l7n3} and Remark~\ref{rem:Sch70}.

For each $d \in (\Z/l\Z)^\times$, we define binary operations ${\bf x}\overset{d}{\ast}{\bf y}:=\varphi_d({\bf x},{\bf y})$ on $V$, where $\varphi_d$ is defined as in Theorem~\ref{mainth3}.
The $d$-composition $\overset{d}{\ast}$ is commutative and associative if and only if $d=-1$ (see Proposition~\ref{prop:CA}).
Thus only $\overset{-1}{\ast}$ can define a commutative algebraic
group law. 
This group law is defined not on the whole variety $V$ but on a dense open subset $W\subset V$.

\begin{thm}[The algebraic $K$-torus $W\subset V$]\label{thm:structureV}
Let $n\geq2$ be an integer and $l$ be an odd prime with $l \equiv 1\ ({\rm mod}\ n)$. 
Write $l=mn+1$.
Let $K$ be a field with $\operatorname{char}(K) \notmid l(n!)$ and 
$\Gal(K(\zeta_l)/K)\simeq (\Z/l\Z)^{\times}=\langle \gamma \rangle$. 
Define $\beta := \gamma^m$ and $M=L^{\langle \beta \rangle}$ where $L=K(\zeta_l)$, 
i.e. $\Gal(L/M)\simeq \Z/n\Z = \langle \beta \rangle \leq  (\Z/l\Z)^{\times}$.
Let $h,h_1,\ldots,h_{m-1}$ and $\varphi_d$ be the $n$-ic forms and bilinear maps defined in Theorem~\ref{mainth3}. 
Let $V \subset K^{l-1}$ be the algebraic $K$-variety defined by
\[
V:=\{{\bf x}\in K^{l-1}\mid h_j({\bf x})=0\ (j=1,\dots,m-1)\}.
\]
For ${\bf x}=(x_1,\dots,x_{l-1})\in V$, define
\[
  \alpha_{\bf x}:=(-1)^{n-1}\sum_{i=1}^{l-1} x_i \zeta_l^{\,i}\in L.
\]
Let
\[
  W:=\{{\bf x}\in V \mid N_{L/K}(\alpha_{\bf x})\neq 0\}=V^\times=V\setminus S
  \quad {\text where}\quad S:=\{{\bf x}\in V \mid N_{L/K}(\alpha_{\bf x})=0\}.
\]
\begin{enumerate}
\item The $K$-variety $W$ equipped with the operation $\overset{-1}{\ast}$ is isomorphic to an algebraic $K$-torus and we have
\[
  W \simeq N_{L/M}^{-1}(\mathbb{G}_{m,K})= \left\{ \alpha \in R_{L/K}(\Gm) \ \Bigm|\  N_{L/M}(\alpha) \in \mathbb{G}_{m,K} \right\}
\]
where $R_{L/K}$ denotes the Weil restriction and $\mathbb{G}_{m,K}$ is regarded as the diagonal subtorus of $R_{M/K}(\Gm)$ induced by $K \subset M$.

\item For ${\bf p}\in K^\times$ we define
\[
W_{\bf p}(K):=\{\,x\in W(K)\mid h(x)={\bf p}\,\}.
\]
For any ${\bf p},{\bf q} \in K^\times$, the group law $\overset{-1}{\ast}$ on $W(K)$ induces a well-defined map
\[
  W_{{\bf p}}(K) \times W_{{\bf q}}(K) \longrightarrow W_{{\bf pq}}(K),
  \qquad (x,y) \longmapsto x \overset{-1}{\ast} y.
\]
\end{enumerate}
\end{thm}

For background on algebraic $K$-tori, we refer to Voskresenskii \cite[Chapter~2]{Vos98} and Manin \cite[Section~30]{Man74}.
The torus description of the dense open subset $W$ also implies that 
$V$ is a complete intersection of $m-1$ $n$-ics (see Theorem~\ref{thm:CI}).
In the quadratic case $n=2$, classical Jacobi sums for prime powers furnish integral points on the fibers $W_{p^r}$.
The $d$-composition then yields integral points on $W_{p^r q^s}$ from those on $W_{p^r}$ and $W_{q^s}$.

\medskip
\medskip
\noindent {\bf Organization of the paper.}~~
We organize this paper as follows.
In Section~\ref{S2} we collect basic definitions and known results on Gaussian periods, Jacobi sums, cyclotomic numbers, and the associated multiplication matrices, and we also review Thaine's generalized Jacobi sums and generalized multiplication matrices of Gaussian periods following Thaine \cite{Tha04}.
In Section~\ref{S3}, we prove Theorems~\ref{mainth1} and~\ref{thm:corMT1} and interpret Theorem~\ref{mainth1} in terms of a twisted double finite Fourier transform, which also yields an alternative proof in a restricted case.
In Section~\ref{S4} we develop the quadratic part of the paper: 
after recalling multiplicative quadratic forms over fields, Pfister forms, the level of fields,
and multiplicative quadratic forms on algebraic $K$-varieties due to Hoshi~\cite{Hos03}, 
we prove Theorem~\ref{mainth2} and apply Theorem~\ref{thm:BF} to obtain explicit lifts of solutions of the Diophantine systems attached to generalized Jacobi sums and to multiplicative quadratic forms on algebraic $K$-varieties.
In Section~\ref{S5} we construct multiplicative $n$-ic forms on affine algebraic $K$-varieties.
Finally, in Section~\ref{S6} we study the induced composition law on the variety $V$ of Section~\ref{S5} and show that an open subset $W\subset V$ has a commutative algebraic $K$-group structure (indeed, an algebraic  $K$-torus).

\medskip
\medskip
\noindent {\bf Notation.}~~
We record here the principal notational conventions used in the main results.

\medskip
\textit{Parameters and Fields:}
\par\smallskip
\begin{tabular}{@{} r @{\quad} l}
 $p^r$: & A prime power.\\
 $e$: & An integer $e \ge 2$ dividing $p^r-1$.\\
 $f_r$: & An integer determined by $p^r-1 = ef_r$. \\
 $l$: & An odd prime number. \\
 $n$: & An integer $n \ge 2$ dividing $l-1$, representing the degree of the form $h(x)$. \\
 $m$: & An integer determined by $l-1 = mn$. \\
 $\zeta_e$: & A primitive $e$-th root of unity. \\
 $K$: & A base field with $\operatorname{char}(K) \notmid l(n!)$ and $\Gal(K(\zeta_l)/K) \simeq (\Z/l\Z)^\times$. \\
 $L$: & The cyclotomic extension $K(\zeta_l)$. \\
 $M$: & The unique intermediate field $K \subset M \subset L$ such that $[L:M]=n$.\\
 $\sigma_a$: & The automorphism in $\Gal(L/K)$ defined by $\sigma_a(\zeta_l) = \zeta_l^a$. \\
 $N_{L/K}$: & The field norm map from $L$ to $K$. \\
\end{tabular}

\medskip
\textit{Sums, matrices, and operations:}
\par\smallskip
\begin{tabular}{@{} r @{\quad} l}
 $J_{p^r}^*(a,b)$: & The modified Jacobi sum for the finite field $\mathbb{F}_{p^r}$. \\
 $E_{\mathbf{p}}(a,b)$: & Thaine's generalized Jacobi sum for an element $\mathbf{p} \in K^\times$. \\
 $C_{p^r}$: & The multiplication matrix of Gaussian periods. \\
 ${\rm Cyc}_{p^r}$: & The matrix of cyclotomic numbers. \\
 $\overset{d}{\ast}$: & Thaine's $d$-composition operation (for matrices or points on $V$). \\
\end{tabular}

\section{Preliminaries on Jacobi sums and Thaine's framework}\label{S2}
\subsection{Classical Jacobi sums and multiplication matrices}\label{S2-1}
We recall some basic definitions, e.g. 
Gaussian periods, period polynomials, Jacobi sums and cyclotomic numbers, 
multiplication matrices of Gaussian periods, 
and their relations (see Berndt, Evans and Williams \cite{BEW98}). 

Let $e \geq 2$ be an integer and 
$p^r$ be a prime power with $p^r \equiv 1\ ({\rm mod}\ e)$. 
Write $p^r=ef_r+1$. 
Let $v_r=f_r$ if $p\ne2$ and $v_r=0$ otherwise.
Let $\F_{p^r}$ be the finite field of $p^r$ elements 
and $\gamma$ be a fixed generator of 
$\F_{p^r}^{\times}=\F_{p^r}\setminus \{0\}$. 
Let $\zeta_p=e^{2\pi i/p}$ be a primitive $p$-th root of unity. 
For $0 \leq i \leq e-1$, 
{\it Gaussian $f_r$-periods $\eta_{p^r}(i)$ of degree $e$ for $\F_{p^r}$} 
are defined by
\begin{align*}
\eta_{p^r}(i) := \sum^{f_r-1}_{j=0}\zeta_p^{\Tr(\gamma^{ej+i})}
\end{align*}
where $\Tr$ is the trace map $\Tr : \F_{p^r} \to \F_p$, 
and {\it the period polynomial $P_{e,p^r}(X)$ of degree $e$ for $\F_{p^r}$} 
is defined by $P_{e,p^r}(X)= \prod^{e-1}_{i=0}\left(X-\eta_{p^r}(i)\right)\in\Z[X]$. 
Note that $\eta_{p^r}(i)$ depends on the choice of $\gamma$. 

For a nontrivial character $\psi$ on $\F_{p^r}^{\times}$ and 
the trivial character $\varepsilon$ on $\F_{p^r}^{\times}$, 
we extend them to $\F_{p^r}$ by setting $\psi(0)=0$ and $\varepsilon(0)=1$. 
Let $\psi_1$, $\psi_2$ be multiplicative characters on $\F_{p^{r}}^{\times}$, extended to $\F_{p^r}$. 
{\it Jacobi sums $J_{p^r}^{\ast}(\psi_1,\psi_2)$ for $\F_{p^r}$} are defined by
\[
J_{p^r}^{\ast}(\psi_1,\psi_2):=\sum_{\substack{\alpha \in \F_{p^r}\\ \alpha \neq 0,1}}\psi_1(\alpha)\psi_2(1-\alpha).
\]
If $\psi_1, \psi_2 \neq \varepsilon$, then 
$J_{p^r}^{\ast}(\varepsilon,\psi_2)=J_{p^r}^{\ast}(\psi_1, \varepsilon)=-1$ and 
$J_{p^r}^{\ast}(\varepsilon, \varepsilon)=p^r-2$.

From now on, we take the character $\chi$ of order $e$ 
on $\F_{p^r}$ with $\chi(\gamma)=\zeta_e$ and $\chi(0)=0$ 
where $\langle\gamma\rangle=\F_{p^r}^\times$. 
For $0 \leq a,b \leq e-1$, 
we simply write 
\[
J_{p^r}^{\ast}(a,b):=J_{p^r}^{\ast}(\chi^a,\chi^b).
\]
Davenport and Hasse's lifting theorem for Jacobi sums claims that 
for any integer $n \geq 1$, if $e$ $\notmid$ $a$, $e$ $\notmid$ $b$ and $e$ $\notmid$ $a+b$ , then
\begin{align*}
J_{p^{nr}}^{\ast}((\chi^{\prime})^a, (\chi^{\prime})^b)=(-1)^{n-1}J_{p^r}^{\ast}(a, b)^n
\end{align*}
where $\chi^{\prime}$ is the lift of $\chi$ 
from $\F_{p^r}$ to $\F_{p^{nr}}$ defined by 
$\chi^{\prime}(\alpha) = \chi(\mathrm{Nr}(\alpha))$ 
where $\mathrm{Nr}$ is the norm map 
(see Davenport and Hasse \cite[Relation (0.8)]{DH35}, 
see also Weil \cite[pages 503--506]{Wei49}, \cite[pages 360--362, Theorem~11.5.2, Corollary~11.5.3]{BEW98}, Hoshi and Kanai \cite[Section 1]{HK22}). 

For $0 \leq i,j \leq e-1$, 
{\it cyclotomic numbers ${\rm Cyc}_{p^r}(i,j)$ of order $e$ for $\F_{p^r}$} 
are defined by
\begin{align*}
{\rm Cyc}_{p^r}(i,j) :=\# \{ (v_1, v_2) \mid 0 \leq v_1, v_2 \leq f_r-1, \ 1+\gamma^{ev_1+i} = \gamma^{ev_2+j}
\text{ in } \mathbb{F}_{p^r} \}.
\end{align*}
We have the following well-known 
relations between cyclotomic numbers ${\rm Cyc}_{p^r}(i,j)$ 
and Jacobi sums $J_{p^r}^{\ast}(a,b)$ 
(see \cite[page 79, Theorem~2.5.1]{BEW98}):
\begin{align}\label{eq:Jr-Cycr}
J_{p^r}^{\ast}(a,b)=(-1)^{v_ra}\sum_{i=0}^{e-1} \sum_{j=0}^{e-1} {\rm Cyc}_{p^r}(i,j) \zeta_e^{ai+bj}
\end{align}
and
\begin{align}\label{eq:Cycr-Jr}
e^2{\rm Cyc}_{p^r}(i,j) =\sum_{a=0}^{e-1} \sum_{b=0}^{e-1} (-1)^{v_ra} \zeta_e^{-(ai+bj)}J_{p^r}^{\ast}(a,b)
\end{align}
Note that both ${\rm Cyc}_{p^r}(i,j)$ and $J_{p^r}^{\ast}(a,b)$ 
depend on a choice of a fixed generator $\gamma$ of $\F_{p^{r}}^{\times}$. 

We see that a product of Gaussian periods is represented 
by a linear combination of Gaussian periods again 
and these coefficients are given in terms of the cyclotomic numbers 
(see \cite[page 328, Lemma~10.10.2,  page 437, Exercise 12.23]{BEW98}):

\begin{align*}
\eta_{p^r}(m)\eta_{p^r}(m+i)=\sum^{e-1}_{j=0} ({\rm Cyc}_{p^r}(i,j) -D_if_r)\eta_{p^r}(m+j)
\end{align*}
where $p\neq 2$ and $D_i=\delta_{0,i}$ (resp. $\delta_{e/2 ,i}$) if $f_r$ is even
(resp. $f_r$ is odd) and $\delta_{i,j}$ is Kronecker's delta.
When the indices are arbitrary integers, $\delta_{i,j}$ is understood modulo $e$.
For $p=2$, we should adopt $D_i=\delta_{0,i}$.
Note that if $p=2$, then $f_r$ and $e$ are odd.

It follows that the Gaussian periods $\eta_{p^r}(0),\ldots,\eta_{p^r}(e-1)$ 
are the eigenvalues of the $e \times e$ matrix 
$C_{p^r} = [C_{p^r}[i,j]]_{0\leq i,j\leq e-1} 
:= [{\rm Cyc}_{p^r}(i,j) - D_if_r]_{0\leq i,j\leq e-1}$ 
which is 
called {\it the multiplication matrix of Gaussian periods 
$\eta_{p^r}(0),\ldots,\eta_{p^r}(e-1)$}. 
Hence the period polynomial $P_{e,p^r}(X)$ can be obtained 
as the characteristic polynomial ${\rm Char}_X(C_{p^r})$ of 
the multiplication matrix $C_{p^r}$. 
By the equations \eqref{eq:Cycr-Jr} and \eqref{eq:Jr-Cycr}, 
we have the following 
relations between $C_{p^r}[i,j]$ 
and  $J_{p^r}^{\ast}(a,b)$: 
\begin{align}\label{eq:Cr-Jr}
C_{p^r}[i,j]
=
-f_rD_i
+\frac{1}{e^2}
\sum_{a=0}^{e-1}\sum_{b=0}^{e-1}
(-1)^{v_ra}\zeta_e^{-(ai+bj)}J_{p^r}^{*}(a,b)
\end{align}
and
\begin{align}\label{eq:Jr-Cr}
J_{p^r}^{\ast}(a,b)= (p^r-1)\delta_{0,b} + (-1)^{v_ra}\sum_{i=0}^{e-1} \sum_{j=0}^{e-1} C_{p^r}[i,j] \zeta_e^{ai+bj}. 
\end{align}

\subsection{Thaine's generalized Jacobi sums and multiplication matrices}

F. Thaine investigated various properties and characterizations of 
Gaussian periods, cyclotomic numbers and Jacobi sums 
with applications to the construction of cyclic polynomials 
in his series of papers \cite{Tha96}, \cite{Tha99}, 
\cite{Tha00}, \cite{Tha01}, \cite{Tha04}, \cite{Tha08} 
(see also Lehmer \cite{Leh88}, 
Schoof and Washington \cite{SW88}, 
Hashimoto and Hoshi \cite{HH05a}, \cite{HH05b}). 

Thaine \cite[Section 3]{Tha04} introduced generalized Jacobi sums $E_{\bf p}(a,b)$ and generalized multiplication matrices $H_{\bf p}$ for ${\bf p}\in K^\times$ by axiomatizing the formal properties of the classical objects recalled in Section \ref{S2-1}. 
We briefly review this framework and record the results that will be used in the proof of Theorem~\ref{mainth1}.

We first recall some basic properties of Jacobi sums and 
the multiplication matrix of Gaussian periods.
\begin{prop}[{see Thaine \cite[Proposition~7]{Tha04}}]\label{prop:T7}
Let $e \geq 2$ be an  
integer and $p^r$ be a prime power with $p^r\equiv 1\ ({\rm mod}\ e)$. 
Define $f_r:=(p^r-1)/e$. Let $v_r=f_r$ if $p\ne2$ and $v_r=0$ otherwise.
For all $a,b,c \in \Z$,
\begin{enumerate}
  \setlength{\parskip}{0cm}
  \setlength{\itemsep}{0cm}
\item $J^\ast_{p^r}(a+e,b)=J^\ast_{p^r}(a,b+e)=J^\ast_{p^r}(a,b)$,
\item $J^\ast_{p^r}(a,b)=J^\ast_{p^r}(b,a)$,
\item $J^\ast_{p^r}(a,b)=(-1)^{v_rb}J^\ast_{p^r}(-a-b,b)$,
\item $J^\ast_{p^r}(a,0)=-1$ if $e$ $\notmid$ $a$,
\item $J^\ast_{p^r}(a,b)J^\ast_{p^r}(-a,-b) = p^r$ 
if $e$ $\notmid$ $a$, $e$ $\notmid$ $b$ and $e$ $\notmid$ $a+b$,
\item $J^\ast_{p^r}(a,b)J^\ast_{p^r}(-a,-c) = (-1)^{v_r(b+c)}J^\ast_{p^r}(-(a+b+c),b)J^\ast_{p^r}(a+b+c,-c)$ if $e$ $\notmid$ $a$, $e$ $\notmid$ $a+b$, 
$e$ $\notmid$ $a+c$ and $e$ $\notmid$ $a+b+c$,
\item if ${\rm gcd}(c,e)=1$ then $\sigma_c (J^\ast_{p^r}(a,b))=J^\ast_{p^r}(ca,cb)$ where $\sigma_c \in \Gal(\Q(\zeta_e)/\Q)$ is defined by $\sigma_c(\zeta_e)=\zeta_e^c$.
\end{enumerate}
\end{prop}

\begin{remark}\label{rem:DefJacobi}
In the prime case $r=1$, Thaine \cite{Tha04} uses a different definition of Jacobi sums 
\[
J_{a,b} := -\sum_{k=2}^{p^r-1}\zeta_e^{a\ {\rm ind}_{s}(k)+b\ {\rm ind}_{s}(1-k)}
\]
where  
$s$ is a primitive root modulo $p^r$ such that $s^{f_1} \equiv \zeta_e \pmod{P}$ 
($P$ is a prime ideal of $\Z[\zeta_e]$ above $p^r$) and 
${\rm ind}_{s}(k)$ is the least non-negative integer such that 
$s^{{\rm ind}_s(k)}\equiv k \pmod{p^r}$. 
In our notation, 
if we take the same primitive root $s$ modulo $p^r$, 
then $J_{a,b}=-J^\ast_{p^r}(a,b)$.
\end{remark}

We take an $e \times e$ circulant matrix $\mathcal{K}=[\delta_{i+1,j}]_{0\leq i,j\leq e-1}$, i.e. 
\begin{align*}
\mathcal{K} =
\left(
\begin{array}{cccc}
 & 1 & &  \\
 & &  \ddots &  \\
 & &  & 1 \\
1 & & & 
\end{array}\right). 
\end{align*}

\begin{prop}[{Thaine \cite[Proposition~8]{Tha04}}]\label{prop:T8}
Let $e\geq 2$ be an integer and let $p^r$ be a prime power such that
$p^r\equiv 1\pmod e$. Define $f_r:=(p^r-1)/e$, and let
$v_r=f_r$ if $p\neq 2$ and $v_r=0$ otherwise. Let
$C_{p^r}=[c_{i,j}]_{0\leq i,j\leq e-1}$ denote the multiplication
matrix of the Gaussian periods $\eta_{p^r}(0),\ldots,\eta_{p^r}(e-1)$, so that $\eta_{p^r}(0)\eta_{p^r}(i)=\sum_{j=0}^{e-1}c_{i,j}\eta_{p^r}(j)$. 
Then the elements $c_{i,j}$ satisfy the following conditions: 
\renewcommand{\labelenumi}{\textnormal{(\roman{enumi})}}
\begin{enumerate}
  \setlength{\parskip}{0cm}
  \setlength{\itemsep}{0cm}
\item $c_{i+e,j}=c_{i,j+e}=c_{i,j}$,
\item $c_{i,j}=c_{j+ev_r/2,\,i+ev_r/2}+f_r(\delta_{0,j}-\delta_{ev_r/2,i})$,
\item $\sum_{k=0}^{e-1}c_{i,k}=f_r-p^r\delta_{ev_r/2,i}$,
\item $\sum^{e-1}_{k=0}c_{k,j}=-\delta_{0,j}$,
\item $c_{i,j}=c_{-i,j-i}$,
\item $\sum^{e-1}_{k=0}c_{i,k}c_{k-j,l-j}=\sum^{e-1}_{k=0}c_{j,k}c_{k-i,l-i}$,
\item  for all $k \in \Z$, $C_{p^r}(\mathcal{K}^{-k}C_{p^r}\mathcal{K}^{k})=(\mathcal{K}^{-k}C_{p^r}\mathcal{K}^{k})C_{p^r}$ where $\mathcal{K}=[\delta_{i+1,j}]_{0\leq i,j\leq e-1}$ is an $e \times e$ circulant matrix.
\end{enumerate}
\end{prop}

Thaine \cite{Tha04} 
defines generalized Jacobi sums $E_{a,b}$ with $E_{0,0}=1, E_{a,0}=1$ 
and generalized multiplication matrices $H_{\bf p}=[h_{i,j}]_{0\leq i,j \leq e-1}\in M_e(K)$ for ${\bf p}\in K^\times$. 
However, we will use the modified version 
$E_{\bf p}(a,b)$ of $E_{a,b}$ 
with $E_{\bf p}(0,0)={\bf p}-2$, $E_{\bf p}(a,0)=-1$ for $e$ $\notmid$ $a$ as follows:

\begin{defn}[{see Thaine \cite[Proposition~9 and Proposition~10]{Tha04}}]\label{def:T9and10}
Let $e \geq 2$ be an integer
and $v$ be an integer with $ve$ even.
Let $K$ be a field with $\operatorname{char}(K)=0$, ${\bf p}\in K^{\times}$ and $\Gal(K(\zeta_e)/K)\simeq(\Z/e\Z)^\times$. 
We consider a map $E_{\bf p}: \Z \times \Z \longrightarrow K(\zeta_e)$ which satisfies the following conditions:
\begin{enumerate}
  \setlength{\parskip}{0cm}
  \setlength{\itemsep}{0cm}
\item $E_{\bf p}(a+e,b)=E_{\bf p}(a,b+e)=E_{\bf p}(a,b)$,
\item $E_{\bf p}(a,b)=E_{\bf p}(b,a)$,
\item $E_{\bf p}(a,b)=(-1)^{vb}E_{\bf p}(-a-b,b)$,
\item $E_{\bf p}(0,0)={\bf p}-2$ and $E_{\bf p}(a,0)=-1$ if $e$ $\notmid$ $a$,
\item $E_{\bf p}(a,b)E_{\bf p}(-a,-b) = {\bf p}$ 
if $e$ $\not{|}$ $a$, $e$ $\notmid$ $b$ and $e$ $\notmid$ $a+b$,
\item $E_{\bf p}(a,b)E_{\bf p}(-a,-c) 
= (-1)^{v(b+c)}E_{\bf p}(-(a+b+c),b)E_{\bf p}(a+b+c,-c)$ \\
if $e$ $\notmid$ $a$, $e$ $\notmid$ $a+b$, $e$ $\notmid$
$a+c$ 
and $e$ $\notmid$ $a+b+c$,
\item if ${\rm gcd}(c,e)=1$, then $\sigma_c (E_{\bf p}(a,b))=E_{\bf p}(ca,cb)$.
\end{enumerate}
We call these elements $E_{\bf p}(a,b) \in K(\zeta_e)$ {\it generalized Jacobi sums  for ${\bf p}$}. 
\end{defn}

\begin{defn}[{see Thaine \cite[Proposition~9 and Proposition~10]{Tha04}}]\label{def:T9and10_2}
Let $e \geq 2$ be an integer
and $v$ be an integer with $ve$ even.
Let $K$ be a field with $\operatorname{char}(K)=0$ and 
${\bf p}\in K^{\times}$. 
Define $f_{\bf p}:=({\bf p}-1)/e$. 
A matrix $H_{\bf p}=[h_{i,j}]_{0\leq i,j\leq e-1} \in M_e(K)$
is called a {\it generalized multiplication matrix for ${\bf p}$} 
if it satisfies the following conditions:
\begin{enumerate}
\item[(i)] $h_{i+e,j}=h_{i,j+e}=h_{i,j}$,
\item[(iii)] $\sum^{e-1}_{k=0}h_{i,k}=f_{\bf p}-{\bf p}\delta_{i,\frac{ev}{2}}$,
\item[(v)] $h_{i,j}=h_{-i,j-i}$,
\item[(vii)]  for all $k \in \Z$, $H_{\bf p}(\mathcal{K}^{-k}H_{\bf p}\mathcal{K}^{k})
=(\mathcal{K}^{-k}H_{\bf p}\mathcal{K}^{k})H_{\bf p}$ where 
$\mathcal{K}=[\delta_{i+1,j}]_{0\leq i,j\leq e-1}$ 
is an $e \times e$ circulant matrix.
\end{enumerate}
\end{defn}

\begin{remark}\label{rem:DefGenJacobi}
In Thaine \cite[Proposition~9 and Proposition~10]{Tha04}, 
the conditions {\rm (i), (iii), (v), (vii)} in Definition \ref{def:T9and10_2} 
are considered together with
\begin{enumerate}
\item[(ii)] $h_{i,j}=h_{j+ve/2,i+ve/2} + f_{\bf p}(\delta_{0,j}-\delta_{i,\frac{ev}{2}})$,
\item[(iv)] $\sum^{e-1}_{k=0}h_{k,j}=-\delta_{0,j}$,
\item[(vi)] $\sum^{e-1}_{k=0}h_{i,k}h_{k-j,l-j}
=\sum^{e-1}_{k=0}h_{j,k}h_{k-i,l-i}$
\end{enumerate}
where $v$ is an integer with $ve$ even. 
However, the conditions {\rm (ii), (iv), (vi)} follow from
{\rm (i), (iii), (v), (vii)} 
as mentioned in Thaine \cite[page 271, Observation of Proposition~9]{Tha04}. 
\end{remark}

The following result, which is a version in terms of $E_p(a,b)$ rather than $E_{a,b}$, is due to Thaine~\cite{Tha04}:

\begin{prop}[{Thaine \cite[Proposition~9 and Proposition~10]{Tha04}}]\label{prop:T9andT10}
Let $e \geq 2$ be an integer and $v$ be an integer with $ve$ even. 
Let $K$ be a field with $\operatorname{char}(K)=0$ and $\Gal(K(\zeta_e)/K)\simeq(\Z/e\Z)^\times$, and let ${\bf p}\in K^{\times}$. 
Define $f_{\bf p}:=({\bf p}-1)/e$. 
Let $E_{\bf p}(a,b)$ be generalized Jacobi sums for ${\bf p}$. 
For $i,j \in \Z$, we define
\begin{equation}\label{eq:3.1}
h_{i,j}=-f_{\bf p}\delta_{i,\frac{ev}{2}}+\frac{1}{e^2}\sum^{e-1}_{a=0} \sum^{e-1}_{b=0} (-1)^{va} \zeta_e^{-(ai+bj)}E_{\bf p}(a,b). \tag{A}
\end{equation}
Then $H_{\bf p}=[h_{i,j}]_{0\leq i,j\leq e-1}$ 
becomes a generalized multiplication matrix for ${\bf p}$
which satisfies 
\begin{equation}\label{eq:3.2}
E_{\bf p}(a,b)=({\bf p}-1)\delta_{0,b} + (-1)^{va}\sum^{e-1}_{i=0} \sum^{e-1}_{j=0} \zeta_e^{ai+bj}h_{i,j}. \tag{B}
\end{equation}
Conversely, if $H_{{\bf p}}=[h_{i,j}]_{0\leq i,j\leq e-1}$ 
is a generalized multiplication matrix for ${\bf p}$ 
and the elements $E_{\bf p}(a,b)$ are defined by \eqref{eq:3.2}, then 
they are generalized Jacobi sums for ${\bf p}$ 
and satisfy \eqref{eq:3.1}.
\end{prop}

\section{{Proof of Theorem~\ref{mainth1}} and {Theorem~\ref{thm:corMT1}: Multiplicative formulae for generalized Jacobi sums}}\label{S3}

In this section, we prove Theorem~\ref{mainth1} and Theorem~\ref{thm:corMT1} and then interpret Theorem~\ref{mainth1} in terms of a twisted double finite Fourier transform. 
The last viewpoint also gives an alternative proof of Theorem~\ref{mainth1} in a restricted case.
\subsection{Proof of Theorem~\ref{mainth1}}

(1) Define
\[
E_{{\bf p}_1{\bf p}_2,d}(a,b)
:=
(E_{{\bf p}_1}\overset{d}{\ast} E_{{\bf p}_2})(a,b)
=
-\sigma_{-d}(E_{{\bf p}_1}(a,b))E_{{\bf p}_2}(a,b)
+2({\bf p}_1-1)({\bf p}_2-1)\delta_{0,a}\delta_{0,b}.
\]
We verify that $E_{{\bf p}_1{\bf p}_2,d}(a,b)$ satisfies the conditions
in Definition~\ref{def:T9and10} with respect to the parameter $v=v_1+v_2$.

First, condition (1) of Definition~\ref{def:T9and10} follows immediately from the periodicity of
$E_{{\bf p}_1}$ and $E_{{\bf p}_2}$:
\[
E_{{\bf p}_1{\bf p}_2,d}(a+e,b)
=
E_{{\bf p}_1{\bf p}_2,d}(a,b),
\qquad
E_{{\bf p}_1{\bf p}_2,d}(a,b+e)
=
E_{{\bf p}_1{\bf p}_2,d}(a,b).
\]

For condition (2), using the symmetry of $E_{{\bf p}_1}$ and
$E_{{\bf p}_2}$, we have
\[
\begin{aligned}
E_{{\bf p}_1{\bf p}_2,d}(a,b)
&=
-\sigma_{-d}(E_{{\bf p}_1}(a,b))E_{{\bf p}_2}(a,b)
+2({\bf p}_1-1)({\bf p}_2-1)\delta_{0,a}\delta_{0,b}\\
&=
-\sigma_{-d}(E_{{\bf p}_1}(b,a))E_{{\bf p}_2}(b,a)
+2({\bf p}_1-1)({\bf p}_2-1)\delta_{0,b}\delta_{0,a}\\
&=
E_{{\bf p}_1{\bf p}_2,d}(b,a).
\end{aligned}
\]

For condition (3), we compute
\[
\begin{aligned}
E_{{\bf p}_1{\bf p}_2,d}(a,b)
&=
-\sigma_{-d}(E_{{\bf p}_1}(a,b))E_{{\bf p}_2}(a,b)
+2({\bf p}_1-1)({\bf p}_2-1)\delta_{0,a}\delta_{0,b}\\
&=
-\sigma_{-d}\bigl((-1)^{v_1b}E_{{\bf p}_1}(-a-b,b)\bigr)
  (-1)^{v_2b}E_{{\bf p}_2}(-a-b,b)
+2({\bf p}_1-1)({\bf p}_2-1)\delta_{0,a}\delta_{0,b}\\
&=
(-1)^{(v_1+v_2)b}
E_{{\bf p}_1{\bf p}_2,d}(-a-b,b).
\end{aligned}
\]

For condition (4), we have
\[
\begin{aligned}
E_{{\bf p}_1{\bf p}_2,d}(0,0)
&=
-({\bf p}_1-2)({\bf p}_2-2)
+2({\bf p}_1-1)({\bf p}_2-1)\\
&=
{\bf p}_1{\bf p}_2-2.
\end{aligned}
\]
If $e\notmid a$, then
\[
\begin{aligned}
E_{{\bf p}_1{\bf p}_2,d}(a,0)
&=
-\sigma_{-d}(E_{{\bf p}_1}(a,0))E_{{\bf p}_2}(a,0)\\
&=
-\sigma_{-d}(-1)(-1)\\
&=
-1.
\end{aligned}
\]

For condition (5), assume that $e\notmid a$, $e\notmid b$, and
$e\notmid a+b$. Then
\[
\begin{aligned}
&E_{{\bf p}_1{\bf p}_2,d}(a,b)
 E_{{\bf p}_1{\bf p}_2,d}(-a,-b)\\
&\quad=
\bigl(-\sigma_{-d}(E_{{\bf p}_1}(a,b))E_{{\bf p}_2}(a,b)\bigr)
\bigl(-\sigma_{-d}(E_{{\bf p}_1}(-a,-b))E_{{\bf p}_2}(-a,-b)\bigr)\\
&\quad=
\sigma_{-d}\bigl(E_{{\bf p}_1}(a,b)E_{{\bf p}_1}(-a,-b)\bigr)
E_{{\bf p}_2}(a,b)E_{{\bf p}_2}(-a,-b)\\
&\quad=
{\bf p}_1{\bf p}_2.
\end{aligned}
\]

For condition (6), assume that $e\notmid a$, $e\notmid a+b$,
$e\notmid a+c$, and $e\notmid a+b+c$.Then
\[
\begin{aligned}
&E_{{\bf p}_1{\bf p}_2,d}(a,b)
 E_{{\bf p}_1{\bf p}_2,d}(-a,-c)\\
&\quad=
\sigma_{-d}\bigl(E_{{\bf p}_1}(a,b)E_{{\bf p}_1}(-a,-c)\bigr)
E_{{\bf p}_2}(a,b)E_{{\bf p}_2}(-a,-c)\\
&\quad=
(-1)^{v_1(b+c)}
\sigma_{-d}\bigl(
E_{{\bf p}_1}(-(a+b+c),b)E_{{\bf p}_1}(a+b+c,-c)
\bigr)\\
&\qquad\qquad\times
(-1)^{v_2(b+c)}
E_{{\bf p}_2}(-(a+b+c),b)E_{{\bf p}_2}(a+b+c,-c)\\
&\quad=
(-1)^{(v_1+v_2)(b+c)}
E_{{\bf p}_1{\bf p}_2,d}(-(a+b+c),b)
E_{{\bf p}_1{\bf p}_2,d}(a+b+c,-c).
\end{aligned}
\]

Finally, for condition (7), let $\gcd(c,e)=1$. Since $\sigma_c$ and
$\sigma_{-d}$ commute in $\operatorname{Gal}(K(\zeta_e)/K)$, we have
\[
\begin{aligned}
\sigma_c(E_{{\bf p}_1{\bf p}_2,d}(a,b))
&=
-\sigma_c\sigma_{-d}(E_{{\bf p}_1}(a,b))
 \sigma_c(E_{{\bf p}_2}(a,b))
+2({\bf p}_1-1)({\bf p}_2-1)\delta_{0,a}\delta_{0,b}\\
&=
-\sigma_{-d}\sigma_c(E_{{\bf p}_1}(a,b))
 E_{{\bf p}_2}(ca,cb)
+2({\bf p}_1-1)({\bf p}_2-1)\delta_{0,a}\delta_{0,b}\\
&=
-\sigma_{-d}(E_{{\bf p}_1}(ca,cb))
 E_{{\bf p}_2}(ca,cb)
+2({\bf p}_1-1)({\bf p}_2-1)\delta_{0,ca}\delta_{0,cb}\\
&=
E_{{\bf p}_1{\bf p}_2,d}(ca,cb).
\end{aligned}
\]
Thus $E_{{\bf p}_1{\bf p}_2,d}(a,b)$ are generalized Jacobi sums for ${\bf p}_1{\bf p}_2$.

(2) Putting $i^{\prime}:=ds+i$ and $j^{\prime}:=dt+j$, we have
\begin{align*}
\sum_{i=0}^{e-1}\sum_{j=0}^{e-1}\zeta_e^{ai+bj}({\rm Cyc}_{{\bf p}_1}\overset{d}{\ast}{\rm Cyc}_{{\bf p}_2})[i,j]
&=\sum_{i=0}^{e-1}\sum_{j=0}^{e-1}\zeta_e^{ai+bj}
\sum^{e-1}_{s=0}\sum^{e-1}_{t=0}{\rm Cyc}_{{\bf p}_1}[s,t]{\rm Cyc}_{{\bf p}_2}[ds+i,dt+j]\\
&=\sum_{i^\prime=0}^{e-1}\sum_{j^\prime=0}^{e-1}\zeta_e^{a(i^\prime-ds)+b(j^\prime-dt)}
\sum^{e-1}_{s=0}\sum^{e-1}_{t=0}{\rm Cyc}_{{\bf p}_1}[s,t]{\rm Cyc}_{{\bf p}_2}[i^\prime,j^\prime]\\
&=\left(\sum^{e-1}_{s=0}\sum^{e-1}_{t=0}\zeta_e^{-(ads+bdt)}{\rm Cyc}_{{\bf p}_1}[s,t]\right)
\left(\sum_{i^\prime=0}^{e-1}\sum_{j^\prime=0}^{e-1}\zeta_e^{ai^\prime+bj^\prime}{\rm Cyc}_{{\bf p}_2}[i^\prime,j^\prime]\right).
\end{align*}
Multiplying both sides by $(-1)^{(v_1+v_2)a+1}$ and noting that
$v_1(d+1)$ is even, we get
\begin{align*}
&(-1)^{(v_1+v_2)a+1}
\sum_{i=0}^{e-1}\sum_{j=0}^{e-1}
\zeta_e^{ai+bj}
({\rm Cyc}_{{\bf p}_1}\overset{d}{\ast}{\rm Cyc}_{{\bf p}_2})[i,j]\\
&=-\left(
(-1)^{v_1a}
\sum^{e-1}_{s=0}\sum^{e-1}_{t=0}
\zeta_e^{-(ads+bdt)}
{\rm Cyc}_{{\bf p}_1}[s,t]
\right)
\left(
(-1)^{v_2a}
\sum_{i^\prime=0}^{e-1}\sum_{j^\prime=0}^{e-1}
\zeta_e^{ai^\prime+bj^\prime}
{\rm Cyc}_{{\bf p}_2}[i^\prime,j^\prime]
\right)\\
&=-(-1)^{v_1a}(-1)^{v_1da}
E_{{\bf p}_1}(-da,-db)E_{{\bf p}_2}(a,b)\\
&=-E_{{\bf p}_1}(-da,-db)E_{{\bf p}_2}(a,b)\\
&=-\sigma_{-d}(E_{{\bf p}_1}(a,b))E_{{\bf p}_2}(a,b)\\
&=E_{{{\bf p}_1{\bf p}_2},d}(a,b)
-2({\bf p}_1-1)({\bf p}_2-1)
\delta_{0,a}\delta_{0,b}.
\end{align*}

For ${\bf p}_1=p^r$ and ${\bf p}_2=q^s$ prime powers, take $v_1$ and $v_2$ as in the statement.
We have $f=(p^rq^s-1)/e=f_rf_se+f_r+f_s$. 
If $p\ne 2$ and $q\ne 2$, then $f_rf_se$ is even. 
Thus we have $(-1)^{(v_1+v_2)a}=(-1)^{fa}=(-1)^{\lambda a}$. 
If $p=2$ or $q=2$, then $e$ is odd and $v_1+v_2$ is even.
Hence $(-1)^{(v_1+v_2)a}=1=(-1)^{\lambda a}$.
\qed\\

\subsection{Proof of Theorem~\ref{thm:corMT1}}

(1) $\Rightarrow$ (2) is proved in 
Hoshi and Kanai \cite[Theorem~1.4]{HK22}.\\

(2) $\Rightarrow$ (1). 
In Theorem~\ref{mainth1} (2), we take ${\bf p}_1={\bf p}_2=p^r$ 
and $d=-1$.
Note that $f_{2r}=\frac{p^{2r}-1}{e}=f_r(p^r+1)$ and 
\[
\lambda_{2r}=
\begin{cases}
f_{2r} & (p\neq 2),\\ 
0 & (p=2).
\end{cases}
\]
Then we have
\begin{align*}
(-1)^{\lambda_{2r} a +1} \sum_{i=0}^{e-1}\sum_{j=0}^{e-1}\zeta_e^{ai+bj}({\rm Cyc}_{p^r}\overset{-1}{\ast}{\rm Cyc}_{p^r})[i,j] = -J^\ast_{p^r}(a,b)^2. 
\end{align*}
By definition of $\lambda_{2r}$,
\begin{align*}
- \sum_{i=0}^{e-1}\sum_{j=0}^{e-1}\zeta_e^{ai+bj}({\rm Cyc}_{p^r}\overset{-1}{\ast}{\rm Cyc}_{p^r})[i,j] = -J^\ast_{p^r}(a,b)^2. 
\end{align*}

For $t=r, 2r$, we define 
\[
D_{t,i}:=
\begin{cases}
[\delta_{ef_t/2,i}]_{0\le i,j\le e-1}&(p\ne2),\\
[\delta_{0,i}]_{0\le i,j\le e-1}&(p=2).
\end{cases}
\]
By ${\rm Cyc}_{p^r} = C_{p^r} + D_{i}f_r$, we have
\[
\begin{aligned}
{\rm Cyc}_{p^r}\overset{-1}{\ast}{\rm Cyc}_{p^r}
&=C_{p^r}^{(2)}
 +2f_r(C_{p^r}\overset{-1}{\ast}D_{r,i})
 +f_r^2(D_{r,i}\overset{-1}{\ast}D_{r,i})\\
&=-C_{p^{2r}}
 +2f_r\bigl([f_r]_{0\le i,j\le e-1}-p^rD_{2r,i}\bigr)
 +ef_r^2D_{2r,i}\\
&=-C_{p^{2r}}-f_{2r}D_{2r,i}
  +[2f_r^2]_{0\le i,j\le e-1}\\
&=-{\rm Cyc}_{p^{2r}}
  +[2f_r^2]_{0\le i,j\le e-1}.
\end{aligned}
\]
where the second equality follows from 
the assumption
(2) $C_{p^{2r}} = -C_{p^r}^{(2)}$. 
Since
\begin{align*}
\sum_{i=0}^{e-1}\sum_{j=0}^{e-1}\zeta_e^{ai+bj}D_{2r,i}[i,j]
=\sum_{j=0}^{e-1}\zeta_e^{bj}=e\delta_{0,b}
\end{align*}
and
\begin{align*}
\sum_{i=0}^{e-1}\sum_{j=0}^{e-1}\zeta_e^{ai+bj}=e^2\delta_{a,0}\delta_{0,b}
\end{align*}
we obtain that 
\begin{align*}
-J^\ast_{p^r}(a,b)^2
&=-\sum_{i=0}^{e-1}\sum_{j=0}^{e-1}\zeta_e^{ai+bj}({\rm Cyc}_{p^r}\overset{-1}{\ast}{\rm Cyc}_{p^r})[i,j]\\
&=\sum_{i=0}^{e-1}\sum_{j=0}^{e-1}\zeta_e^{ai+bj}C_{p^{2r}}[i,j] 
+ f_{2r}\sum_{i=0}^{e-1}\sum_{j=0}^{e-1}\zeta_e^{ai+bj}D_{2r,i}[i,j] - 2f_r^2\sum_{i=0}^{e-1}\sum_{j=0}^{e-1}\zeta_e^{ai+bj}\\
&=-2e^2f_r^2\delta_{a,0}\delta_{0,b} + (p^{2r}-1)\delta_{0,b} + \sum_{i=0}^{e-1}\sum_{j=0}^{e-1}\zeta_e^{ai+bj}C_{p^{2r}}[i,j]\\
&=(-1)^{\lambda_{2r}a}\sum_{i=0}^{e-1}\sum_{j=0}^{e-1}\zeta_e^{ai+bj}C_{p^{2r}}[i,j] + (p^{2r}-1)\delta_{0,b} - 2e^2f_r^2\delta_{a,0}\delta_{0,b}\\
&=J_{p^{2r}}^{\ast}((\chi^{\prime})^a, (\chi^{\prime})^b) - 2e^2f_r^2\delta_{a,0}\delta_{0,b}.
\end{align*}
In the last equality, we have used the fact that $f_{2r}$ is even when $p \neq 2$; 
hence $(-1)^{f_{2r}a}=1$.
For $p =2$, the corresponding sign is also $1$.
Consequently, if $e\notmid a$, $e\notmid b$, and
$e\notmid a+b$, then we have
\[
J_{p^{2r}}^{\ast}((\chi^{\prime})^a, (\chi^{\prime})^b)=-J^\ast_{p^r}(a,b)^2. 
\]
For arbitrary $n$, the $(-1)$-composition is convolution, so assumption (2) gives
\[
\sum_{i,j=0}^{e-1}C_{p^{nr}}[i,j]\zeta_e^{ai+bj}
 =(-1)^{n-1}
 \left(\sum_{i,j=0}^{e-1}C_{p^r}[i,j]\zeta_e^{ai+bj}\right)^n.
\]
If $e\notmid b$, then \eqref{eq:Jr-Cr} gives
\[
(\chi')^a(-1)J_{p^{nr}}^*((\chi')^a,(\chi')^b)
=(-1)^{n-1}\bigl(\chi^a(-1)J_{p^r}^*(a,b)\bigr)^n.
\]
Since $(\chi')^a(-1)=\chi^a(-1)^n$, we obtain
\[
J_{p^{nr}}^*((\chi')^a,(\chi')^b)
 =(-1)^{n-1}J_{p^r}^*(a,b)^n.
\]
In particular, this proves (1) under its stated conditions.
\qed

\subsection{A twisted double finite Fourier transform}
In this subsection, we explain that Theorem~\ref{mainth1} may be viewed as a twisted version of the convolution theorem for the double finite Fourier transform relating cyclotomic numbers and Jacobi sums; see Hoshi and Kanai \cite[Section 4]{HK22}. 
In the restricted case $d=-1$ and both $f_r$ and $f_s$ even, the twist disappears, and the convolution theorem yields an alternative proof of Theorem~\ref{mainth1} (2).

We briefly recall the Fourier transform on finite abelian groups; see Terras \cite[Chapter 10]{Ter99}.

Let $G$ be a finite abelian group and 
\[
L^2(G) = \{f : G \to \C\}
\]
be the vector space of all $\C$-valued functions on $G$ with 
the inner product 
$\langle f,g\rangle=\sum_{x\in G}f(x)\overline{g(x)}$. 
Let $\widehat{G}={\rm Hom}(G,\C^{\times})$ be the dual 
of $G$. 
For $f\in L^2(G\times G)$, 
{\it the double finite Fourier transform $\mathscr{F}(f)=\widehat{f}\in L^2(\widehat{G}\times \widehat{G})$ of $f$} is defined to be 
\begin{align*}
(\mathscr{F}(f))(\chi, \psi) = \widehat{f}(\chi,\psi) =\sum_{x \in G}\sum_{y \in G} f(x,y)\overline{\chi(x)\psi(y)}.
\end{align*} 
Then $\mathscr{F}:L^2(G\times G)\rightarrow L^2(\widehat{G}\times \widehat{G})$ becomes a 
bijective linear transformation with the inverse 
\begin{align*}
(\mathscr{F}^{-1}(\widehat{f}))(x,y) = f(x,y) 
= \frac{1}{(\# G)^2}\sum_{\chi \in \widehat{G}}\sum_{\psi \in \widehat{G}}\widehat{f}(\chi,\psi)\chi(x)\psi(y).
\end{align*}
For $f, g\in L^2(G\times G)$, we define {\it the convolution $f \ast g\in L^2(G\times G)$ 
of $f$ and $g$} by
\[
(f \ast g)(x,y) =\sum_{a \in G}\sum_{b \in G} f(a,b)g(x-a,y-b).
\]
We have the compatibility of the convolution $\ast$ and 
the double finite Fourier transform $\mathscr{F}(f)=\widehat{f}$: 
\begin{align}\label{eq:comb}
\widehat{(f\ast g)}(\chi, \psi)=\widehat{f}(\chi, \psi)\widehat{g}(\chi, \psi). 
\end{align}
(see Terras \cite[page 168, Theorem~2]{Ter99}). 

From now on, we will apply the convolution theorem to our case.
We regard the cyclotomic numbers ${\rm Cyc}_{p^r}(i,j)$ of degree $e$ for $\F_{p^r}$ as the function ${\rm Cyc}_{p^r}:\Z/e\Z \times \Z/e\Z \rightarrow \C$, 
$(i,j)\mapsto {\rm Cyc}_{p^r}(i,j)$ 
and Jacobi sums $J_{p^r}^{\ast}(a,b)$ for $\F_{p^r}$ as the function 
$J_{p^r}^{\ast}:\widehat{\Z/e\Z} \times \widehat{\Z/e\Z} \rightarrow \C$, $(\chi^a, \chi^b) \mapsto J_{p^r}^{\ast}(a,b)$. 
We recall the well-known relations \eqref{eq:Jr-Cycr} and \eqref{eq:Cycr-Jr}: 
\begin{align*}
J_{p^r}^{\ast}(a,b)=(-1)^{v_ra}\sum_{i=0}^{e-1} \sum_{j=0}^{e-1} {\rm Cyc}_{p^r}(i,j) \zeta_e^{ai+bj},\ \quad
{\rm Cyc}_{p^r}(i,j)=\frac{1}{e^2}\sum_{a=0}^{e-1} \sum_{b=0}^{e-1} (-1)^{v_ra} \zeta_e^{-(ai+bj)}J_{p^r}^{\ast}(a,b).
\end{align*}
If $f_r$ is even,
then we find that $J_{p^r}^{\ast}$ and ${\rm Cyc}_{p^r}$
are each other's double finite Fourier transform with some twist. 
Indeed, we have
\begin{align*}
(\mathscr{F}({\rm Cyc}_{p^r}))(\chi^a, \chi^b)
&=\sum_{i \in \Z/e\Z} \sum_{j \in \Z/e\Z} {\rm Cyc}_{p^r}(i,j) \overline{\chi^a(i)\chi^b(j)}\\
&=\sum_{i=0}^{e-1} \sum_{j=0}^{e-1} {\rm Cyc}_{p^r}(i,j) \zeta_e^{-(ai+bj)}\\
&=J_{p^r}^{\ast}(\chi^{-a},\chi^{-b}).\\
\end{align*}
We also have
\begin{align*}
(\mathscr{F}^{-1}(J_{p^r}^{\ast}))(i,j)
&=\frac{1}{e^2}\sum_{\psi_1 \in \widehat{\Z/e\Z}} \sum_{\psi_2 \in \widehat{\Z/e\Z}}J_{p^r}^{\ast}(\psi_1,\psi_2) \psi_1(i)\psi_2(j)\\
&=\frac{1}{e^2}\sum_{a=0}^{e-1} \sum_{b=0}^{e-1}J_{p^r}^{\ast}(\chi^a,\chi^b) \chi^a(i)\chi^b(j)\\
&=\frac{1}{e^2}\sum_{a=0}^{e-1} \sum_{b=0}^{e-1} \zeta_e^{ai+bj}J_{p^r}^{\ast}(a,b)\\
&={\rm Cyc}_{p^r}(-i,-j).
\end{align*}
Together, these relations form the two-dimensional analogue of the relations between Gaussian periods $\eta_{p^r}$
and Gauss sums $G^*_{p^r}$ given by
\begin{align*}
(\mathscr{F}(\eta_{p^r}))(\chi^a)=G_{p^r}^{\ast}(\chi^{-a}), \quad
(\mathscr{F}^{-1}(G_{p^r}^{\ast}))(i)=\eta_{p^r}(-i)
\end{align*}
where $\mathscr{F}$ is the finite Fourier transform on $L^2(\Z/e\Z)$.

Therefore, in the restricted case $d=-1$ and both $f_r$ and $f_s$ even, equation \eqref{eq:comb} gives another proof of Theorem~\ref{mainth1} (2):
\begin{align*}
-\sum_{i=0}^{e-1}\sum_{j=0}^{e-1}({\rm Cyc}_{p^r}\overset{-1}{\ast}{\rm Cyc}_{q^s})[i,j]\,\zeta_e^{ai+bj}
&=-\widehat{({\rm Cyc}_{p^r}\ast {\rm Cyc}_{q^s})}(\chi^{-a},\chi^{-b})\\
&=-\widehat{{\rm Cyc}_{p^r}}(\chi^{-a},\chi^{-b})\widehat{{\rm Cyc}_{q^s}}(\chi^{-a},\chi^{-b})\\
&=-J_{p^r}^{\ast}(a,b)J_{q^s}^{\ast}(a,b)\\
&=J^{*}_{p^rq^s,-1}(a,b)-2(p^r-1)(q^s-1)\delta_{0,a}\delta_{0,b}.
\end{align*}

\section{Proof of Theorem 1.4: Multiplicative quadratic forms on algebraic varieties}\label{S4}

This section is devoted to the quadratic case. 
We begin by recalling the classical theory of multiplicative quadratic forms and its extension to algebraic $K$-varieties. 
We then prove Theorem~\ref{mainth2}, which gives multiplicative quadratic forms of dimension $l-1$ on complete intersections of quadrics, and conclude with applications to explicit lifts of solutions of Diophantine systems.

\subsection{Multiplicative quadratic forms and the level of fields}\label{S4-1}

Let $K$ be a field with $\operatorname{char}(K)\neq 2$. 
Hurwitz \cite{Hur1898} proved that if there exists an identity of the type 
\[
(x_1^2+\cdots+x_n^2)(y_1^2+\cdots+y_n^2)=z_1^2+\cdots+z_n^2
\]
where the $z_k$'s are bilinear forms of the independent variables 
$x_i$ and $y_j$ over $K$, then $n=1,2,4,8$.

In general, for a regular quadratic form 
$q(\mathbf{x}):=q(x_1,\ldots,x_n)=\sum_{i,j=1}^n a_{ij}x_ix_j$ 
of dimension $n$ over $K$, 
i.e. $a_{ij}=a_{ji}\in K$ and
$\det[a_{ij}]_{1\leq i,j\leq n}\neq 0$,
(i) $q(\mathbf{x})$ is called {\it multiplicative} if there exists a formula
\begin{align}\label{eq:multi}
q(\mathbf{x})q(\mathbf{y})=q(\textbf{\textit{z}})
\end{align}
where the $x_i$ and $y_j$ are independent variables and 
$z_k\in K(\mathbf{x},\mathbf{y})$ 
and 
(ii) 
$q(\mathbf{x})$ is called {\it strictly multiplicative} 
if there exists a formula 
\eqref{eq:multi} with $z_k$ linear in $y_j$ over $K(\mathbf{x})$.
It is known that if $q(\mathbf{x})$ is isotropic, then 
$q(\mathbf{x})$ is  
multiplicative and in this case 
$q(\mathbf{x})$ is strictly multiplicative if and only if 
$q(\mathbf{x})$ is hyperbolic 
(see Pfister \cite[Chapter 2]{Pfi95}, Rajwade \cite[Chapter 12]{Raj93}). 
A quadratic form $q$ is called an $m$-fold Pfister form if
$q=\langle 1,a_1\rangle \otimes \cdots \otimes \langle 1,a_m\rangle$,
i.e. a tensor product of binary quadratic forms of the type $\langle 1,a_i\rangle$.
Pfister \cite{Pfi65b} proved the following theorem: 
\begin{thm}[Pfister\ \cite{Pfi65b}]\label{thm:Pfister} 
If $q$ is a Pfister form, then $q$ is strictly multiplicative. 
Conversely, if $q$ is an anisotropic multiplicative form over $K$, 
then $q$ must be a Pfister form. 
\end{thm}
Let $D_K(n)$ be the set of values in $K^{\times}$ represented by 
a sum of $n$ squares in $K$:  
\[
D_K(n)=\{\alpha\in K^{\times}\ |\ \alpha=\alpha_1^2+\cdots+\alpha_n^2,\ \alpha_j\in K\}.
\]
The {\it level} (or {\it Stufe}) of a field $K$ is defined as 
$s(K):=\text{Inf}\{n\in\mathbf{N} \mid -1\in D_K(n)\}$. 
By Theorem~\ref{thm:Pfister}, we see that if $n$ is a power of $2$, 
then $D_K(n)$ becomes a multiplicative group. 
Using this, Pfister \cite{Pfi65a} proved the following remarkable theorem 
(see also Pfister \cite[page 28, Historical Note]{Pfi95}): 

\begin{thm}[{Pfister \cite[Satz 4, Satz 5]{Pfi65a}, see also Rajwade \cite[Chapter 2, Theorem~2.2]{Raj93}, Pfister \cite[Chapter 3, Theorem~1.3, Theorem~1.4]{Pfi95}}]
The level $s(K)$ of a field $K$ is, if finite, always a power of $2$. 
Conversely, every power of $2$ is the level $s(K)$ of some field $K$. 
\end{thm}

For multiplicative quadratic forms, 
see also Pfister \cite{Pfi66}, \cite{Pfi67}, 
Knebusch and Scharlau \cite{KS80}, Scharlau \cite{Sch85}, 
Shapiro \cite{Sha00}, Lam \cite{Lam05}. 

Hoshi \cite{Hos03} defined multiplicative quadratic forms 
on algebraic $K$-varieties as follows:

\begin{defn}[{Hoshi \cite[Section 2]{Hos03}}]\label{def: quad-on-V}
Let $K$ be a field with $\operatorname{char}(K)\neq 2$ and 
$V \subset K^n$ be an algebraic $K$-variety. 
We say a regular
quadratic form $q({\bf x})$ is {\it multiplicative on $V$} 
if there exists a bilinear map 
$\varphi :K^n \times K^n \rightarrow K^n$ such that
\[
\varphi(V \times V) \subset V \quad \mbox{and}\quad q({\bf x})q({\bf y})=q(\varphi({\bf x,y}))\mbox{ for any }{\bf x,y} \in V.
\]
\end{defn}

\begin{remark}\label{rem:morph}
Throughout the paper, all varieties and morphisms are defined over a
base field $K$. 
In the definition above, we formulated the notion of a multiplicative form in terms of $K$-rational points $V(K)$, since our main interest lies in arithmetic applications. 
In our concrete situation, however, the operations $\overset{d}{\ast}$ on $V$ arising from Thaine's $d$-composition are given by polynomial expressions with coefficients in $K$. 
In particular, they define morphisms of varieties over $K$ and extend functorially to $F$-rational points for any field extension $F/K$. Thus one can equivalently view a multiplicative quadratic (or $n$-ic) form as a morphism
\[
  q,h \colon V \longrightarrow \A^1
\]
satisfying the corresponding multiplicative identities
\[
  q \circ \overset{d}{\ast} = m \circ (q \times q),
  \qquad
  h \circ \overset{d}{\ast} = m \circ (h \times h)
\]
where \(m:\mathbb A^1\times\mathbb A^1\to\mathbb A^1\) is the usual multiplication morphism. 
On the open subset where $q$ or $h$ is nonzero, this restricts to the multiplication map on $\Gm$.
\end{remark}

Hoshi \cite[Proposition~1]{Hos03} showed that an $m$-fold Pfister form is multiplicative on a variety $V$ defined as the intersection of $2^{m-1}-1$ quadrics (see also \cite[Example 2]{Hos03}). 
Moreover, Hoshi \cite[Theorem~4, Theorem~6]{Hos03} 
constructed {\it non}-Pfister forms of dimension $4$ and $6$ 
which are multiplicative on certain algebraic $K$-varieties. 

The next subsection extends the construction of multiplicative
quadratic forms on algebraic $K$-varieties in Hoshi~\cite{Hos03} to dimension $l-1$, where $l$ is an odd prime, and proves Theorem~\ref{mainth2} via the multiplicative formulae for generalized Jacobi sums of Theorem~\ref{mainth1}.

\subsection{Proof of Theorem~\ref{mainth2}}\label{S4-2}

We now prove Theorem~\ref{mainth2}. 
The key point is that Theorem~\ref{mainth1} yields explicit bilinear formulae compatible with the quadratic relations defining $V$, so that the multiplicative formulae for generalized Jacobi sums show that $V$ is closed under the corresponding composition law.
We first isolate the identities needed for the quadratic construction, then prove that the resulting bilinear map preserves the defining equations of $V$, and finally deduce Theorem~\ref{mainth2}.

We begin with an elementary lemma:

\begin{lem} [{see \cite[Theorem~3.9.1]{BEW98}, \cite[Theorem~(i), (ii), page 186]{KR85a}}] \label{lem:BEW}
Let $l$ be an odd prime. 
Let $K$ be a field with $\operatorname{char}(K)\neq 2$, $l$ and 
$\Gal(K(\zeta_l)/K)\simeq (\Z/l\Z)^{\times}$.  
Let ${\bf p}\in K^{\times}$ and
\[
\alpha=\sum_{k=0}^{l-1}a_k\zeta_l^k \quad (a_k \in K)\\
\]
with
\[
\alpha\overline{\alpha}= {\bf p}
\]
where $\overline{\alpha}=\sigma_{-1}(\alpha)$ and $\sigma_{-1} \in \Gal(K(\zeta_l)/K)$ 
satisfies $\sigma_{-1}(\zeta_l)=\zeta_l^{-1}$. 
Then we have
\begin{enumerate}
\item ${\bf p} =\sum_{k=0}^{l-1}a_k^2 - \sum_{k=0}^{l-1}a_ka_{k+1}$,
\item $\sum_{k=0}^{l-1}a_ka_{k+1}=\sum_{k=0}^{l-1}a_ka_{k+j}$ $(2\leq j\leq (l-1)/2)$ 
\end{enumerate}
where each subscript of the $a_k$'s should be taken modulo $l$. 
Moreover, we may assume that $a_0=0$ without loss of generality.
\end{lem}
\begin{proof}
Since $\operatorname{char}(K)\neq l$, we have
$1+\zeta_l+\cdots+\zeta_l^{l-1}=0$.
Thus, replacing each $a_k$ by $a_k-a_0$ leaves
$\alpha$ unchanged, so we may assume $a_0=0$.

By $\alpha\overline{\alpha}= {\bf p}$, we have 
\[
\left( a_1\zeta_l +  \cdots +a_{l-1}\zeta_l^{l-1}\right) \left(a_1\zeta_l^{l-1} + \cdots + a_{l-1}\zeta_l\right) = {\bf p},
\]
that is,
\[
S_0 + S_1\left(\zeta_l + \zeta_l^{l-1}\right) +  \cdots + S_{(l-1)/2}\left(\zeta_l^{(l-1)/2} + \zeta_l^{l-(l-1)/2}\right)= {\bf p}
\]
where $S_j = \sum_{i-k=j}a_ia_{k}$ ($j=0,\ldots,l-1$).
Note that  $S_j =S_{l-j}$.
By $\zeta_l+\cdots+\zeta_l^{l-1}=-1$, we have
\[
(S_1 - S_0 + {\bf p})\left(\zeta_l + \zeta_l^{l-1}\right) + 
 \cdots + (S_{(l-1)/2} - S_0 + {\bf p})\left(\zeta_l^{(l-1)/2} + \zeta_l^{l-(l-1)/2}\right) = 0.
\]
Because it follows from the assumption 
$\Gal(K(\zeta_l)/K)\simeq (\Z/l\Z)^{\times}$ that 
$\zeta_l,\zeta_l^2, \dots, \zeta_l^{l-1}$ are linearly independent over $K$, 
we have
\[
S_1 - S_0 + {\bf p} = S_j - S_0 + {\bf p}=0\ \ (2\leq j\leq (l-1)/2).
\]
Hence ${\bf p}=S_0-S_1$ and $S_1=S_j$ for $2\le j\le (l-1)/2$, which proves the lemma.
\end{proof}

Theorem~\ref{mainth2} follows from the following theorem: 

\begin{thm}\label{thm:BF}
Let $l$ be an odd prime and $d \in (\Z/l\Z)^{\times}$. 
Let $K$ be a field with $\operatorname{char}(K)=0$ and $\Gal(K(\zeta_l)/K)\simeq (\Z/l\Z)^{\times}$.
Let
$E_{{\bf p}_1}(a,b)=\sum_{i=0}^{l-1}a_i\zeta_l^i$ and $E_{{\bf p}_2}(a,b)=\sum_{j=0}^{l-1}b_j\zeta_l^j$
be
generalized Jacobi sums for 
${\bf p}_1$, ${\bf p}_2\in K^\times$ respectively. 
For $a,b\in \Z$ with $l\notmid a$, $l\notmid b$, and $l\notmid a+b$,
let $E_{{\bf p}_1{\bf p}_2,d}(a,b)=-\sigma_{-d}(E_{{\bf p}_1}(a,b))E_{{\bf p}_2}(a,b)$
be a generalized Jacobi sum for ${\bf p}_1{\bf p}_2\in K^\times$ 
$($see Theorem~\ref{mainth1}$)$.
Then $E_{{\bf p}_1{\bf p}_2,d}(a,b)=\sum_{k=0}^{l-1}c_k\zeta_l^k$ 
is  given by
$$c_k=-\sum_{i=0}^{l-1}a_{-d^{-1}i}b_{k-i}$$
and the coefficients $c_{0},\dots,c_{l-1}$ satisfy
\begin{enumerate}
\item ${\bf p}_1{\bf p}_2=\sum_{k=0}^{l-1}c_k^2 - \sum_{k=0}^{l-1}c_kc_{k+1}$,
\item $\sum_{k=0}^{l-1}c_kc_{k+1}
=\sum_{k=0}^{l-1}c_kc_{k+j}$\ $(2\leq j\leq (l-1)/2)$
\end{enumerate}
with all indices taken modulo $l$.
Moreover, we may assume $a_0=b_0=c_0=0$ 
without loss of generality and in this case 
$E_{{\bf p}_1{\bf p}_2,d}(a,b)=\sum_{k=1}^{l-1}c_k\zeta_l^k$ 
is  given by 
$$c_k=-\left( \sum_{i=1}^{l-1}a_{-d^{-1}i}b_{k-i} - \sum_{i=1}^{l-1}a_{-d^{-1}i}b_{-i}\right).$$
\end{thm}
\begin{proof}
By Theorem~\ref{mainth1} (1), the element
\[
E_{{\bf p}_1{\bf p}_2,d}(a,b):=-\sigma_{-d}(E_{{\bf p}_1}(a,b))E_{{\bf p}_2}(a,b)
\]
is a generalized Jacobi sum for ${\bf p}_1{\bf p}_2$. 
In particular, since $l\notmid a$, $l\notmid b$, and $l\notmid a+b$, condition (5) in
Definition~\ref{def:T9and10} gives
\[
E_{{\bf p}_1{\bf p}_2,d}(a,b)E_{{\bf p}_1{\bf p}_2,d}(-a,-b) =E_{{\bf p}_1{\bf p}_2,d}(a,b)\overline{E_{{\bf p}_1{\bf p}_2,d}(a,b)}= {\bf p}_1{\bf p}_2.
\]
Moreover, writing
\[
E_{{\bf p}_1}(a,b)=\sum_{i=0}^{l-1} a_i\zeta_l^i,
\qquad
E_{{\bf p}_2}(a,b)=\sum_{j=0}^{l-1} b_j\zeta_l^j,
\]
we obtain
\[
\sigma_{-d}(E_{{\bf p}_1}(a,b))
=\sigma_{-d}\left(\sum_{i=0}^{l-1}a_i\zeta_l^i\right)
=\sum_{i=0}^{l-1}a_i\zeta_l^{-di}
=\sum_{i=0}^{l-1}a_{-d^{-1}i}\zeta_l^i.
\]
Thus we have
\begin{align*}
E_{{\bf p}_1{\bf p}_2,d}(a,b)
&= -\sigma_{-d}(E_{{\bf p}_1}(a,b))E_{{\bf p}_2}(a,b)\\
&=-\Bigl(\sum_{i=0}^{l-1}a_{-d^{-1}i}\zeta_l^i\Bigr) \Bigl(\sum_{j=0}^{l-1}b_j\zeta_l^j\Bigr)\\
&= -\sum_{k=0}^{l-1}\left(\sum_{i+j\equiv k \,({\rm mod}\,l)}a_{-d^{-1}i}b_j\right)\zeta_l^k \\
&=-\sum_{k=0}^{l-1}\left(\sum_{i=0}^{l-1}a_{-d^{-1}i}b_{k-i}\right)\zeta_l^k\\
&=\sum_{k=0}^{l-1}c_k\zeta_l^k.
\end{align*}

Hence the first assertion follows from Lemma~\ref{lem:BEW}. 
For the reduced form, assume $a_0=b_0=0$. 
Using $1=-(\zeta_l+\cdots+\zeta_l^{l-1})$,
we can eliminate the constant term and rewrite
\begin{align*}
E_{{\bf p}_1{\bf p}_2,d}(a,b)
&= -\sigma_{-d}(E_{{\bf p}_1}(a,b))E_{{\bf p}_2}(a,b)\\
&=-\Bigl(\sum_{i=1}^{l-1}a_{-d^{-1}i}\zeta_l^i\Bigr) \Bigl(\sum_{j=1}^{l-1}b_j\zeta_l^j\Bigr)\\
&= -\sum_{k=0}^{l-1}\left(\sum_{i+j\equiv k \,({\rm mod}\,l)}a_{-d^{-1}i}b_j\right)\zeta_l^k \\
&= -\sum_{k=1}^{l-1}\left(\sum_{i+j\equiv k \,({\rm mod}\,l)}a_{-d^{-1}i}b_j\right)\zeta_l^k 
- \left( \sum_{i+j=l}a_{-d^{-1}i}b_j\right)\\
&=-\sum_{k=1}^{l-1}\left(\sum_{i=1}^{l-1}a_{-d^{-1}i}b_{k-i} - \sum_{i=1}^{l-1}a_{-d^{-1}i}b_{-i}\right)\zeta_l^k\\
&=\sum_{k=1}^{l-1}c_k\zeta_l^k.
\end{align*}
\end{proof}

\begin{remark}\label{rem:base_change}
In the reduced case $a_0=b_0=c_0=0$, the formula for $c_k$ in Theorem~\ref{thm:BF} is obtained by formal expansion and coefficient comparison, and hence is a polynomial identity in the coefficients $a_i$ and $b_j$. 
The same argument as in the proof of Theorem~\ref{thm:BF} applies after any base change and specialization for which the relations in Lemma~\ref{lem:BEW} are satisfied.
\end{remark}

{\it Proof of Theorem~\ref{mainth2}}. 

We first see that the quadratic form
\begin{align*}
q({\bf x})=f_0({\bf x})=\sum_{i=1}^{l-1}x_i^2-\sum_{i=1}^{l-1}x_ix_{i+1}
=(x_1,\ldots,x_{l-1})\,A_{l-1}
\left( 
\begin{array}{c}
x_1\\
\vdots\\
x_{l-1}
\end{array}
\right)
\end{align*}
is regular because it follows from a linear recurrence relation 
det\! $A_{n+2}=$ det\! $A_{n+1}-\frac{1}{4}$ det\! $A_n$ 
with det\! $A_2=\frac{3}{4}$ and det\! $A_1=1$ 
that det\! $A_{l-1}=l/2^{l-1}\neq 0$ 
where $A_{l-1}$ is the corresponding $(l-1)\times (l-1)$ symmetric matrix: 

\renewcommand{\arraystretch}{1.5}
\begin{align*}
A_{l-1}=\left( 
\begin{array}{ccccc}
1 & -\tfrac{1}{2} & 0 & \cdots & 0\\
-\tfrac{1}{2} & 1 & -\frac{1}{2} & \ddots & \vdots\\
0 & -\tfrac{1}{2} & 1 & \ddots & 0\\
\vdots & \ddots & \ddots & \ddots & -\tfrac{1}{2}\\
0 & \cdots & 0 & -\tfrac{1}{2} & 1
\end{array}
\right).
\end{align*}
\renewcommand{\arraystretch}{1.0}

By Remark~\ref{rem:base_change}, for $x,y\in V$ we may specialize $a_i=x_i$ and $b_j=y_j$ in the reduced formulae of Theorem~\ref{thm:BF}.
Thus the map $\varphi_d$ defined by
\[
z_k=-\left(\sum_{i=1}^{l-1} x_{-d^{-1}i}y_{k-i}-\sum_{i=1}^{l-1} x_{-d^{-1}i}y_{-i}\right)
\]
satisfies $\varphi_d(V\times V)\subset V$ and $q(\varphi_d(x,y))=q(x)q(y)$.
Hence $q$ is multiplicative on $V$.
\qed
\begin{remark}\label{rem:reg}
We can consider an analog of Theorem~\ref{mainth2} 
when $a_0=x_0\neq0$ and $b_0=y_0\neq 0$ 
as in the first half of Theorem~\ref{thm:BF}. 
However, we see that the corresponding quadratic form
\begin{align*}
q({\bf x})=f_0(x_0,\ldots,x_{l-1})
=\sum_{i=0}^{l-1}x_i^2-\sum_{i=0}^{l-1}x_ix_{i+1}
=(x_0,\ldots,x_{l-1})\,B_l
\left( 
\begin{array}{c}
x_0\\
\vdots\\
x_{l-1}
\end{array}
\right)
\end{align*}
is {\it not} regular, since each row and column sums to zero, and hence det\! $B_l=0$ 
where $B_l$ is the corresponding $l\times l$ symmetric matrix:
\renewcommand{\arraystretch}{1.5}
\begin{align*}
B_l=\left(
\begin{array}{c|ccccc}
1 & -\tfrac{1}{2} & 0 & \cdots & 0 & -\tfrac{1}{2}\\\hline
-\tfrac{1}{2}& 1 & -\tfrac{1}{2} & 0 & \cdots & 0\\
0 & -\tfrac{1}{2} & 1 & -\frac{1}{2} & \ddots & \vdots\\
\vdots & 0 & -\tfrac{1}{2} & 1 & \ddots & 0\\
0 & \vdots & \ddots & \ddots & \ddots & -\tfrac{1}{2}\\
-\tfrac{1}{2} & 0 & \cdots & 0 & -\tfrac{1}{2} & 1
\end{array}
\right).
\end{align*}
\renewcommand{\arraystretch}{1.0}
Thus, to obtain a regular quadratic form, one has to eliminate one variable, for instance by imposing $x_0=0$. 
This is precisely the reduced situation in the second half of Theorem~\ref{thm:BF}, where one also normalizes the coefficients by $a_0=b_0=0$.
\end{remark}

\begin{example}[Examples of Theorem~\ref{mainth2} for $l=3,5,7$]\label{ex:mainth2}
Let $K$ be a field with $\operatorname{char}(K)\neq 2$, $l$ and 
$\Gal(K(\zeta_l)/K)\simeq (\Z/l\Z)^{\times}$. 
Each subscript of the $x_i$'s and the $y_j$'s 
should be taken modulo $l$
and we take $x_0=y_0=0$ 
(for the regularity of $q({\bf x})$, see Remark \ref{rem:reg}). \\

(1) $l=3$. It follows from Theorem~\ref{mainth2} that 
the quadratic form $q({\bf x})=q(x_1,x_{2})$ defined by
\[
q({\bf x})=f_0({\bf x})
=\sum_{i=1}^{2}x_i^2-\sum_{i=1}^{2}x_ix_{i+1}
=(x_1^2+x_2^2) - (x_1x_2+x_2x_0)
=x_1^2+x_2^2 - x_1x_2
\]
(with the convention $x_0=0$) is multiplicative on $V=K^2$. 

Indeed, for $d \in (\Z/3\Z)^\times$, 
we obtain an explicit bilinear map $\varphi_d :K^2 \times K^2 \rightarrow K^2$ such that 
\begin{align*}
q({\bf x})q({\bf y})=q(\varphi_d ({\bf x}, {\bf y}))
\end{align*}
for ${\bf x},{\bf y} \in K^2$.
It is given explicitly by
\[
\varphi_d ({\bf x}, {\bf y})
={\bf z}_d
= (z_{1,d}, z_{2,d})
\]
where $z_{k,d}=-\left(\sum_{i=1}^{2}x_{-d^{-1}i}y_{k-i} - \sum_{i=1}^{2}x_{-d^{-1}i}y_{-i}\right)$.

For example, we have for $d=1 \in (\Z/3\Z)^\times$,
\begin{align*}
z_{1,1}
=- x_1y_2 + (x_2y_2 + x_1y_1),\\
z_{2,1}
=-x_2y_1+(x_2y_2 + x_1y_1).
\end{align*}

(2) $l=5$. 
We take an algebraic $K$-variety $V \subsetneq K^4$ defined by
$f_1({\bf x})=0$
where
\begin{align*}
f_1({\bf x})
&=\sum_{i=1}^{4}x_ix_{i+1}-\sum_{i=1}^{4}x_ix_{i+2}\\
&=(x_1x_2 + x_2x_3 + x_3x_4) - (x_1x_3 + x_2x_4 + x_4x_1).
\end{align*}
By Theorem~\ref{mainth2}, the quadratic form 
\begin{align*}
q({\bf x})=f_0({\bf x})
=\sum_{i=1}^{4}x_i^2-\sum_{i=1}^{4}x_ix_{i+1}
=(x_1^2+x_2^2+x_3^2+x_4^2) - (x_1x_2+x_2x_3 + x_3x_4)
\end{align*}
is multiplicative on $V=\{{\bf x}\in K^4\mid f_1({\bf x})=0\}$. 
Indeed, for $d \in (\Z/5\Z)^\times$, we obtain an explicit bilinear map  $\varphi_d :K^4 \times K^4 \rightarrow K^4$ such that 
\begin{align*}
q({\bf x})q({\bf y})=q(\varphi_d ({\bf x}, {\bf y}))
\end{align*}
for ${\bf x},{\bf y} \in V$.
It is given explicitly by
\[
\varphi_d ({\bf x}, {\bf y})
={\bf z}
= (z_{1,d}, z_{2,d}, z_{3,d}, z_{4,d})
\]
where $z_{k,d}=-\left(\sum_{i=1}^{4}x_{-d^{-1}i}y_{k-i} - \sum_{i=1}^{4}x_{-d^{-1}i}y_{-i}\right)$.

For example, we have for $d=1 \in (\Z/5\Z)^\times$,
\begin{align*}
z_{1,1}
&=-(x_3y_4 + x_2y_3 + x_1y_2 )
+(x_4y_4 + x_3y_3 + x_2y_2 + x_1y_1 ),\\
z_{2,1}
&=-(x_{4}y_1 + x_2y_4 + x_1y_3 )
+(x_{4}y_4 + x_3y_3 + x_2y_2 + x_1y_1 ),\\
z_{3,1}
&=-(x_4y_2 + x_3y_1 + x_1y_4 )
+(x_4y_4 + x_3y_3 + x_2y_2 + x_1y_1 ),\\
z_{4,1}
&=-(x_{4}y_3 + x_3y_2 + x_2y_1)
+(x_{4}y_4 + x_3y_3 + x_2y_2 + x_1y_1 ).
\end{align*}

Moreover, for $d=\pm 1\in (\Z/5\Z)^\times$, 
one can check the following identities, which in particular imply that $q({\bf x})$ is multiplicative on $V$, i.e. 
$q({\bf x})q({\bf y})=q(\varphi_d({\bf x}, {\bf y}))$ for 
${\bf x},{\bf y} \in V=\{{\bf x}\in K^4\mid f_1({\bf x})=0\}$:
\begin{align*}
q(\varphi_d({\bf x},{\bf y}))&=
q({\bf x})q({\bf y})+f_1({\bf x})f_1({\bf y}),\\
f_1(\varphi_d({\bf x},{\bf y}))&=
q({\bf x})f_1({\bf y})+f_1({\bf x})q({\bf y})+f_1({\bf x})f_1({\bf y}).
\end{align*}
Note that $q({\bf x})$ and $f_1({\bf x})$ are invariant 
under the change $d=1\mapsto -1$. 
For $d=\pm 2\in (\Z/5\Z)^\times$, we obtain that 
\begin{align*}
q(\varphi_d({\bf x},{\bf y}))&=
q({\bf x})q({\bf y})+f_1({\bf x})q({\bf y})-f_1({\bf x})f_1({\bf y}),\\
f_1(\varphi_d({\bf x},{\bf y}))&=
q({\bf x})f_1({\bf y})-f_1({\bf x})q({\bf y})
\end{align*}
because $q({\bf x})\mapsto q({\bf x})+f_1({\bf x})$, 
$f_1({\bf x})\mapsto -f_1({\bf x})$ after changing $d=1\mapsto \pm 2\in (\Z/5\Z)^\times$.\\

(3) $l=7$. 
We consider an algebraic $K$-variety $V \subsetneq K^6$ defined by
$f_1({\bf x})=0$, $f_2({\bf x})=0$
where
\begin{align*}
f_1({\bf x})
&=\sum_{i=1}^{6}x_ix_{i+1}-\sum_{i=1}^{6}x_ix_{i+2}\\
&=(x_1x_2+x_2x_3 + x_3x_4 + x_4x_5 + x_5x_6)\\
& - (x_1x_3+x_2x_4 + x_3x_5 + x_4x_6 + x_6x_1), \\
f_2({\bf x})
&=\sum_{i=1}^{6}x_ix_{i+2}-\sum_{i=1}^{6}x_ix_{i+3}\\
&= (x_1x_3+x_2x_4 + x_3x_5 + x_4x_6 + x_6x_1)\\
& - (x_1x_4+x_2x_5 + x_3x_6 + x_5x_1 + x_6x_2).
\end{align*}
By Theorem~\ref{mainth2}, the quadratic form 
\begin{align*}
q({\bf x})=f_0({\bf x})
=\sum_{i=1}^{6}x_i^2-\sum_{i=1}^{6}x_ix_{i+1}
&=(x_1^2+x_2^2+x_3^2+x_4^2+x_5^2+x_6^2)\\ 
&- (x_1x_2+x_2x_3 + x_3x_4 + x_4x_5 + x_5x_6)
\end{align*}
is multiplicative on $V=\{{\bf x}\in K^6\mid f_1({\bf x})=f_2({\bf x})=0\}$. 
Indeed, for $d \in (\Z/7\Z)^\times$, we obtain an explicit bilinear map  $\varphi_d :K^6 \times K^6 \rightarrow K^6$ such that 
\begin{align*}
q({\bf x})q({\bf y})=q(\varphi_d ({\bf x}, {\bf y}))
\end{align*}
for ${\bf x},{\bf y} \in V$.
It is given explicitly by
\[
\varphi_d ({\bf x}, {\bf y})
= (z_{1,d}, z_{2,d}, z_{3,d}, z_{4,d},z_{5,d},z_{6,d})
\]
where $z_{k,d}=-\left(\sum_{i=1}^{6}x_{-d^{-1}i}y_{k-i} - \sum_{i=1}^{6}x_{-d^{-1}i}y_{-i}\right)$.
For example, we have for $d=1 \in (\Z/7\Z)^\times$,
\begin{align*}
z_{1,1}
&=-( x_5y_6 + x_4y_5 + x_3y_4 + x_2y_3 + x_1y_2 )
+(x_6y_6 + x_5y_5 + x_4y_4 + x_3y_3 + x_2y_2 + x_1y_1 ),\\
z_{2,1}
&=-(x_6y_1 + x_4y_6 + x_3y_5 + x_2y_4 + x_1y_3 )
+(x_6y_6 + x_5y_5 + x_4y_4 + x_3y_3 + x_2y_2 + x_1y_1 ),\\
z_{3,1}
&=-(x_6y_2 + x_5y_1 + x_3y_6 + x_2y_5 + x_1y_4 )
+(x_6y_6 + x_5y_5 + x_4y_4 + x_3y_3 + x_2y_2 + x_1y_1 ),\\
z_{4,1}
&=-(x_6y_3 + x_5y_2 + x_4y_1 + x_2y_6 + x_1y_5 )
+(x_6y_6 + x_5y_5 + x_4y_4 + x_3y_3 + x_2y_2 + x_1y_1 ),\\
z_{5,1}
&=-( x_6y_4 + x_5y_3 + x_4y_2 + x_3y_1 + x_1y_6 )
+(x_6y_6 + x_5y_5 + x_4y_4 + x_3y_3 + x_2y_2 + x_1y_1 ),\\
z_{6,1}
&=-(x_6y_5 + x_5y_4 + x_4y_3 + x_3y_2 + x_2y_1)
+(x_6y_6 + x_5y_5 + x_4y_4 + x_3y_3 + x_2y_2 + x_1y_1 ). 
\end{align*}

Moreover, for $d=\pm 1\in (\Z/7\Z)^\times$, 
one can check the following identities, which in particular imply that $q({\bf x})$ is multiplicative on $V$, i.e. 
$q({\bf x})q({\bf y})=q(\varphi_d({\bf x}, {\bf y}))$ for 
${\bf x},{\bf y} \in V=\{{\bf x}\in K^6\mid f_1({\bf x})=f_2({\bf x})=0\}$:
\begin{align*}
q(\varphi_d({\bf x},{\bf y}))&=
q({\bf x})q({\bf y})+f_1({\bf x})f_1({\bf y})+f_2({\bf x})f_2({\bf y}),\\
f_1(\varphi_d({\bf x},{\bf y}))&=
q({\bf x})f_1({\bf y})+f_1({\bf x})q({\bf y})+f_1({\bf x})f_1({\bf y})+f_1({\bf x})f_2({\bf y})+f_2({\bf x})f_1({\bf y})+f_2({\bf x})f_2({\bf y}),\\
f_2(\varphi_d({\bf x},{\bf y}))&=
q({\bf x})f_2({\bf y})+f_2({\bf x})q({\bf y})+f_1({\bf x})f_1({\bf y})+f_1({\bf x})f_2({\bf y})+f_2({\bf x})f_1({\bf y}).
\end{align*}
Note that $q({\bf x})$, $f_1({\bf x})$, $f_2({\bf x})$ 
are invariant under the change $d=1\mapsto -1$. 
For $d=\pm 3\in (\Z/7\Z)^\times$, we obtain that
\begin{align*}
q(\varphi_d({\bf x},{\bf y}))&=
q({\bf x})q({\bf y})+f_1({\bf x})q({\bf y})-f_1({\bf x})f_2({\bf y})+f_2({\bf x})f_1({\bf y})-f_2({\bf x})f_2({\bf y}),\\
f_1(\varphi_d({\bf x},{\bf y}))&=
q({\bf x})f_1({\bf y})+f_2({\bf x})q({\bf y})-f_1({\bf x})f_2({\bf y}),\\
f_2(\varphi_d({\bf x},{\bf y}))&=
q({\bf x})f_2({\bf y})-f_1({\bf x})q({\bf y})-f_2({\bf x})q({\bf y})-f_1({\bf x})f_1({\bf y})+f_1({\bf x})f_2({\bf y})+f_2({\bf x})f_2({\bf y})
\end{align*}
because we get 
$\overline{\sigma}: q({\bf x})\mapsto q({\bf x})+f_1({\bf x})$, 
$f_1({\bf x})\mapsto f_2({\bf x})$, $f_2({\bf x})\mapsto -f_1({\bf x})-f_2({\bf x})$ by changing $d=1\mapsto \pm 3\in (\Z/7\Z)^\times$. 
Applying this twice, we get for $d=\pm 2$ $(=\pm 3^2)\in (\Z/7\Z)^\times$,
\begin{align*}
q(\varphi_d({\bf x},{\bf y}))&=
q({\bf x})q({\bf y})+f_1({\bf x})q({\bf y})+f_2({\bf x})q({\bf y})-f_1({\bf x})f_1({\bf y})+f_1({\bf x})f_2({\bf y})-f_2({\bf x})f_1({\bf y}),\\
f_1(\varphi_d({\bf x},{\bf y}))&=
q({\bf x})f_1({\bf y})-f_1({\bf x})q({\bf y})-f_2({\bf x})q({\bf y})+f_1({\bf x})f_1({\bf y})-f_2({\bf x})f_2({\bf y}),\\
f_2(\varphi_d({\bf x},{\bf y}))&=
q({\bf x})f_2({\bf y})+f_1({\bf x})q({\bf y})-f_2({\bf x})f_1({\bf y}) 
\end{align*}
where $\overline{\sigma}^{\,2}: 
q({\bf x})\mapsto q({\bf x})+f_1({\bf x})+f_2({\bf x})$, 
$f_1({\bf x})\mapsto -f_1({\bf x})-f_2({\bf x})$, 
$f_2({\bf x})\mapsto f_1({\bf x})$ by changing $d=1\mapsto \pm 2$ 
$(=\pm 3^2)\in (\Z/7\Z)^\times$. 
Note that $(\Z/7\Z)^\times/\{\pm 1\}$ is a cyclic group of order $3$ 
and $\overline{\sigma}^{\,3}={\rm id}$. 
\end{example}

\subsection{Explicit lifts of solutions for Diophantine equations}\label{S4-3}

We apply Theorem~\ref{thm:BF} to systems of Diophantine equations associated with Jacobi sums and cyclotomic numbers. 
This recovers explicit lifting formulae for the prime-power systems studied in Hoshi and Kanai~\cite{HK22}, and also explains the corresponding multiplicative quadratic forms on algebraic $K$-varieties considered in Hoshi~\cite{Hos03}.

With the notation $W_{\bf p}$ introduced in Theorem~\ref{thm:structureV} (2), 
classical Jacobi sums in the prime-power case furnish integral points on the fibers $W_{p^r}$. 
The composition law then yields integral points on $W_{p^r q^s}$ from those on $W_{p^r}$ and $W_{q^s}$.
Throughout this subsection, $l$ is an odd prime and plays the role of the integer $e$ in Section~2 and Section~3.\\

(1) $l=3$ (cf. Gauss \cite[Section 358]{Gau1801}, Katre and Rajwade \cite[Proposition~1]{KR85a}, \cite[Section 3.1, Section 10.10]{BEW98}). 
To recover the classical parametrizations, let $p$ and $q$ be primes with $p\equiv q\equiv1\pmod 3$, and let $r,s\geq1$.
Note that we can consider both of the cases $p\neq q$ 
and $p=q$ (see the end of this section for the case where $p=q$).  
Let
$J^\ast_{p^r}(1,1)=\sum_{i=1}^{2}a_{i}\zeta_3^{i}$ $($resp. $J^\ast_{q^s}(1,1)=\sum_{j=1}^{2}b_{j}\zeta_3^{j}$$)$ 
be Jacobi sums for $\F_{p^r}$ $($resp. $\F_{q^s}$$)$. 
We see that $J^\ast_{p^r}(1,1)$ and $J^\ast_{q^s}(1,1)$ are given by 
\begin{align*}
&J^\ast_{p^r}(1,1)
=\frac{1}{2}(x_1+3x_2\sqrt{-3})
=\frac{-x_1+3x_2}{2}\zeta_3 + \frac{-x_1-3x_2}{2}\zeta_3^2,\\
&J^\ast_{q^s}(1,1)
=\frac{1}{2}(y_1+3y_2\sqrt{-3})
=\frac{-y_1+3y_2}{2}\zeta_3 + \frac{-y_1-3y_2}{2}\zeta_3^2
\end{align*}
where $x_1,x_2,y_1,y_2 \in\Z$ are given as integer solutions of the systems of Diophantine equations
\begin{align}\label{eq:DS3p}
\begin{cases}
4p^r=x_1^2+27x_2^2,\\
x_1\equiv 1\ ({\rm mod}\ 3),\ p\notmid x_1,
\end{cases}
\end{align}
\begin{align*} 
\begin{cases}
4q^s=y_1^2+27y_2^2,\\
y_1\equiv 1\ ({\rm mod}\ 3),\ q\notmid y_1. 
\end{cases}
\end{align*}
Each system has two solutions $(x_1,\pm x_2)$ and $(y_1,\pm y_2)$ 
which depend on a choice of generator of 
$\F_{p^r}^\times$ and $\F_{q^s}^\times$ respectively.

By Theorem~\ref{mainth1} and Theorem~\ref{thm:BF},
for $d\in(\Z/3\Z)^{\times}$, 
generalized Jacobi sums
\[
J^\ast_{p^rq^s,d}(1,1)=-\sigma_{-d}(J^\ast_{p^r}(1,1))J^\ast_{q^s}(1,1)=\sum_{k=1}^{2}c_k\zeta_3^k
\]
for $p^rq^s$ are given by
$c_k=-(\sum_{i=1}^{2}a_{-d^{-1}i}b_{k-i} - \sum_{i=1}^{2}a_{-d^{-1}i}b_{-i})$ 
where $J^\ast_{p^r}(1,1)=\sum_{i=1}^2 a_i\zeta_3^i$, 
$J^\ast_{q^s}(1,1)=\sum_{j=1}^2 b_j\zeta_3^j$. 

By Theorem~\ref{thm:BF} again, we have
\begin{align}\label{eq:DS3c}
p^rq^s=c_1^2+c_2^2-c_1c_2.
\end{align}
We define
\begin{align*}
z_{1,d}&:=-(c_1+c_2)=-\frac{1}{2}\sigma^i(x_1y_1+27x_2y_2),\\
z_{2,d}&:=\frac{1}{3}\left(c_1-c_2\right)=-\frac{1}{2}\sigma^i(x_1y_2-y_1 x_2)
\end{align*}
where $\sigma(x_1,x_2,y_1,y_2)=(x_1,-x_2,y_1,y_2)$ and $i=0,1$ if $d=1,2$. 
By the
equation (\ref{eq:DS3c}), we can verify $z_{1,d}, z_{2,d}$ satisfy
\begin{align*}
z_{1,d}^2+27z_{2,d}^2
=4(c_1^2+c_2^2-c_1c_2)
=4p^rq^s.
\end{align*} 
Thus $(z_{1,d},z_{2,d})$ yields solutions of the system of Diophantine equations for $p^rq^s$:
\begin{align*}
\begin{cases}
4p^rq^s=z_{1,d}^2+27z_{2,d}^2,\\
z_{1,d}\equiv 1\ ({\rm mod}\ 3).
\end{cases}
\end{align*}

From the argument above, if we take
${\bf z}^\prime=(z_{1,d}^{\prime}, z_{2,d}^{\prime}) = (2z_{1,d}, 2z_{2,d})$, 
then we get a quadratic form $q({\bf x})=x_1^2+27x_2^2$ which 
is multiplicative on $V=K^2$ with 
$\varphi_d ({\bf x}, {\bf y})={\bf z}^\prime=(z_{1,d}^{\prime}, z_{2,d}^{\prime})$. 
Indeed, we have
\begin{align*}
q(\varphi_d ({\bf x}, {\bf y}))=q({\bf z}^\prime)
=(z_{1,d}^{\prime})^2+27(z_{2,d}^{\prime})^2
=4(z_{1,d}^2+27z_{2,d}^2)
=4(4p^rq^s)
=(4p^r)(4q^s)
=q({\bf x})q({\bf y}).
\end{align*}
Thus we recover the well-known formula:
\begin{align*}
(x_1^2+27x_2^2)(y_1^2+27y_2^2)=(x_1y_1\pm 27 x_2y_2)^2+27(x_1y_2\mp y_1x_2)^2. 
\end{align*}

In the case
$p=q$, by taking $x_i=y_i$, 
a lift of Jacobi sums $J^\ast_{p^r}(1,1)$ to $J^\ast_{p^{nr}}(1,1)$ 
can be obtained by Davenport and Hasse's lifting theorem 
(see Theorem~\ref{thm:corMT1} and Section \ref{S2}). 
Hoshi and Kanai \cite[Section 7, (1) $e=3$]{HK22} 
gave such an explicit lift of solutions of the system of Diophantine equations (\ref{eq:DS3p}) corresponding to Davenport and Hasse's lifting theorem, cf. Theorem~\ref{thm:corMT1} with 
$d=2=-1 \in(\Z/3\Z)^{\times}$. 
Indeed, in such a case, 
\begin{align*}
z_{1,-1}&=-\frac{1}{2}(x_1^2-27x_2^2),\\
z_{2,-1}&=-x_1x_2
\end{align*}
are given as $c^{(2)}$, $d^{(2)}$ in Hoshi and Kanai \cite[page 19, line 8]{HK22}.\\

(2) $l=5$ (cf. Lehmer \cite[Equation (10)]{Leh51}, Berndt and Evans \cite[Section 5]{BE81},  Katre and Rajwade \cite{KR85b}, \cite[Section 3.7]{BEW98}, Hoshi \cite[Section 5]{Hos03}, \cite[Section 3]{Hos06}).
Let $p$ and $q$ be primes with $p\equiv q\equiv1\pmod 5$, and let $r,s\geq1$.
Note that we can consider both of the cases $p\neq q$ 
and $p=q$ (see the end of this section for the case where $p=q$). 
Let 
$J^\ast_{p^r}(1,1)=\sum_{i=1}^{4}a_{i}\zeta_5^{i}$ $($resp. $J^\ast_{q^s}(1,1)=\sum_{j=1}^{4}b_{j}\zeta_5^{j}$$)$ 
be 
Jacobi sums for $\F_{p^r}$ $($resp. $\F_{q^s}$$)$. 
We see that $J^\ast_{p^r}(1,1)$ and $J^\ast_{q^s}(1,1)$ are given by 
\begin{align*}
J^\ast_{p^r}(1,1)
=\frac{1}{4}\left(T\zeta_5+\sigma^3(T)\zeta_5^2+\sigma(T)\zeta_5^3+\sigma^2(T)\zeta_5^4\right),
\end{align*}
\begin{align*}
J^\ast_{q^s}(1,1)
=\frac{1}{4}\left(U\zeta_5+\tau^3(U)\zeta_5^2+\tau(U)\zeta_5^3+\tau^2(U)\zeta_5^4\right)
\end{align*}
where $T=-x_1+5x_2+4x_3+2x_4$, $U=-y_1+5y_2+4y_3+2y_4$, $\sigma(x_1,x_2,x_3,x_4)=(x_1,-x_2,-x_4,x_3)$, $\tau(y_1,y_2,y_3,y_4)=(y_1,-y_2,-y_4,y_3)$ and
$x_1,x_2,x_3,x_4, y_1,y_2,y_3,y_4\in\Z$ are obtained as 
integer solutions of the systems of Diophantine equations 
\begin{align}\label{eq:DS5p}
\begin{cases}
16p^r=x_1^2+125x_2^2+50x_3^2+50x_4^2,\\
x_1x_2=x_3^2-4x_3x_4-x_4^2, \\
x_1 \equiv 1\ ({\rm mod}\ 5),\ p\notmid x_1^2-125x_2^2,
\end{cases} 
\end{align}
\begin{align*} 
\begin{cases}
16q^s=y_1^2+125y_2^2+50y_3^2+50y_4^2,\\
y_1y_2=y_3^2-4y_3y_4-y_4^2, \\
y_1 \equiv 1\ ({\rm mod}\ 5),\ q\notmid y_1^2-125y_2^2.
\end{cases} 
\end{align*}
Each system has four solutions $\sigma^i(x_1,x_2,x_3,x_4)$ and $\tau^i(y_1,y_2,y_3,y_4)$ $(i=0,1,2,3)$ 
which depend on a choice of generator of $\F_{p^r}^\times$ and $\F_{q^s}^\times$ 
respectively.

By Theorem~\ref{mainth1} and Theorem~\ref{thm:BF},
for $d\in(\Z/5\Z)^{\times}$, generalized Jacobi sums
\[
J^\ast_{p^rq^s,d}(1,1)=
-\sigma_{-d}(J^\ast_{p^r}(1,1))J^\ast_{q^s}(1,1)=\sum_{k=1}^{4}c_k\zeta_5^k
\]
for $p^rq^s$ are given by
$c_k=-(\sum_{i=1}^{l-1}a_{-d^{-1}i}b_{k-i} - \sum_{i=1}^{l-1}a_{-d^{-1}i}b_{-i})$, 
$a_1=T/4, a_2=\sigma^3(T)/4, a_3=\sigma(T)/4, a_4=\sigma^2(T)/4, b_1=U/4, b_2=\tau^3(U)/4, b_3=\tau(U)/4, b_4=\tau^2(U)/4$.
The $c_k$'s satisfy
\begin{align}\label{eq:DS5c}
\begin{dcases}
p^rq^s
=c_1^2+c_2^2+c_3^2+c_4^2-c_1c_2-c_2c_3-c_3c_4, \\
c_1c_2+c_2c_3+c_3c_4 - (c_1c_3+c_2c_4+c_4c_1)=0.
\end{dcases}
\end{align}

We define 
\begin{align*} 
\begin{split}
z_{1,d}&:=-(c_1+c_2+c_3+c_4)=-\frac{1}{4}\sigma^i\left(x_1y_1 + 125x_2y_2 + 50x_3y_3 + 50x_4y_4\right),\\
z_{2,d}&:=\frac{1}{5}\left(c_1-c_2-c_3+c_4\right)=-\frac{1}{4}\sigma^i\left(x_1y_2 + x_2y_1 - 2x_3y_3 + 4x_3y_4 + 4x_4y_3  + 2x_4y_4\right),\\
z_{3,d}&:=\frac{1}{5}\left(2c_1-c_2+c_3-2c_4\right)=-\frac{1}{4}\sigma^i\left( x_1y_3 - 5x_2y_3 + 10x_2y_4 - x_3y_1 + 5x_3y_2 - 10x_4y_2 \right),\\
z_{4,d}&:=\frac{1}{5}\left(c_1+2c_2-2c_3-c_4\right)=-\frac{1}{4}\sigma^i\left(x_1y_4 + 10x_2y_3 + 5x_2y_4 - 10x_3y_2 - x_4y_1 - 5x_4y_2  \right)
\end{split}
\end{align*}
where $\sigma(x_1,x_2,x_3,x_4,y_1,y_2,y_3,y_4)=(x_1,-x_2,-x_4,x_3,y_1,y_2,y_3,y_4)$ and $i=0$ 
(resp. $1,3,2$) if $d=1$ (resp. $2,3,4$). 
By the equation (\ref{eq:DS5c}), 
we can verify that $z_{1,d},z_{2,d},z_{3,d}, z_{4,d}$ satisfy
\begin{align*}
z_{1,d}^2+125z_{2,d}^2+50z_{3,d}^2+50z_{4,d}^2
&=16(c_1^2 + c_2^2 + c_3^2 + c_4^2) - 8( (c_1c_2  +  c_2c_3 + c_3c_4) +( c_1c_3 + c_2c_4 + c_4c_1))\\
&=16(c_1^2+c_2^2+c_3^2+c_4^2-(c_1c_2+c_2c_3+c_3c_4))\\
&=16p^rq^s
\end{align*}
and
\begin{align*}
z_{1,d}z_{2,d}-z_{3,d}^2+4z_{3,d}z_{4,d}+z_{4,d}^2=\frac{4}{5}\left(c_1c_2 + c_2c_3 + c_3c_4 - (c_1c_3 + c_2c_4 + c_4c_1)\right)=0.
\end{align*}
Thus $(z_{1,d},z_{2,d},z_{3,d},z_{4,d})$
yields solutions of the system of Diophantine equations 
for $p^rq^s$:
\begin{align*}
\begin{cases}
16p^rq^s=z_{1,d}^2+125z_{2,d}^2+50z_{3,d}^2+50z_{4,d}^2,\\
z_{1,d}z_{2,d}=z_{3,d}^2-4z_{3,d}z_{4,d}-z_{4,d}^2, \\
z_{1,d} \equiv 1\ ({\rm mod}\ 5). 
\end{cases}
\end{align*}

For 
${\bf z}^{\prime} = (z_{1,d}^{\prime},z_{2,d}^{\prime},z_{3,d}^{\prime},z_{4,d}^{\prime})=(4z_{1,d},4z_{2,d},4z_{3,d},4z_{4,d})$, 
the quadratic form $q({\bf x})=x_1^2+125x_2^2+50x_3^2+50x_4^2$ is 
multiplicative on $V$ defined by $x_1x_2=x_3^2-4x_3x_4-x_4^2$
with 
$\varphi_d ({\bf x}, {\bf y})=(z_{1,d}^{\prime}, z_{2,d}^{\prime}, z_{3,d}^{\prime}, z_{4,d}^{\prime})$. 
Indeed, we have
\begin{align*}
q(\varphi_d ({\bf x}, {\bf y}))=q({\bf z}^{\prime})
&={(z_{1,d}^{\prime})}^2+125{(z_{2,d}^{\prime})}^2+50{(z_{3,d}^{\prime})}^2+50{(z_{4,d}^{\prime})}^2\\
&=16(z_{1,d}^2+125z_{2,d}^2+50z_{3,d}^2+50z_{4,d}^2)\\
&=16(16p^rq^s)
=(16p^r)(16q^s)\\
&=(x_1^2+125x_2^2+50x_3^2+50x_4^2)(y_1^2+125y_2^2+50y_3^2+50y_4^2)\\
&=q({\bf x})q({\bf y})
\end{align*}
and
\[
z_{1,d}^{\prime}z_{2,d}^{\prime}={(z_{3,d}^{\prime})}^2-4z_{3,d}^{\prime}z_{4,d}^{\prime}-{(z_{4,d}^{\prime})}^2.
\]
If we take $d=1$ and define $Z_{i}=-z^{\prime}_{i,1}$, 
then we recover the bilinear map $\varphi_1 ({\bf x}, {\bf y})=(Z_1,Z_2,Z_3,Z_4)$  
given in Hoshi \cite[Section 3, Equation (6)]{Hos03}:
\begin{align*}
Z_1&=x_1y_1 + 125x_2y_2 + 50x_3y_3 + 50x_4y_4,\\
Z_2&=x_1y_2 + x_2y_1 - 2x_3y_3 + 4x_3y_4 + 4x_4y_3  + 2x_4y_4,\\
Z_3&=x_1y_3 - 5x_2y_3 + 10x_2y_4 - x_3y_1 + 5x_3y_2 - 10x_4y_2,\\
Z_4&=x_1y_4 + 10x_2y_3 + 5x_2y_4 - 10x_3y_2 - x_4y_1 - 5x_4y_2.
\end{align*}

In the case
$p=q$, by taking $x_i=y_i$, 
a lift of Jacobi sums $J^\ast_{p^r}(1,1)$
to $J^\ast_{p^{nr}}(1,1)$
can be obtained by Davenport and Hasse's lifting theorem 
(see Theorem~\ref{thm:corMT1} and Section \ref{S2}).  
Hoshi and Kanai \cite[Section 7, (2) $e=5$]{HK22} 
gave such an explicit lift of solutions of the system of Diophantine equations (\ref{eq:DS5p}) corresponding to Davenport and Hasse's lifting theorem, cf. Theorem~\ref{thm:corMT1} with 
$d=4=-1 \in(\Z/5\Z)^{\times}$. 
Indeed,
\begin{align*}
z_{1,-1}&=-\frac{1}{4}(x_1^2+125x_2^2-50x_3^2-50x_4^2),\\
z_{2,-1}&=-\frac{1}{2}(x_1x_2+x_3^2-4x_3x_4-x_4^2),\\
z_{3,-1}&=-\frac{1}{2}(x_1x_3+10x_2x_4-5x_3x_2),\\
z_{4,-1}&=-\frac{1}{2}(x_1x_4+10x_2x_3+5x_4x_2)
\end{align*}
are given as
$x^{(2)}$, $w^{(2)}$, $v^{(2)}$, $u^{(2)}$ 
in Hoshi and Kanai \cite[page 21, lines 10--13]{HK22}.\\

(3) $l=7$ (cf. Leonard and Williams \cite{LW75a}, \cite[Section 3.9]{BEW98}).
Let $p$ and $q$ be primes with $p\equiv q\equiv1\pmod 7$, and let $r,s\geq1$.
Note that we can consider both of the cases $p\neq q$ 
and $p=q$ (see the end of this section for the case where $p=q$). 
Let
$J^\ast_{p^r}(1,1)=\sum_{k=1}^{6}a_k\zeta_7^k$ $($resp. $J^\ast_{q^s}(1,1)=\sum_{k=1}^{6}b_k\zeta_7^k$$)$ 
be
Jacobi sums for $\F_{p^r}$ $($resp. $\F_{q^s}$$)$. 
We see that 
$J^\ast_{p^r}(1,1)$ and
$J^\ast_{q^s}(1,1)$
are given by 
\begin{align*}
J^\ast_{p^r}(1,1)
&=\frac{1}{12}\left(T\zeta_7+\sigma^4(T)\zeta_7^2+\sigma^5(T)\zeta_7^3
+\sigma^2(T)\zeta_7^4+\sigma(T)\zeta_7^5+\sigma^3(T)\zeta_7^6\right),\\
J^\ast_{q^s}(1,1) &=\frac{1}{12}\left(U\zeta_7+\tau^4(U)\zeta_7^2+\tau^5(U)\zeta_7^3
+\tau^2(U)\zeta_7^4+\tau(U)\zeta_7^5+\tau^3(U)\zeta_7^6\right)
\end{align*}
where $T=-2x_1+6x_2+7x_5+21x_6$, $U=-2y_1+6y_2+7y_5+21y_6$, $\sigma(x_1,x_2,x_3,x_4,x_5,x_6)=(x_1,-x_3,x_4,x_2,(-x_5-3x_6)/2,(x_5-x_6)/2)$, $\tau(y_1,y_2,y_3,y_4,y_5,y_6)=(y_1,-y_3,y_4,y_2,(-y_5-3y_6)/2,(y_5-y_6)/2)$  and
$x_1,\dots,x_6, y_1,\dots,y_6\in\Z$ are obtained as 
the integer solutions of the systems of Diophantine equations 
\begin{align}\label{eq:DS7p}
\begin{cases}
72p^r=2x_1^2+42x_2^2+42x_3^2+42x_4^2+343x_5^2+1029x_6^2, \\
12x_2^2-12x_4^2+147x_5^2-441x_6^2+56x_1x_6+24x_2x_3
-24x_2x_4+48x_3x_4+98x_5x_6=0, \\
12x_3^2-12x_4^2+49x_5^2-147x_6^2+28x_1x_5+28x_1x_6+48x_2x_3+24x_2x_4+24x_3x_4+490x_5x_6=0,\\
x_1\equiv 1\ ({\rm mod}\ 7),\ (x_5,x_6)\neq (0,0),\\
\end{cases}
\end{align}
\begin{align*} 
\begin{cases}
72q^s=2y_1^2+42y_2^2+42y_3^2+42y_4^2+343y_5^2+1029y_6^2, \\
12y_2^2-12y_4^2+147y_5^2-441y_6^2+56y_1y_6+24y_2y_3
-24y_2y_4+48y_3y_4+98y_5y_6=0, \\
12y_3^2-12y_4^2+49y_5^2-147y_6^2+28y_1y_5+28y_1y_6+48y_2y_3+24y_2y_4+24y_3y_4+490y_5y_6=0,\\
y_1\equiv 1\ ({\rm mod}\ 7),\ (y_5,y_6)\neq (0,0),\\
\end{cases}
\end{align*}
i.e. 
$a_1=T/12, a_2=\sigma^4(T)/12, a_3=\sigma^5(T)/12, a_4=\sigma^2(T)/12,a_5=\sigma(T)/12, a_6=\sigma^3(T)/12,
b_1=U/12, b_2=\tau^4(U)/12, b_3=\tau^5(U)/12, b_4=\tau^2(U)/12,b_5=\tau(U)/12, b_6=\tau^3(U)/12$. 
Each choice of a generator of
$\mathbb{F}_{p^r}^{\times}$
(resp. $\mathbb{F}_{q^s}^{\times}$)
gives one of the six solutions
$\sigma^i(x_1,x_2,x_3,x_4,x_5,x_6)$
(resp. $\tau^i(y_1,y_2,y_3,y_4,y_5,y_6)$),
where $i=0,\ldots,5$.

By Theorem~\ref{mainth1} and Theorem~\ref{thm:BF},
for $d\in(\Z/7\Z)^{\times}$, 
generalized Jacobi sums
\[
J^\ast_{p^rq^s,d}(1,1)
=-\sigma_{-d}(J^\ast_{p^r}(1,1))J^\ast_{q^s}(1,1)=\sum_{k=1}^{6}c_k\zeta_7^k
\]
for $p^rq^s$ are given by 
$c_k=-(\sum_{i=1}^{6}a_{-d^{-1}i}b_{k-i} - \sum_{i=1}^{6}a_{-d^{-1}i}b_{-i})$
and the $c_k$'s satisfy
\begin{align}\label{eq:DS7c}
\begin{dcases}
p^rq^s=
c_1^2+c_2^2+c_3^2+c_4^2+c_5^2+c_6^2- c_1c_2-c_2c_3-c_3c_4-c_4c_5-c_5c_6,\\
c_1c_2+c_2c_3+c_3c_4+c_4c_5+c_5c_6-(c_1c_3+c_2c_4+c_3c_5+c_4c_6+c_6c_1)=0,\\
c_1c_2+c_2c_3+c_3c_4+c_4c_5+c_5c_6-(c_1c_4+c_2c_5+c_3c_6+c_5c_1+c_6c_2)=0.
\end{dcases}
\end{align}
We define
\begin{align}\label{eq:BF7}
\begin{split}
z_{1,d}&:=-(c_1+c_2+c_3+c_4+c_5+c_6)\\
&=-\frac{1}{12}\sigma^i(2 x_1 y_1 + 42 x_2 y_2 + 42 x_3 y_3 + 42 x_4 y_4 + 343 x_5 y_5 + 1029 x_6 y_6),\\
z_{2,d}&:=c_1-c_6\\
&=-\frac{1}{12}\sigma^i(
2 x_1 y_2 
- 2 x_2 y_1 
+ 7 x_2 y_5 
- 21 x_2 y_6
- 21 x_3 y_5 
- 21 x_3 y_6
- 21 x_4 y_5 
+ 21 x_4 y_6\\
&\quad 
- 7 x_5 y_2
+ 21 x_5 y_3 
+ 21 x_5 y_4 
+ 21 x_6 y_2 
+ 21 x_6 y_3
- 21 x_6 y_4
),\\
z_{3,d}&:=c_2-c_5\\
&=-\frac{1}{12}\sigma^i(
2 x_1 y_3 
- 21 x_2 y_5
- 21 x_2 y_6
- 2 x_3 y_1
- 14 x_3 y_5
- 42 x_4 y_6
+ 21 x_5 y_2
+ 14 x_5 y_3\\
&\quad 
+ 21 x_6 y_2
+ 42 x_6 y_4 
),\\
z_{4,d}&:=c_3-c_4\\
&=-\frac{1}{12}\sigma^i(
2 x_1 y_4
- 21 x_2 y_5
+ 21 x_2 y_6
- 42 x_3 y_6
-2 x_4 y_1
+ 7 x_4 y_5
+ 21 x_4 y_6
+ 21 x_5 y_2
- 7 x_5 y_4\\
&\quad 
- 21 x_6 y_2
+ 42 x_6 y_3
- 21 x_6 y_4
 ),\\
z_{5,d}&:=\frac{1}{7}(c_1+c_2-2c_3-2c_4+c_5+c_6)\\
&=-\frac{1}{168}\sigma^i(
28 x_1 y_5
- 12 x_2 y_2
+ 36 x_2 y_3
+ 36 x_2 y_4
+ 36 x_3 y_2
+ 24 x_3 y_3
+ 36 x_4 y_2
- 12 x_4 y_4 \\
&\quad 
+ 28 x_5 y_1
- 49 x_5 y_5
+ 441 x_5 y_6
+ 441 x_6 y_5
+ 147 x_6 y_6
),\\
z_{6,d}&:=\frac{1}{7}(c_1-c_2-c_5+c_6)\\
&=-\frac{1}{168}\sigma^i(
28 x_1 y_6
+ 12 x_2 y_2
+ 12 x_2 y_3
- 12 x_2 y_4
+ 12 x_3 y_2
+ 24 x_3 y_4 \\
&\quad 
- 12 x_4 y_2
+ 24 x_4 y_3
 - 12 x_4 y_4 
+ 147 x_5 y_5
+ 49 x_5 y_6 
+ 28 x_6 y_1
+ 49 x_6 y_5
- 441 x_6 y_6
) 
\end{split}
\end{align}
where $\sigma(x_1,x_2,x_3,x_4,x_5,x_6,y_1,y_2,y_3,y_4,y_5,y_6)=(x_1,-x_3,x_4,x_2,\frac{-x_5-3x_6}{2},\frac{x_5-x_6}{2},y_1,y_2,y_3,y_4,y_5,y_6)$ and $i=0$ 
(resp. $2,1,4,5,3$)
if $d=1$ (resp. $2,3,4,5,6$). 
By the equation (\ref{eq:DS7c}), 
we can verify that 
$z_{1,d},z_{2,d},z_{3,d}, z_{4,d},z_{5,d},z_{6,d}$ satisfy
\begin{align*}
&2z_{1,d}^2+42z_{2,d}^2+42z_{3,d}^2+42z_{4,d}^2+343z_{5,d}^2+1029z_{6,d}^2\\
&=72(c_1^2+ c_2^2 + c_3^2  + c_4^2 + c_5^2 + c_6^2)\\
&\quad  - 24(
(c_1c_2 + c_2c_3 + c_3c_4 + c_4c_5 + c_5c_6)+
(c_1c_3 + c_2c_4 + c_3c_5 + c_4c_6 + c_6c_1)\\
&\quad
+(c_1c_4 + c_2c_5 + c_3c_6 + c_5c_1 + c_6c_2))\\
&=72(c_1^2+ c_2^2 + c_3^2  + c_4^2 + c_5^2 + c_6^2- (c_1c_2 + c_2c_3 + c_3c_4 + c_4c_5 + c_5c_6))\\
&=72p^rq^s,
\end{align*}
and
\begin{align*}
&12z_{2,d}^2-12z_{4,d}^2+147z_{5,d}^2-441z_{6,d}^2+56z_{1,d}z_{6,d}+24z_{2,d}z_{3,d}-24z_{2,d}z_{4,d}+48z_{3,d}z_{4,d}+98z_{5,d}z_{6,d}\\
&=48(c_1c_2 + c_2c_3 + c_3c_4 + c_4c_5  + c_5c_6 - (c_1c_3 + c_2c_4 + c_3c_5 + c_4c_6  + c_6c_1))\\
&=0,\\
&12z_{3,d}^2-12z_{4,d}^2+49z_{5,d}^2-147z_{6,d}^2+28z_{1,d}z_{5,d}+28z_{1,d}z_{6,d}+48z_{2,d}z_{3,d}+24z_{2,d}z_{4,d}+24z_{3,d}z_{4,d}\\
&+490z_{5,d}z_{6,d}\\
&=48(c_1c_2 + c_2c_3 + c_3c_4 + c_4c_5 + c_5c_6 - (c_1c_4 + c_2c_5 + c_3c_6 + c_5c_1 + c_6c_2))\\
&=0.
\end{align*}
Thus $(z_{1,d},z_{2,d},z_{3,d},z_{4,d},z_{5,d},z_{6,d})$ yields solutions of the system of Diophantine equations 
for $p^rq^s$:
\begin{align*}
\begin{cases}
72p^rq^s=2z_{1,d}^2+42z_{2,d}^2+42z_{3,d}^2+42z_{4,d}^2+343z_{5,d}^2+1029z_{6,d}^2, \\
12z_{2,d}^2-12z_{4,d}^2+147z_{5,d}^2-441z_{6,d}^2+56z_{1,d}z_{6,d}+24z_{2,d}z_{3,d}
-24z_{2,d}z_{4,d}+48z_{3,d}z_{4,d}\\
+98z_{5,d}z_{6,d}=0, \\
12z_{3,d}^2-12z_{4,d}^2+49z_{5,d}^2-147z_{6,d}^2+28z_{1,d}z_{5,d}+28z_{1,d}z_{6,d}+48z_{2,d}z_{3,d}+24z_{2,d}z_{4,d}+24z_{3,d}z_{4,d}\\
+490z_{5,d}z_{6,d}=0,\\
z_{1,d}\equiv 1\ ({\rm mod}\ 7).
\end{cases}
\end{align*}

For
\[
{\bf z}^{\prime}=(z_{1,d}^{\prime},z_{2,d}^{\prime},z_{3,d}^{\prime},z_{4,d}^{\prime},z_{5,d}^{\prime},z_{6,d}^{\prime})
=(6z_{1,d},6z_{2,d},6z_{3,d},6z_{4,d},6z_{5,d},6z_{6,d}),
\]
consider the quadratic form
\[
q({\bf x})=2x_1^2+42x_2^2+42x_3^2+42x_4^2+343x_5^2+1029x_6^2
\]
on the variety $V$ defined by
\begin{align*}
\begin{dcases}
12x_2^2-12x_4^2+147x_5^2-441x_6^2+56x_1x_6+24x_2x_3
-24x_2x_4+48x_3x_4+98x_5x_6=0, \\
12x_3^2-12x_4^2+49x_5^2-147x_6^2+28x_1x_5+28x_1x_6+48x_2x_3+24x_2x_4+24x_3x_4+490x_5x_6=0.
\end{dcases}
\end{align*}
Then $q({\bf x})$ is not multiplicative on $V$ with respect to
\[
\varphi_d({\bf x},{\bf y})=(z_{1,d}^{\prime},z_{2,d}^{\prime},z_{3,d}^{\prime},z_{4,d}^{\prime},z_{5,d}^{\prime},z_{6,d}^{\prime}).
\]
Indeed, the point ${\bf z}^{\prime}$ still satisfies the defining equations of $V$:
\begin{align*}
\begin{cases}
12{(z^{\prime}_2)}^2-12{(z^{\prime}_4)}^2+147{(z^{\prime}_5)}^2-441{(z^{\prime}_6)}^2+56z^{\prime}_1z^{\prime}_6
+24z^{\prime}_2z^{\prime}_3-24z^{\prime}_2z^{\prime}_4+48z^{\prime}_3z^{\prime}_4+98z^{\prime}_5z^{\prime}_6=0, \\
12{(z^{\prime}_3)}^2-12{(z^{\prime}_4)}^2+49{(z^{\prime}_5)}^2-147{(z^{\prime}_6)}^2+28z^{\prime}_1z^{\prime}_5+28z^{\prime}_1z^{\prime}_6+48z^{\prime}_2z^{\prime}_3+24z^{\prime}_2z^{\prime}_4+24z^{\prime}_3z^{\prime}_4+490z^{\prime}_5z^{\prime}_6=0,
\end{cases}
\end{align*}
but
\begin{align*}
q(\varphi_d({\bf x},{\bf y}))=q({\bf z}^{\prime})
&=2{(z_{1,d}^{\prime})}^2+42{(z_{2,d}^{\prime})}^2+42{(z_{3,d}^{\prime})}^2+42{(z_{4,d}^{\prime})}^2+343{(z_{5,d}^{\prime})}^2+1029{(z_{6,d}^{\prime})}^2\\
&=36(2z_{1,d}^2+42z_{2,d}^2+42z_{3,d}^2+42z_{4,d}^2+343z_{5,d}^2+1029z_{6,d}^2)\\
&=36(72p^rq^s)\\
&=\frac{1}{2}(72p^r)(72q^s)\\
&=\frac{1}{2}q({\bf x})q({\bf y})
\neq q({\bf x})q({\bf y}).
\end{align*}

This suggests renormalizing the quadratic form. Define
\[
q^\prime({\bf x}):=\frac{1}{2}q({\bf x})
=x_1^2+21x_2^2+21x_3^2+21x_4^2+\frac{343}{2}x_5^2+\frac{1029}{2}x_6^2.
\]
Then
\begin{align*}
q^\prime(\varphi_d({\bf x},{\bf y}))
&=q^\prime({\bf z}^{\prime})
=\frac{1}{2}q({\bf z}^{\prime})\\
&={(z_{1,d}^{\prime})}^2+21{(z_{2,d}^{\prime})}^2+21{(z_{3,d}^{\prime})}^2+21{(z_{4,d}^{\prime})}^2+\frac{343}{2}{(z_{5,d}^{\prime})}^2+\frac{1029}{2}{(z_{6,d}^{\prime})}^2\\
&=\frac{1}{2}\left(\frac{1}{2}q({\bf x})q({\bf y})\right)\\
&=\left(\frac{1}{2}q({\bf x})\right)\left(\frac{1}{2}q({\bf y})\right)\\
&=q^\prime({\bf x})q^\prime({\bf y}).
\end{align*}
Hence $q^\prime({\bf x})$ is multiplicative on $V$ with respect to $\varphi_d({\bf x},{\bf y})=(z_{1,d}^{\prime},z_{2,d}^{\prime},z_{3,d}^{\prime},z_{4,d}^{\prime},z_{5,d}^{\prime},z_{6,d}^{\prime})$.

To compare this with Hoshi~\cite[Theorem 6]{Hos03}, we now specialize to $d=1$ and write
\begin{align*}
{\bf X}&=(X_1,X_2,X_3,X_4,X_5,X_6):=\left(x_1,x_2,x_3,x_4,\frac{7}{2}x_5,\frac{7}{2}x_6\right),\\
{\bf Y}&=(Y_1,Y_2,Y_3,Y_4,Y_5,Y_6):=\left(y_1,y_2,y_3,y_4,\frac{7}{2}y_5,\frac{7}{2}y_6\right),\\
{\bf Z}&=(Z_1,Z_2,Z_3,Z_4,Z_5,Z_6)
:=(z_{1,1}^{\prime},z_{2,1}^{\prime},z_{3,1}^{\prime},z_{4,1}^{\prime},\tfrac{7}{2}z_{5,1}^{\prime},\tfrac{7}{2}z_{6,1}^{\prime})\\
&\phantom{~=(Z_1,Z_2,Z_3,Z_4,Z_5,Z_6)}=(6z_{1,1},6z_{2,1},6z_{3,1},6z_{4,1},21z_{5,1},21z_{6,1}).
\end{align*}
Then
\[
q^\prime({\bf X})=X_1^2+21X_2^2+21X_3^2+21X_4^2+14X_5^2+42X_6^2,
\]
and
\begin{align*}
\begin{split}
Z_1
&=-(X_1Y_1+21X_2Y_2+21X_3Y_3+21X_4Y_4+14X_5Y_5+42X_6Y_6),\\
Z_2
&=-(X_1Y_2-X_2Y_1+X_2Y_5-3X_2Y_6-3X_3Y_5-3X_3Y_6-3X_4Y_5+3X_4Y_6-X_5Y_2\\
&\qquad\quad +3X_5Y_3+3X_5Y_4+3X_6Y_2+3X_6Y_3-3X_6Y_4),\\
Z_3
&=-(X_1Y_3-3X_2Y_5-3X_2Y_6-X_3Y_1-2X_3Y_5-6X_4Y_6+3X_5Y_2+2X_5Y_3\\
&\qquad\quad +3X_6Y_2+6X_6Y_4),\\
Z_4
&=-(X_1Y_4-3X_2Y_5+3X_2Y_6-6X_3Y_6-X_4Y_1+X_4Y_5+3X_4Y_6+3X_5Y_2-X_5Y_4\\
&\qquad\quad -3X_6Y_2+6X_6Y_3-3X_6Y_4),\\
Z_5
&=-\frac{1}{2}(2X_1Y_5-3X_2Y_2+9X_2Y_3+9X_2Y_4+9X_3Y_2+6X_3Y_3+9X_4Y_2-3X_4Y_4\\
&\qquad\qquad +2X_5Y_1-X_5Y_5+9X_5Y_6+9X_6Y_5+3X_6Y_6),\\
Z_6
&=-\frac{1}{2}(2X_1Y_6+3X_2Y_2+3X_2Y_3-3X_2Y_4+3X_3Y_2+6X_3Y_4-3X_4Y_2+6X_4Y_3\\
&\qquad\qquad -3X_4Y_4+3X_5Y_5+X_5Y_6+2X_6Y_1+X_6Y_5-9X_6Y_6).
\end{split}
\end{align*}
Therefore we recover the multiplicative quadratic form $q^\prime({\bf X})$ on the variety $V$
defined by
\begin{align*}
\begin{dcases}
3X_2^2-3X_4^2+3X_5^2-9X_6^2+4X_1X_6+6X_2X_3-6X_2X_4+12X_3X_4+2X_5X_6=0, \\
3X_3^2-3X_4^2+X_5^2-3X_6^2+2X_1X_5+2X_1X_6+12X_2X_3+6X_2X_4+6X_3X_4+10X_5X_6=0
\end{dcases}
\end{align*}
with $\varphi_1({\bf X},{\bf Y})=(Z_1,Z_2,Z_3,Z_4,Z_5,Z_6)$,
which coincides with the one given by Hoshi \cite[Section 2, Theorem~6]{Hos03}. (There is a typo in \cite[p.~74, line 18]{Hos03}: the term $-3X_6Y_6$ in the sixth component should be moved to the fifth component.)

In the case
$p=q$, by taking $x_i=y_i$, 
a lift of Jacobi sums $J^\ast_{p^r}(1,1)$ to $J^\ast_{p^{nr}}(1,1)$ 
can be obtained by Davenport and Hasse's lifting theorem 
(see Theorem~\ref{thm:corMT1} and Section \ref{S2}). 
Hoshi and Kanai \cite[Section 7, (3) $e=7$]{HK22} 
gave such an explicit lift of solutions of the system of Diophantine equations (\ref{eq:DS7p}) corresponding to Davenport and Hasse's lifting theorem, cf. Theorem~\ref{thm:corMT1} with 
$d=6=-1 \in(\Z/7\Z)^{\times}$. 
Indeed,   
\begin{align*}
z_{1,-1} &= -\frac{1}{12}(2x_1^2 - 42x_2^2 - 42x_3^2 - 42x_4^2 + 343x_5^2 +1029x_6^2),\\
z_{2,-1} &= -\frac{1}{12}(4x_1x_2 - 14x_2x_5 + 42x_3x_5 + 42x_4x_5 + 42x_2x_6 + 42x_3x_6 - 42x_4x_6),\\
z_{3,-1} &= -\frac{1}{12}(4x_1x_3 + 42x_2x_5 + 28x_3x_5 + 42x_2x_6 + 84x_4x_6),\\
z_{4,-1} &= -\frac{1}{12}(4x_1x_4 + 42x_2x_5 - 14x_4x_5 - 42x_2x_6 + 84x_3x_6 - 42x_4x_6),\\
z_{5,-1} &= -\frac{1}{168}(12x_2^2 - 72x_2x_3 - 24x_3^2 - 72x_2x_4 + 12x_4^2 + 56x_1x_5 - 49x_5^2 + 882x_5x_6 + 147x_6^2),\\
z_{6,-1} &= -\frac{1}{168}(-12x_2^2 - 24x_2x_3 + 24x_2x_4 - 48x_3x_4 + 12x_4^2 + 147x_5^2 + 56x_1x_6 + 98x_5x_6 - 441x_6^2)
\end{align*}
are given as $x_1^{(2)}$, $x_2^{(2)}$, $x_3^{(2)}$, 
$x_4^{(2)}$, $x_5^{(2)}$, $x_6^{(2)}$ 
in Hoshi and Kanai \cite[page 24, lines 5--10]{HK22}. 

\begin{remark}
A similar result can be obtained for $l=11$ 
by using Leonard and Williams \cite{LW75b}.
\end{remark}

\section{Proof of Theorem~\ref{mainth3}: Multiplicative $n$-ic forms on algebraic varieties}\label{S5}

In this section, we construct multiplicative $n$-ic forms on algebraic $K$-varieties for $n \ge 2$. 
The defining equations of $V$ are recast in terms of correlation sums arising from the cyclotomic setup, and this formulation is compatible with the multiplicative formula of Theorem~\ref{thm:BF2}, which yields the induced composition law and shows that $V$ is stable under this law. 
The regularity of the form $h$ is then obtained from its derived form and its relation to the norm form. 
We use these results to prove Theorem~\ref{mainth3}.

Let $h(x)=h(x_1,\dots,x_r)$ be an $n$-ic form, i.e. a homogeneous polynomial of total degree $n$ in $r$ variables. 
Assume that $h$ is regular, i.e. that the associated symmetric $n$-linear map $\theta:K^r\times\cdots\times K^r\to K$ has the property that $\theta({\bf v}_1,{\bf v}_2,\dots,{\bf v}_n)=0$ for any ${\bf v}_2,\dots, {\bf v}_n$ implies ${\bf v}_1=0$
(see Shapiro \cite[Chapter 16]{Sha00}, Schafer \cite{Sch70}, cf. Section \ref{S4} for $n=2$).

\begin{defn}[see Definition \ref{def: quad-on-V} for $n=2$]\label{def:mult-n-ic}
Let $n \geq 2$ be an integer and $K$ be a field with $\operatorname{char}(K)\notmid n!$.  
Let $V \subset K^r$ be an algebraic $K$-variety. 
We say that a regular $n$-ic form $h({\bf x})$ is {\it multiplicative on $V$} 
if there exists a bilinear map 
$\varphi :K^r \times K^r \rightarrow K^r$ such that
\[
\varphi(V \times V) \subset V \quad \mbox{and}\quad h({\bf x})h({\bf y})=h(\varphi({\bf x,y}))\mbox{ for any }{\bf x,y} \in V.
\]
\end{defn}

In particular, a regular $n$-ic form on $K^r$ is multiplicative on $\A_K^r$ if and only if it permits bilinear composition on $K^r$.
The following classification of regular forms permitting composition combines Schafer's structural result \cite[Theorem~3]{Sch70} with the decomposition theorem \cite[Theorem~4]{Sch63}:

\begin{thm}[{Schafer \cite[Theorem~4]{Sch63}, \cite[Theorem~3]{Sch70}}]
\label{thm:schafer}
Let $n\geq 3$, and let $A$ be a finite-dimensional
unital algebra over a field $K$ with
$\operatorname{char}(K)\notmid n!$.
Assume that $A$ permits composition, 
i.e. that there exists a regular $n$-ic form $F:A\to K$ satisfying
\[
F(xy)=F(x)F(y),\qquad F(1)=1.
\]
Then $A$ is a separable alternative algebra.
Moreover, writing
\[
A=A_1\oplus\cdots\oplus A_s
\]
as a direct sum of simple ideals, one has
\[
F(x_1+\cdots+x_s)
=\prod_{i=1}^{s}N_i(x_i)^{e_i},
\qquad
n=\sum_{i=1}^{s}e_i d_i,
\]
where $N_i$ is the generic norm of $A_i$ over $K$,
$d_i=\deg N_i$, and the $e_i$ are positive integers.

Conversely, every form of this type on a separable
alternative algebra is regular and permits composition.
\end{thm}

The cubic case of Schafer's classification gives
the following sharper dimension restriction.

\begin{cor}[{Schafer \cite[p.~140]{Sch70}}]
\label{cor:schafer-cubic}
Let $F:K^r\to K$ be a regular cubic form, 
where $\operatorname{char}(K)\notmid 3!$.
If $F$ is multiplicative on $\A_{K}^r$, then
\[
r\in\{1,2,3,5,9\}.
\]
\end{cor}

For general $n\geq 3$, Schafer's classification also yields the following uniform dimension bound.

\begin{lem}\label{lem:schafer-dimension}
Let $n\geq 3$, and let $F:K^r\to K$ be a regular $n$-ic form with $\operatorname{char}(K)\nmid n!$.
If $F$ is multiplicative on $\A_{K}^r$, then $r\leq n^2$.
\end{lem}

\begin{proof}
We first reduce to the unital case in order to apply
Schafer's classification.
Since $\operatorname{char}(K)\notmid n!$, 
the multiplicative identity holds as a polynomial identity. 
Thus, extending scalars if necessary, 
we may assume that $K$ is algebraically closed.
Let $B$ be a bilinear composition for $F$, 
and choose $a\in K^r$ with $F(a)=1$.
Define
\[
L_a(y)=B(a,y),\qquad R_a(x)=B(x,a).
\]
Since
\[
F(L_a(y))=F(R_a(y))=F(y),
\]
the regularity of $F$ implies that $L_a$ and $R_a$
are invertible.

Define a new product by
\[
x*y=B(R_a^{-1}x,L_a^{-1}y).
\]
Then $e=B(a,a)$ is the identity element for $*$,
and
\[
F(x*y)=F(x)F(y).
\]
Thus, by Theorem~\ref{thm:schafer}, the unital algebra
$(K^r,*)$ is a direct sum of matrix algebras and
octonion algebras.

Each matrix algebra $M_{d_i}(K)$ has dimension $d_i^2$
and generic norm of degree $d_i$, while an octonion
algebra has dimension $8$ and generic norm of degree $2$.
Consequently, for some positive integers $d_i$
and an integer $t\geq 0$, we have
\[
r=\sum_i d_i^2+8t,
\qquad
\sum_i d_i+2t\leq n.
\]
Writing $u:=\sum_i d_i$, we obtain
\begin{align*}
r
&\leq u^2+8t\\
&\leq (n-2t)^2+8t\\
&=n^2-4t(n-t-2)\\
&\leq n^2,
\end{align*}
since $0\leq t\leq\lfloor n/2\rfloor\leq n-2$.
\end{proof}

The restriction $n\geq 3$ is essential.
Indeed, the norm of an octonion algebra is a regular
quadratic form in eight variables permitting bilinear
composition, whereas $8>2^2$.

We now turn to our construction of multiplicative
$n$-ic forms on algebraic varieties.
We begin by expressing the relative-norm condition
in terms of the coefficients of elements in a
cyclotomic extension.

\begin{lem} \label{lem:BEW2}
Let $n\geq2$ be an integer and $l$ be an odd prime with $l \equiv 1\ ({\rm mod}\ n)$. 
Write $l=mn+1$.
Let $K$ be a field with $\operatorname{char}(K)\notmid l(n!)$ and 
$\Gal(K(\zeta_l)/K)\simeq (\Z/l\Z)^{\times}=\langle \gamma \rangle$. 
Define $\beta := \gamma^m$ and $M=K(\zeta_l)^{\langle \beta \rangle}$, 
i.e., $\Gal(K(\zeta_l)/M)\simeq \Z/n\Z = \langle \beta \rangle \leq  (\Z/l\Z)^{\times}$.
Let ${\bf p}\in K^{\times}$ and
\[
\alpha_{\bf p}=\sum_{k=0}^{l-1}a_k\zeta_l^k \quad (a_k \in K)\\
\quad \text{with} \quad
N_{K(\zeta_l)/M}(\alpha_{\bf p})= {\bf p}.
\]
Then we have
\begin{enumerate}
\item [$(1)$] ${\bf p} =S_0 - S_1$,
\item [$(2)$] $S_1=S_2=\cdots =S_{m}$
\end{enumerate}
where 
\begin{align*}
S_j = 
\left\{
\begin{array}{ll}
\displaystyle
\sum_{i_0 + \beta i_1 + \cdots +\beta^{n-1}i_{n-1}\equiv 0 \ (l)}a_{i_0}a_{i_1}\cdots a_{i_{n-1}} & (j = 0)\\
\displaystyle
\sum_{i_0 + \beta i_1 + \cdots +\beta^{n-1}i_{n-1}\equiv \gamma^{j} \ (l)}a_{i_0}a_{i_1}\cdots a_{i_{n-1}} & (j = 1,\ldots,m)
\end{array}
\right.
\end{align*}
and each subscript of the $a_k$'s should be taken modulo $l$. 
Moreover, we may assume that $a_0=0$ without loss of generality.
\end{lem}
\begin{proof}
By $N_{K(\zeta_l)/M}(\alpha_{\bf p})= {\bf p}$ and $[K(\zeta_l):M]=n$, we have 
\begin{align*}
{\bf p} = \Big( a_1\zeta_l +  \cdots +a_{l-1}\zeta_l^{l-1}\Big) \left( a_1\zeta_l^{\beta} +  \cdots +a_{l-1}\zeta_l^{\beta(l-1)}\right) \cdots \left( a_1\zeta_l^{\beta^{n-1}} +  \cdots +a_{l-1}\zeta_l^{\beta^{n-1}(l-1)}\right) ,
\end{align*}
that is,
\begin{align*}
{\bf p} 
&= S_0 + S_1\eta(1) + \cdots + S_{m-1}\eta(m-1) + S_m \eta(0)
\end{align*}
where, for $0\leq i \leq m-1$, $\eta(i)=\sum_{j=0}^{n-1}\zeta_l^{\gamma ^{mj+i}}$ are Gaussian $n$-periods of degree $m$.
Note that  $S_j =S_{j+mi} (1\leq i \leq n-1)$ 
because multiplication of the congruence $i_0+\beta i_1+\cdots+\beta^{n-1}i_{n-1}\equiv\gamma^j\pmod l$ by $\beta^i$ only changes the order of the variables.
By $\zeta_l+\cdots+\zeta_l^{l-1}=-1$, we have
\[
(S_1 - S_0 + {\bf p})\eta(1)  + 
 \cdots + (S_{m-1} - S_0 + {\bf p})\eta(m-1) + (S_m - S_0 + {\bf p})\eta(0) = 0.
\]
Because $\{\eta(0), \ldots, \eta(m-1)\}$ forms a normal basis of $M/K$, 
we have
\[
S_1 - S_0 + {\bf p} = S_j - S_0 + {\bf p}=0\ \ (1\leq j\leq m).
\]
This implies that 
(1) ${\bf p}= S_0 - S_1$ and (2) $S_1=S_j$ 
$(2\leq j\leq m)$. 
\end{proof}

\begin{thm}\label{thm:BF2}
Let $n\geq2$ be an integer. 
Let $l$ be an odd prime with $l \equiv 1\ ({\rm mod}\ n)$ and $d \in (\Z/l\Z)^{\times}$. 
Write $l=mn+1$.
Let $K$ be a field with $\operatorname{char}(K)\notmid l(n!)$ and 
$\Gal(K(\zeta_l)/K)\simeq (\Z/l\Z)^{\times}=\langle \gamma \rangle$. 
Define $\beta := \gamma^m$ and $M=K(\zeta_l)^{\langle \beta \rangle}$, 
i.e., $\Gal(K(\zeta_l)/M)\simeq \Z/n\Z = \langle \beta \rangle \leq  (\Z/l\Z)^{\times}$.
For ${\bf p}_1, {\bf p}_2\in K^{\times}$ we define
\begin{align*}
\alpha_{{\bf p}_1}=\sum_{i=0}^{l-1}a_i\zeta_l^i\quad \text{with} \quad
N_{K(\zeta_l)/M}(\alpha_{{\bf p}_1})= {\bf p}_1, \quad \alpha_{{\bf p}_2}=\sum_{j=0}^{l-1}b_j\zeta_l^j\quad \text{with} \quad
N_{K(\zeta_l)/M}(\alpha_{{\bf p}_2})= {\bf p}_2.
\end{align*}
Define also $\alpha_{{\bf p}_1{\bf p}_2,d}:=(-1)^{n-1}\sigma_{-d}(\alpha_{{\bf p}_1})\alpha_{{\bf p}_2}$ where $\sigma_{-d}\in \Gal(K(\zeta_l)/K)$ satisfies $\sigma_{-d}(\zeta_l)=\zeta_l^{-d}$.
Then $\alpha_{{\bf p}_1{\bf p}_2,d} = \sum_{k=0}^{l-1} c_k\zeta_l^k$ is given by
$$c_k = (-1)^{n-1}\sum_{i=0}^{l-1}a_{-d^{-1}i}b_{k-i},$$
and these coefficients satisfy
\begin{enumerate}
\item ${\bf p}_1{\bf p}_2 = S_0 - S_1$,
\item $S_1 = S_2 = \cdots = S_m$.
\end{enumerate}
Here, the sums $S_j$ are defined as
\begin{align*}
S_j = 
\left\{
\begin{array}{ll}
\displaystyle
\sum_{i_0 + \beta i_1 + \cdots +\beta^{n-1}i_{n-1}\equiv 0 \ (l)}c_{i_0}c_{i_1}\cdots c_{i_{n-1}} & (j = 0)\\
\displaystyle
\sum_{i_0 + \beta i_1 + \cdots +\beta^{n-1}i_{n-1}\equiv \gamma^{j} \ (l)}c_{i_0}c_{i_1}\cdots c_{i_{n-1}} & (j = 1,\ldots,m),
\end{array}
\right.
\end{align*}
with all indices taken modulo $l$.
Moreover, we may assume $a_0=b_0=c_0=0$ 
without loss of generality and in this case 
$\alpha_{{\bf p}_1{\bf p}_2,d}=\sum_{k=1}^{l-1}c_k\zeta_l^k$ 
is  given by 
$$c_k=(-1)^{n-1}\left( \sum_{i=1}^{l-1}a_{-d^{-1}i}b_{k-i} - \sum_{i=1}^{l-1}a_{-d^{-1}i}b_{-i}\right).$$
\end{thm}
\begin{proof}
By the definition of $\alpha_{{\bf p}_1}$ and $\alpha_{{\bf p}_2}$, we have
\begin{align*}
N_{K(\zeta_l)/M}(\alpha_{{\bf p}_1{\bf p}_2,d})
=N_{K(\zeta_l)/M}((-1)^{n-1}\sigma_{-d}(\alpha_{{\bf p}_1})
\alpha_{{\bf p}_2})
=(-1)^{n(n-1)}{\bf p}_1{\bf p}_2={\bf p}_1{\bf p}_2.
\end{align*}
Using a calculation similar to that in the proof of Theorem~\ref{thm:BF}, we obtain
\begin{align*}
\alpha_{{\bf p}_1{\bf p}_2,d}=(-1)^{n-1}\sigma_{-d}(\alpha_{{\bf p}_1})\alpha_{{\bf p}_2}=(-1)^{n-1}\Bigl(\sum_{i=0}^{l-1}a_{-d^{-1}i}\zeta_l^i\Bigr) \Bigl(\sum_{j=0}^{l-1}b_j\zeta_l^j\Bigr)=\sum_{k=0}^{l-1}c_k\zeta_l^k
\end{align*}
where $c_k=(-1)^{n-1}\sum_{i=0}^{l-1}a_{-d^{-1}i}b_{k-i}$. 
Thus the assertion follows from Lemma \ref{lem:BEW2}.

The case $a_0=b_0=0$ also follows from the same argument as in the proof of Theorem~\ref{thm:BF}, yielding $c_k=(-1)^{n-1}(\sum_{i=1}^{l-1}a_{-d^{-1}i}b_{k-i} - \sum_{i=1}^{l-1}a_{-d^{-1}i}b_{-i})$ 
when $\alpha_{{\bf p}_1{\bf p}_2,d}=\sum_{k=1}^{l-1}c_k\zeta_l^k$.
\end{proof}

To consider the regularity of the $n$-ic form $h({\bf x})$, we define the associated symmetric $n$-linear map $\theta_h$ explicitly.
Following Dr\'apal and Vojt\v{e}chovsk\'y \cite[\S2.1]{DV09}, for a map $P\colon V\to F$
with $P(0)=0$ we define {\it the $n$-th defect} (or {\it derived form})
\[
\Delta^{n}P(u_1,\dots,u_n)
=\sum_{\emptyset\neq I\subset\{1,\dots,n\}}(-1)^{n-|I|}
\,P\Big(\sum_{i\in I}u_i\Big),
\qquad
\theta_P:=\frac{1}{n!}\Delta^{n}P.
\]
If $\operatorname{char}(F)\notmid n!$ and $P$ is a homogeneous polynomial map of degree $n$, then
$\theta_P$ is a symmetric $n$-linear form and $P(x)=\theta_P(x,\dots,x)$.

\begin{prop}\label{prop:norm_polarization}
Let $E/F$ be a separable field extension of degree $n$, and assume that $\operatorname{char}(F) \notmid n!$. 
Let $N_{E/F}\colon E\to F$ be the norm map.
Then
\[
\theta_N(u_1,\dots,u_n):=\frac{1}{n!}\,\Delta^n N_{E/F}(u_1,\dots,u_n)
\]
is a symmetric $n$-linear form on $E$.
\end{prop}

\begin{proof}
Regard $E$ as an $F$-vector space of dimension $n$. 
For $x\in E$, let $m_x\colon E\to E$ be multiplication by $x$. 
Then $x\mapsto m_x$ is $F$-linear and
\[
N_{E/F}(x)=\det(m_x),
\]
so $N_{E/F}$ is a homogeneous polynomial map of degree $n$ on $E$.

Now apply the basic properties of derived forms recalled above.
Since $N_{E/F}$ is a homogeneous polynomial map of degree $n$ and
$\operatorname{char}(F)\notmid n!$, it follows from \cite[\S2.1]{DV09} that the $n$-th defect $\Delta^nN_{E/F}$ is symmetric and $F$-multilinear, hence $(1/n!)\Delta^n N_{E/F}$ is $F$-linear
in each argument. 
Therefore $\theta_N$ is a symmetric $n$-linear form on $E$.
\end{proof}

\begin{lem}\label{lem:h_regularity}
With the notation of Theorem~\ref{mainth3}, the $n$-ic form $h({\bf x})$ on $K^{l-1}$ is regular.
\end{lem}

\begin{proof}
Define $L=K(\zeta_l)$ and $M=L^{\langle\beta\rangle}$, so $[L:M]=n$.
Via the $K$-linear isomorphism
\[
\iota:K^{l-1}\xrightarrow{\ \sim\ }L,\qquad
{\bf x}=(x_1,\dots,x_{l-1})\longmapsto
\alpha_{\bf x}:=\sum_{i=1}^{l-1}x_i\zeta_l^i,
\]
we identify $K^{l-1}$ with $L$. Let $\sigma\in\Gal(L/M)$ be the generator determined by
$\sigma(\zeta_l)=\zeta_l^\beta$. Then
\[
N_{L/M}(z)=\prod_{t=0}^{n-1}\sigma^t(z)\qquad(z\in L).
\]

\medskip\noindent
\textbf{Step 1: $h$ as a $K$-linear projection of the norm.}
As in the proof of Lemma~\ref{lem:BEW2}, expanding $N_{L/M}(\alpha_{\bf x})$ and grouping terms by
$\langle\beta\rangle$-orbits yields
\begin{equation}\label{eq:norm_period_expansion}
N_{L/M}(\alpha_x)
=
S_0(x)+S_m(x)\eta(0)
+\sum_{j=1}^{m-1}S_j(x)\eta(j)
\end{equation}
where $S_j({\bf x})$ are the polynomials defined in Lemma~\ref{lem:BEW2} (with $x_0=0$ and indices modulo $l$).
Since $\eta(0)+\cdots+\eta(m-1)=\zeta_l+\cdots+\zeta_l^{l-1}=-1$, we can rewrite \eqref{eq:norm_period_expansion} as
\[
N_{L/M}(\alpha_x)
=
\bigl(S_m(x)-S_0(x)\bigr)\eta(0)
+
\sum_{j=1}^{m-1}
\bigl(S_j(x)-S_0(x)\bigr)\eta(j).
\]

If \(m=1\), then \(M=K\) and \(\eta(0)=-1\), so
\[
N_{L/K}(\alpha_x)=S_0(x)-S_1(x)=h(x);
\]
in this case, set \(\lambda=\mathrm{id}_K\).
If $m\geq 2$, define a $K$-linear projection
$\lambda:M\to K$ by
\[
\lambda(\eta(1))=-1,
\qquad
\lambda(\eta(j))=0
\quad (j\neq 1).
\]
Then, in either case, applying $\lambda$ gives
\[
\lambda\bigl(N_{L/M}(\alpha_x)\bigr)
=
-\bigl(S_1(x)-S_0(x)\bigr)
=
h(x),
\]
hence
\begin{equation}\label{eq:h_as_proj_norm}
h=\lambda\circ N_{L/M}\circ\iota.
\end{equation}

\medskip\noindent
\textbf{Step 2: Polarization and a trace specialization.}
Let $\theta_N$ be the symmetric $n$-linear form associated with $N_{L/M}$ over the base field $M$
(Proposition~\ref{prop:norm_polarization} applied to $E=L$, $F=M$), and let $\theta_h$ be the symmetric
$n$-linear form associated with $h$.
By \eqref{eq:h_as_proj_norm} and linearity of polarization,
\begin{equation}\label{eq:theta_h_comp}
\theta_h=\lambda\circ\theta_N\circ(\iota\times\cdots\times\iota).
\end{equation}
A standard computation from $N_{L/M}(z)=\prod_{t=0}^{n-1}\sigma^t(z)$ gives, for $\alpha,v\in L$,
\begin{equation}\label{eq:thetaN_trace}
\theta_N(v,\alpha,\dots,\alpha)
=\frac{1}{n}\sum_{k=0}^{n-1}\sigma^k(v)\prod_{j\ne k}\sigma^j(\alpha)
=\frac{1}{n}\Tr_{L/M}\!\bigl(vP(\alpha)\bigr),
\qquad
P(\alpha):=\prod_{j=1}^{n-1}\sigma^j(\alpha).
\end{equation}

\medskip\noindent
\textbf{Step 3: Regularity.}
It suffices to show that if $\alpha \in L$ satisfies $\theta_h(v,\alpha,\dots,\alpha)=0$ for all $v\in L$,
then $\alpha=0$.
Indeed, assuming this, let $u\in L$ be such that
$\theta_h(u,w_1,\dots,w_{n-1})=0$ for all $w_1,\dots,w_{n-1}\in L$.
Then, in particular, by taking $(w_1,\dots,w_{n-1})=(v,u,\dots,u)$, we obtain
$\theta_h(u,v,u,\dots,u)=0$ for all $v\in L$.
By symmetry of $\theta_h$, this implies
$\theta_h(v,u,\dots,u)=0$ for all $v\in L$.
Hence $u=0$.

Suppose $\alpha\in L$ satisfies $\theta_h(v,\alpha,\dots,\alpha)=0$ for all $v\in L$.
Using \eqref{eq:theta_h_comp} and \eqref{eq:thetaN_trace} we obtain
\[
\lambda\!\left(\Tr_{L/M}(vP(\alpha))\right)=0\qquad(\forall v\in L).
\]
Define $\Psi:=\lambda\circ\Tr_{L/M}:L\to K$. 
Then $\Psi\neq 0$: indeed, $\eta(0)+\cdots+\eta(m-1)=-1$ implies $\lambda(1)=1$, hence
\[
\Psi(1)=\lambda\!\left(\Tr_{L/M}(1)\right)=\lambda(n\cdot 1)=n\neq 0
\]
since $\operatorname{char}(K)\notmid n$.
Thus $\Psi(vP(\alpha))=0$ for all $v\in L$.

Because $L/K$ is finite separable, the trace pairing $(x,y)\mapsto\Tr_{L/K}(xy)$ is non-degenerate.
Therefore there exists $c\in L^\times$ such that $\Psi(z)=\Tr_{L/K}(cz)$ for all $z\in L$. 
Consequently,
\[
0=\Psi(vP(\alpha))=\Tr_{L/K}(c\,v\,P(\alpha))\qquad(\forall v\in L),
\]
and non-degeneracy forces $cP(\alpha)=0$, hence $P(\alpha)=0$. Since $L$ is a field and
$P(\alpha)=\prod_{j=1}^{n-1}\sigma^j(\alpha)$, we conclude $\alpha=0$.
Thus $\theta_h(v,\alpha,\dots,\alpha)\equiv 0$ implies $\alpha=0$, i.e.\ $h$ is regular.
\end{proof}

We now prove Theorem~\ref{mainth3}.

\begin{proof}[Proof of Theorem~\ref{mainth3}]
The regularity of the $n$-ic form $h({\bf x})$ follows from Lemma \ref{lem:h_regularity}.

By the definition of $V$ in Theorem~\ref{mainth3} and Lemma~\ref{lem:BEW2}, for ${\bf x}\in K^{l-1}$ we have
\[
{\bf x}\in V \quad\Longleftrightarrow\quad h_j({\bf x})=0\ (j=1,\dots,m-1)
\quad\Longleftrightarrow\quad S_1({\bf x})=S_2({\bf x})=\cdots=S_m({\bf x})
\]
where $S_j(x)$ denotes the $j$-th correlation sum of the coefficients of $\sum_{i=1}^{l-1}x_i\zeta_l^{i}$ (cf. proof of Lemma~\ref{lem:BEW2}).

Now let ${\bf x,y}\in V$ and define ${\bf z}=\varphi_d({\bf x,y})$.
If ${\bf x}=0$ or ${\bf y}=0$, the assertion is immediate, so we may assume that ${\bf x,y}\ne0$.
By Theorem~\ref{thm:BF2}, the coefficients of $\alpha_{\bf z}=\sum_{i=1}^{l-1}c_i\zeta_l^i$ are given by $c_k=(-1)^{n-1}\left( \sum_{i=1}^{l-1}a_{-d^{-1}i}b_{k-i} - \sum_{i=1}^{l-1}a_{-d^{-1}i}b_{-i}\right)$, and in
particular the corresponding correlation sums satisfy
\[
S_1({\bf z})=S_2({\bf z})=\cdots=S_m({\bf z}).
\]
Hence ${\bf z}\in V$, i.e.\ $\varphi_d(V\times V)\subset V$.

Finally, Theorem~\ref{thm:BF2} (1) gives $h({\bf z})=h({\bf x})h({\bf y})$, i.e.
\[
h(\varphi_d({\bf x,y}))=h({\bf x})h({\bf y}),
\]
and Theorem~\ref{mainth3} follows.
\end{proof}

We can now compare the forms constructed in
Theorem~\ref{mainth3} with the classical
dimension restrictions.

\begin{proof}[Proof of Corollary~\ref{cor:no-composition}]
If $l-1>n^2$, Lemma~\ref{lem:schafer-dimension} shows that
$h$ is not multiplicative on $\A_K^{l-1}$.
If $n=3$, the same conclusion follows from Corollary~\ref{cor:schafer-cubic}, since $6\mid(l-1)$.

On the other hand, by Theorem~\ref{mainth3},
$h$ is multiplicative on $V$.
Hence $V\subsetneq\A_K^{l-1}$.

Now assume that $\operatorname{char}(K)=0$.
By Dirichlet's theorem, there exist arbitrarily large primes $l\equiv 1\pmod n$ for each fixed $n\geq 3$.
Applying Theorem~\ref{mainth3} over $\Q$ and extending scalars to $K$, 
we obtain regular $n$-ic forms that are multiplicative on $V$, 
since regularity and the relative norm description of $V$ and $h$ are preserved under base change (cf.~Remark~\ref{rem:morph}).
For $l-1>n^2$, Lemma~\ref{lem:schafer-dimension} excludes multiplicativity on $\A_K^{l-1}$.
\end{proof}

\begin{example}
[$l=5, n=4, m=1, \gamma =\beta =2$]
\begin{align*}
h({\bf x})=h(x_1,\ldots,x_4)&=S_0 - S_1
\end{align*}
where
\begin{align*}
S_0({\bf x})&= \sum_{s + 2t + 4u + 3v\equiv 0 \ (5)} x_{s}x_{t}x_{u}x_{v}\\
&= x_1^4 + x_2^4 + x_3^4 + x_4^4\\
&+ 4(
x_1x_2x_3x_4 
+x_1^2x_2x_3 
+ x_1x_2x_3^2 
+ x_1^2x_2x_4 
+ x_1x_2^2x_4 
+ x_1x_3^2x_4
+ x_1x_3x_4^2 
+ x_2^2x_3x_4 
+ x_2x_3x_4^2 
) \\
&+ 2(x_1^2x_2^2  
+ x_1^2x_3^2 
+ x_1^2x_4^2 
+ x_2^2x_3^2  
+ x_2^2x_4^2 
+ x_3^2x_4^2 ),\\
S_1({\bf x})&= \sum_{s + 2t + 4u + 3v\equiv 2 \ (5)} x_{s}x_{t}x_{u}x_{v}\\
&= 5x_1x_2x_3x_4 + x_1^2x_2^2 + x_1^2x_3^2 + x_1^2x_4^2 + x_2^2x_3^2+ x_2^2x_4^2 + x_3^2x_4^2 \\
&+ x_1x_2^3 + x_1x_3^3 + x_1x_4^3 
+ x_1^3x_2 + x_1^3x_3 + x_1^3x_4 
+ x_2x_3^3 + x_2x_4^3 
+ x_2^3x_3 + x_2^3x_4 
+ x_3^3x_4   + x_3x_4^3 \\
& + 2(x_1x_2x_3^2  +  x_1^2x_2x_3 
+ x_1x_2^2x_4  +  x_1^2x_2x_4 
+ x_1x_3^2x_4 + x_1x_3x_4^2 
+ x_2^2x_3x_4 + x_2x_3x_4^2 ) \\
&+ 3(x_1x_2^2x_3 + x_1x_2x_4^2 + x_1^2x_3x_4 + x_2x_3^2x_4 ).
\end{align*}
Thus
\begin{align*}
h({\bf x})&= - x_1x_2x_3x_4 
+ x_1^4 + x_2^4 + x_3^4  + x_4^4 
+ x_1^2x_3^2 + x_1^2x_2^2  + x_1^2x_4^2 + x_2^2x_3^2 + x_2^2x_4^2 + x_3^2x_4^2 \\
&- (x_1x_2^3 + x_1x_3^3 + x_1x_4^3  
+ x_1^3x_2   + x_1^3x_3   + x_1^3x_4  
+ x_2x_3^3 + x_2x_4^3
+ x_2^3x_4 + x_2^3x_3
+ x_3x_4^3 + x_3^3x_4  )\\
&+ 2( x_1x_2x_3^2  + x_1x_2^2x_4 + x_1x_3^2x_4 + x_1x_3x_4^2 
+ x_1^2x_2x_3 + x_1^2x_2x_4 + x_2^2x_3x_4  + x_2x_3x_4^2)\\
&- 3(x_1x_2^2x_3 + x_1x_2x_4^2 + x_1^2x_3x_4 + x_2x_3^2x_4).
\end{align*}

Indeed, for $d \in (\Z/5\Z)^\times$, we obtain an explicit bilinear map
$\varphi_d : K^4 \times K^4 \to K^4$
such that
\[
h({\bf x})h({\bf y}) = h(\varphi_d({\bf x},{\bf y}))
\]
for ${\bf x},{\bf y}\in V$.
It is given explicitly by $\varphi_d({\bf x},{\bf y})
= (z_{1,d}, z_{2,d}, z_{3,d}, z_{4,d})$
where
\[
z_{k,d}
= -\left(\sum_{i=1}^{4}x_{-d^{-1}i}y_{k-i} - \sum_{i=1}^{4}x_{-d^{-1}i}y_{-i}\right).
\]
For example, we have for $d=1 \in (\Z/5\Z)^\times$,
\begin{align*}
z_{1,1}
&=-(x_3y_4 + x_2y_3 + x_1y_2 )
+(x_4y_4 + x_3y_3 + x_2y_2 + x_1y_1 ),\\
z_{2,1}
&=-(x_{4}y_1 + x_2y_4 + x_1y_3 )
+(x_{4}y_4 + x_3y_3 + x_2y_2 + x_1y_1 ),\\
z_{3,1}
&=-(x_4y_2 + x_3y_1 + x_1y_4 )
+(x_4y_4 + x_3y_3 + x_2y_2 + x_1y_1 ),\\
z_{4,1}
&=-(x_{4}y_3 + x_3y_2 + x_2y_1)
+(x_{4}y_4 + x_3y_3 + x_2y_2 + x_1y_1 ).
\end{align*}
\end{example}

\begin{example}[$l=7, n=3, m=2, \gamma =3, \beta =2$]\label{ex:l7n3}
\begin{align*}
h({\bf x})=h(x_1,\ldots,x_6)&=S_0 - S_1
\end{align*}
and
\begin{align*}
f_1({\bf x})=f_1(x_1,\ldots,x_6)=S_1-S_2
\end{align*}
where
\begin{align*}
S_0({\bf x})&= \sum_{s + 2t + 4u\equiv 0 \ (7)} x_{s}x_{t}x_{u}\\
&= x_1^3 + x_2^3 + x_4^3 + x_3^3 + x_5^3 + x_6^3\\
&+ 3(x_1x_2x_4) + 3(x_3x_5x_6) + 3(x_1x_2x_6 + x_2x_4x_5+x_1x_3x_4)  + 3(x_1x_5x_6 + x_2x_3x_5 +  x_3x_4x_6),\\
S_1({\bf x})&= \sum_{s + 2t + 4u\equiv 3 \ (7)} x_{s}x_{t}x_{u}\\
&=x_1x_2^2 + 
x_1x_3^2 + 
x_1x_5^2 +
x_2x_4^2 +
x_2x_3^2 +
x_2x_6^2 +
x_3x_4^2 + 
x_3x_5^2 +
x_4x_5^2 + 
x_4x_6^2 + 
x_5x_6^2 \\
&+x_1^2x_4 +
x_1^2x_6 + 
x_2^2x_5 +
x_3^2x_6 \\
&+x_1x_2x_3 + 
x_1x_2x_5 +
x_1x_3x_5 + 
x_1x_3x_6 +
x_1x_4x_6\\
&+ x_2x_3x_4 +
x_2x_3x_6 +
x_2x_4x_6 +
x_2x_5x_6 +
x_3x_4x_5 + 
x_4x_5x_6
\\
&+x_1x_2x_4 +
x_1x_4x_5 +
x_1x_5x_6
+x_2x_3x_5 +
x_3x_4x_6,\\
S_2({\bf x})&= \sum_{s + 2t + 4u\equiv 2 \ (7)} x_{s}x_{t}x_{u}\\
&= 
x_1x_4^2 +
x_1x_6^2 +
x_2 x_5^2+
x_3x_6^2 \\
&+ x_1^2x_2 +
x_1^2x_3 +
x_1^2x_5 +
x_2^2x_3 +
x_2^2x_4 +
x_2^2x_6 +
x_3^2x_4 +
x_3^2x_5 +
x_4^2x_6 +
x_4^2x_5 +
x_5^2x_6 \\
&+ x_1x_2x_3 +
x_1x_2x_5 +
x_1x_3x_5 +
x_1x_3x_6 +
x_1x_4x_6\\
&+ x_2x_3x_4 +
x_2x_3x_6 +
x_2x_4x_6 +
x_2x_5x_6 +
x_3x_4x_5 +
x_4x_5x_6
\\
&+x_1x_2x_6 +
x_1x_3x_4 +
x_1x_4x_5
+ 
x_2x_4x_5 +
x_3x_5x_6.
\end{align*}
Thus
\begin{align*}
h({\bf x})
&=x_1^3+x_2^3+x_3^3+x_4^3+x_5^3+x_6^3\\
&\quad-\bigl(
x_1x_2^2+x_1x_3^2+x_1x_5^2
+x_2x_4^2+x_2x_3^2+x_2x_6^2
+x_3x_4^2+x_3x_5^2+x_4x_5^2+x_4x_6^2+x_5x_6^2\\
&\qquad
+x_1^2x_4+x_1^2x_6+x_2^2x_5+x_3^2x_6\\
&\qquad
+x_1x_2x_3+x_1x_2x_5+x_1x_3x_5
+x_1x_3x_6+x_1x_4x_5+x_1x_4x_6\\
&\qquad
+x_2x_3x_4+x_2x_3x_6+x_2x_4x_6
+x_2x_5x_6+x_3x_4x_5+x_4x_5x_6
\bigr)\\
&\quad
+2(x_1x_2x_4+x_1x_5x_6+x_2x_3x_5+x_3x_4x_6)\\
&\quad
+3(x_1x_3x_4+x_2x_4x_5+x_1x_2x_6+x_3x_5x_6),
\end{align*}
and
\begin{align*}
f_1({\bf x})
&=-x_1^2x_2+x_1x_2^2-x_1^2x_3-x_2^2x_3
+x_1x_3^2+x_2x_3^2\\
&\quad
+x_1^2x_4+x_1x_2x_4-x_2^2x_4-x_1x_3x_4-x_3^2x_4\\
&\quad
-x_1x_4^2+x_2x_4^2+x_3x_4^2-x_1^2x_5+x_2^2x_5\\
&\quad
+x_2x_3x_5-x_3^2x_5-x_2x_4x_5-x_4^2x_5
+x_1x_5^2-x_2x_5^2\\
&\quad
+x_3x_5^2+x_4x_5^2+x_1^2x_6-x_1x_2x_6-x_2^2x_6+x_3^2x_6\\
&\quad
+x_3x_4x_6-x_4^2x_6+x_1x_5x_6-x_3x_5x_6-x_5^2x_6\\
&\quad
-x_1x_6^2+x_2x_6^2-x_3x_6^2+x_4x_6^2+x_5x_6^2.
\end{align*}

Indeed, for $d \in (\Z/7\Z)^\times$, we obtain an explicit bilinear map $\varphi_d :K^6 \times K^6 \rightarrow K^6$ such that 
\begin{align*}
h({\bf x})h({\bf y})=h(\varphi_d ({\bf x}, {\bf y}))
\end{align*}
for ${\bf x},{\bf y} \in V$.
It is given explicitly by
$\varphi_d ({\bf x}, {\bf y}) = (z_{1,d}, z_{2,d}, z_{3,d}, z_{4,d},z_{5,d},z_{6,d})$
where 
\[
z_{k,d}=(-1)^{n-1}\left(\sum_{i=1}^{6}x_{-d^{-1}i}y_{k-i} - \sum_{i=1}^{6}x_{-d^{-1}i}y_{-i}\right).
\]
For example, we have for $d=1 \in (\Z/7\Z)^\times$,
\begin{align*}
z_{1,1}
&=( x_5y_6 + x_4y_5 + x_3y_4 + x_2y_3 + x_1y_2 )
-(x_6y_6 + x_5y_5 + x_4y_4 + x_3y_3 + x_2y_2 + x_1y_1 ),\\
z_{2,1}
&=(x_6y_1 + x_4y_6 + x_3y_5 + x_2y_4 + x_1y_3 )
-(x_6y_6 + x_5y_5 + x_4y_4 + x_3y_3 + x_2y_2 + x_1y_1 ),\\
z_{3,1}
&=(x_6y_2 + x_5y_1 + x_3y_6 + x_2y_5 + x_1y_4 )
-(x_6y_6 + x_5y_5 + x_4y_4 + x_3y_3 + x_2y_2 + x_1y_1 ),\\
z_{4,1}
&=(x_6y_3 + x_5y_2 + x_4y_1 + x_2y_6 + x_1y_5 )
-(x_6y_6 + x_5y_5 + x_4y_4 + x_3y_3 + x_2y_2 + x_1y_1 ),\\
z_{5,1}
&=( x_6y_4 + x_5y_3 + x_4y_2 + x_3y_1 + x_1y_6 )
-(x_6y_6 + x_5y_5 + x_4y_4 + x_3y_3 + x_2y_2 + x_1y_1 ),\\
z_{6,1}
&=(x_6y_5 + x_5y_4 + x_4y_3 + x_3y_2 + x_2y_1)
-(x_6y_6 + x_5y_5 + x_4y_4 + x_3y_3 + x_2y_2 + x_1y_1 ). 
\end{align*}

Moreover, for $d=1\in (\Z/7\Z)^\times$, 
we obtain the following relations 
which also guarantee that $h({\bf x})$ is multiplicative on $V$, i.e. 
$h({\bf x})h({\bf y})=h(\varphi_d({\bf x}, {\bf y}))$ for 
${\bf x},{\bf y} \in V=\{{\bf x}\in K^6\mid f_1({\bf x})=0\}$:
\begin{align*}
h(\varphi_d({\bf x},{\bf y}))&=
h({\bf x})h({\bf y})+f_1({\bf x})h({\bf y})+2f_1({\bf x})f_1({\bf y}),\\
f_1(\varphi_d({\bf x},{\bf y}))&=
h({\bf x})f_1({\bf y})-f_1({\bf x})h({\bf y}).
\end{align*}
\end{example}

\begin{remark}
\label{rem:Sch70}
With the notation of Example~\ref{ex:l7n3}, one has
\[
 h(\varphi_{-1}({\bf x},{\bf y}))
 =h(x)h(y)-2f_1({\bf x})f_1({\bf y}).
\]
Consequently,
\[
 h(\varphi_{-1}({\bf x},{\bf x}))=h({\bf x})^2
 \quad\Longleftrightarrow\quad {\bf x}\in V.
\]
For instance, the point
${\bf a}=(0,0,-1,0,1,0) \notin V$
satisfies $h({\bf a})=1$ and
\[
 \varphi_{-1}({\bf a},{\bf a})=(-2,0,1,0,0,1),
 \qquad
 h(\varphi_{-1}({\bf a},{\bf a}))=-7\neq 1=h({\bf a})^2.
\]
Thus the restriction to $V$ is essential for the
multiplicative identity associated with $\varphi_{-1}$.
Moreover, by Corollary~\ref{cor:no-composition},
the form $h$ permits no bilinear composition identity on the whole ambient space, regardless of the choice of bilinear map.
\end{remark}

\section{Proof of Theorem~\ref{thm:structureV}: The algebraic $K$-tori $W\subset V$}\label{S6}

In this section, we explain how the composition law constructed in Section~\ref{S5}, after restriction to a suitable dense open subset, gives rise to an algebraic $K$-group structure.
Identifying $W$ with a norm-defined subgroup of the Weil restriction $R_{L/K}(\Gm)$ makes the operation $\overset{-1}{\ast}$ transparent and leads to the torus structure asserted in Theorem~\ref{thm:structureV}. 
Thus the torus structure records the group-theoretic content of the coefficient composition constructed in Section~\ref{S5}.
Once this description is in place, Theorem~\ref{thm:CI} follows from the resulting dimension argument.

Throughout this section, we retain the notation and assumptions of Theorem~\ref{mainth3}.
In particular, let $n\geq 2$ be an integer, $l=mn+1$ be an odd prime, 
$K$ be a field such that $\operatorname{char}(K)\notmid l(n!)$
and $\Gal(K(\zeta_l)/K)\simeq (\Z/l\Z)^\times=\langle\gamma\rangle$. 
We also define $L:=K(\zeta_l), \beta:=\gamma^m, M:=L^{\langle\beta\rangle}$.
Then we have
\[
K\subset M\subset L,\qquad [L:K]=l-1=mn,\qquad [L:M]=n,\qquad [M:K]=m.
\]
We retain the notation for the $n$-ic form $h$ and the maps $\varphi_d$ introduced in Theorem~\ref{mainth3}. 
In particular, let
\[
V=\{{\bf x}\in K^{l-1}\mid h_j({\bf x})=0\ (j=1,\dots,m-1)\},
\]
where
\begin{align*}
h_j({\bf x})&=h_j(x_1,\ldots,x_{l-1})\\
&=\left(\sum_{i_0+\beta i_1+\cdots+\beta^{n-1}i_{n-1}\equiv\gamma^j\ (l)}x_{i_0}x_{i_1}\cdots x_{i_{n-1}}\right)-\left(\sum_{i_0+\beta i_1+\cdots+\beta^{n-1}i_{n-1}\equiv\gamma^{j+1}\ (l)}x_{i_0}x_{i_1}\cdots x_{i_{n-1}}\right)
\end{align*}
with subscripts taken modulo $l$ and with the convention $x_0=0$.

For each $d\in(\Z/l\Z)^{\times}$, let
$
\varphi_d: V(K)\times V(K)\to V(K)
$
be the corresponding binary operation on $V(K)$, and define
\[
{\bf x}\overset d\ast {\bf y}:=\varphi_d({\bf x},{\bf y}).
\]
For ${\bf x}=(x_1,\dots,x_{l-1})\in V(K)$, we define
\[
\alpha_{\bf x}:=(-1)^{n-1}\sum_{i=1}^{l-1}x_i\zeta_l^{\,i}\in L,
\]
and
\[
W:=\{{\bf x}\in V\mid N_{L/K}(\alpha_{\bf x})\neq 0\}.
\]
The next proposition shows that $V(K)$ is a left unital groupoid with respect to $\overset{d}{\ast}$. 
In the special case $d=-1$, the operation becomes commutative and associative, which leads to the algebraic $K$-group structure on the open subset $W\subset V$ appearing in Theorem~\ref{thm:structureV} (cf. Hoshi and Kanai \cite[Section 3]{HK22}).

\begin{prop}\label{prop:CA}
For ${\bf x}, {\bf y}, {\bf z} \in V(K)$ and $d,d_1,d_2\in (\Z/l\Z)^\times$,  the following hold: 
\begin{enumerate}
\item [$(1)$]  $({\bf 1}\overset{d}{\ast}{\bf x})_k=x_k$ and $({\bf x}\overset{d}{\ast}{\bf 1})_k=x_{-d^{-1}k}$ where ${\bf 1} := (-1)^{n}(1,1,\dots,1)$. 
The operation $\overset{d}{\ast}$ is left unital, i.e. ${\bf 1}\overset{d}{\ast}{\bf x}= {\bf x}$, and the operation $\overset{-1}{\ast}$ is unital, i.e. ${\bf 1}\overset{-1}{\ast}{\bf x}= {\bf x}\overset{-1}{\ast}{\bf 1}={\bf x}$.
\item [$(2)$] ${\bf x}\overset{d_1}{\ast}({\bf y}\overset{d_2}{\ast}{\bf z})=({\bf x}\overset{-d_2^{-1}d_1}{\ast}{\bf y})\overset{d_2}{\ast}{\bf z}$ where $d_1, d_2 \in (\Z/l\Z)^{\times}$. 
In particular, ${\bf x}\overset{d_1}{\ast}({\bf y}\overset{-1}{\ast}{\bf z})=({\bf x}\overset{d_1}{\ast}{\bf y})\overset{-1}{\ast}{\bf z}$.
\item [$(3)$] $({\bf y}\overset{d}{\ast}{\bf x})_k=({\bf x}\overset{d^{-1}}{\ast}{\bf y})_{-d^{-1}k}$. 
In particular, the operation $\overset{d}{\ast}$ is commutative on $V(K)$ if and only if $d=-1$.
\end{enumerate}
\end{prop}
\begin{proof}
(1) By the convention $y_0=0$, we have
\begin{align*}
({\bf x}\overset{d}{\ast}{\bf y})_k &= (-1)^{n-1}\left( \sum_{i=1}^{l-1}x_{-d^{-1}i}y_{k-i} - \sum_{i=1}^{l-1}x_{-d^{-1}i}y_{-i} \right)\\
&=(-1)^{n-1}\left( x_{-d^{-1}k}y_0 + \sum_{\substack{i=1\\i\neq k}}^{l-1}x_{-d^{-1}i}y_{k-i} - \sum_{i=1}^{l-1}x_{-d^{-1}i}y_{-i} \right)\\
&=(-1)^{n-1}\left( \sum_{\substack{i=1\\i\neq k}}^{l-1}x_{-d^{-1}i}y_{k-i} - \sum_{i=1}^{l-1}x_{-d^{-1}i}y_{-i} \right).
\end{align*}
Hence we have
\begin{align*}
({\bf x}\overset{d}{\ast}{\bf 1})_k &=(-1)^{n-1}\left( \sum_{\substack{i=1\\i\neq k}}^{l-1}x_{-d^{-1}i}\cdot (-1)^{n} - \sum_{i=1}^{l-1}x_{-d^{-1}i}\cdot (-1)^{n} \right)\\
&=(-1)^{(n-1)+(n+1)}x_{-d^{-1}k}\\
&=x_{-d^{-1}k}
\end{align*}
and
\begin{align*}
({\bf 1}\overset{d}{\ast}{\bf x})_k &=(-1)^{n-1}\left( \sum_{\substack{i=1\\i\neq k}}^{l-1}(-1)^{n}\cdot x_{k-i} - \sum_{i=1}^{l-1}(-1)^{n}\cdot x_{-i} \right)\\
&=x_{k}.
\end{align*}

(2) We calculate the left-hand side and the right-hand side separately. 
Define ${\bf v}:={\bf y}\overset{d_2}{\ast}{\bf z}$ and ${\bf w}:={\bf x}\overset{-d_2^{-1}d_1}{\ast}{\bf y}$. 
In the following calculations, all summation indices run from $0$ to $l-1$, with $x_0=y_0=z_0=v_0=w_0=0$.
\begin{align*}
(({\bf x}\overset{-d_2^{-1}d_1}{\ast}{\bf y})\overset{d_2}{\ast}{\bf z})_k&=({\bf w}\overset{d_2}{\ast}{\bf z})_k\\
&=(-1)^{n-1}\left( \sum_{j=0}^{l-1}w_{-d_2^{-1}j}z_{k-j} - \sum_{j=0}^{l-1}w_{-d_2^{-1}j}z_{-j} \right)\\
&=(-1)^{n-1}\left( \sum_{j=0}^{l-1}(-1)^{n-1}\left( \sum_{i=0}^{l-1}x_{d_2d_1^{-1}i}y_{(-d_2^{-1}j)-i} - \sum_{i=0}^{l-1}x_{d_2d_1^{-1}i}y_{-i} \right)z_{k-j}\right.\\
& \left. - \sum_{j=0}^{l-1}(-1)^{n-1}\left( \sum_{i=0}^{l-1}x_{d_2d_1^{-1}i}y_{(-d_2^{-1}j)-i} - \sum_{i=0}^{l-1}x_{d_2d_1^{-1}i}y_{-i} \right)z_{-j} \right)\\
&=\sum_{i=0}^{l-1}\sum_{j=0}^{l-1}x_{d_2d_1^{-1}i}y_{-d_2^{-1}j-i}z_{k-j} - \sum_{i=0}^{l-1}\sum_{j=0}^{l-1}x_{d_2d_1^{-1}i}y_{-i}z_{k-j} \\
&- \sum_{i=0}^{l-1}\sum_{j=0}^{l-1}x_{d_2d_1^{-1}i}y_{-d_2^{-1}j-i}z_{-j} + \sum_{i=0}^{l-1}\sum_{j=0}^{l-1}x_{d_2d_1^{-1}i}y_{-i}z_{-j}\\
&=\sum_{i=0}^{l-1}\sum_{j=0}^{l-1}x_{d_2d_1^{-1}i}y_{-d_2^{-1}j-i}z_{k-j} - \sum_{i=0}^{l-1}\sum_{j=0}^{l-1}x_{d_2d_1^{-1}i}y_{-d_2^{-1}j-i}z_{-j}.
\end{align*}
On the other hand, by setting $j^\prime :=i+j$ and $i^\prime:=-d_2^{-1}i$, we have
\begin{align*}
({\bf x}\overset{d_1}{\ast}({\bf y}\overset{d_2}{\ast}{\bf z}))_k&=({\bf x}\overset{d_1}{\ast}{\bf v})_k\\
&=(-1)^{n-1}\left( \sum_{i=0}^{l-1}x_{-d_1^{-1}i}v_{k-i} - \sum_{i=0}^{l-1}x_{-d_1^{-1}i}v_{-i} \right)\\
&=(-1)^{n-1} \left( \sum_{i=0}^{l-1}x_{-d_1^{-1}i}(-1)^{n-1}\left( \sum_{j=0}^{l-1}y_{-d_2^{-1}j}z_{(k-i)-j} - \sum_{j=0}^{l-1}y_{-d_2^{-1}j}z_{-j} \right) \right.\\
&- \left. \sum_{i=0}^{l-1}x_{-d_1^{-1}i}(-1)^{n-1} \left( \sum_{j=0}^{l-1}y_{-d_2^{-1}j}z_{(-i)-j} - \sum_{j=0}^{l-1}y_{-d_2^{-1}j}z_{-j} \right) \right)\\
&=\sum_{i=0}^{l-1}\sum_{j=0}^{l-1}x_{-d_1^{-1}i}y_{-d_2^{-1}j}z_{k-(i+j)} - \sum_{i=0}^{l-1}\sum_{j=0}^{l-1}x_{-d_1^{-1}i}y_{-d_2^{-1}j}z_{-j} \\
&- \sum_{i=0}^{l-1}\sum_{j=0}^{l-1}x_{-d_1^{-1}i} y_{-d_2^{-1}j}z_{-(i+j)} + \sum_{i=0}^{l-1}\sum_{j=0}^{l-1}x_{-d_1^{-1}i}y_{-d_2^{-1}j}z_{-j}\\
&=\sum_{i=0}^{l-1}\sum_{j=0}^{l-1}x_{-d_1^{-1}i}y_{-d_2^{-1}j}z_{k-(i+j)} - \sum_{i=0}^{l-1}\sum_{j=0}^{l-1}x_{-d_1^{-1}i} y_{-d_2^{-1}j}z_{-(i+j)}\\
&=\sum_{i=0}^{l-1}\sum_{j^\prime=0}^{l-1}x_{-d_1^{-1}i}y_{-d_2^{-1}(j^\prime-i)}z_{k-j^\prime} - \sum_{i=0}^{l-1}\sum_{j^\prime=0}^{l-1}x_{-d_1^{-1}i} y_{-d_2^{-1}(j^\prime-i)}z_{-j^\prime}\\
&=\sum_{i^\prime=0}^{l-1}\sum_{j^\prime=0}^{l-1}x_{d_2d_1^{-1}i^\prime}y_{-d_2^{-1}j^\prime-i^\prime}z_{k-j^\prime} - \sum_{i^\prime=0}^{l-1}\sum_{j^\prime=0}^{l-1}x_{d_2d_1^{-1}i^\prime} y_{-d_2^{-1}j^\prime-i^\prime}z_{-j^\prime}\\
&=(({\bf x}\overset{-d_2^{-1}d_1}{\ast}{\bf y})\overset{d_2}{\ast}{\bf z})_k.
\end{align*}

(3) By substituting $i:=-d^{-1}(k-i^\prime)$ and $j:=d^{-1}j^\prime$, we have
\begin{align*}
({\bf x}\overset{d^{-1}}{\ast}{\bf y})_{-d^{-1}k}&=
 (-1)^{n-1}\left( \sum_{i=1}^{l-1}x_{-di}y_{-d^{-1}k-i} - \sum_{j=1}^{l-1}x_{-dj}y_{-j} \right)\\
 &= (-1)^{n-1}\left( \sum_{i^\prime=1}^{l-1}x_{k-i^\prime}y_{-d^{-1}i^\prime} - \sum_{j^\prime=1}^{l-1}x_{-j^\prime}y_{-d^{-1}j^\prime} \right)\\
 &= (-1)^{n-1}\left( \sum_{i^\prime=1}^{l-1}y_{-d^{-1}i^\prime}x_{k-i^\prime} - \sum_{j^\prime=1}^{l-1}y_{-d^{-1}j^\prime}x_{-j^\prime} \right)\\
 &= ({\bf y}\overset{d}{\ast}{\bf x})_{k}.
\end{align*}
\end{proof}

\begin{remark}
We notice that the operation $\overset{d}{\ast}$ on $V(K)$ is parallel to the $d$-composition $\overset{d}{\ast}$ of matrices in \cite[Section 3]{HK22}. 
\end{remark}

\begin{example}
The multiplication $\overset{-1}{\ast}$ is given by 
\[
\overset{-1}{\ast} : V \times V \to V\ ; \ ( x_1 , \dots , x_{l-1}, y_1 , \dots , y_{l-1} ) \mapsto (z_1 , \dots , z_{l-1})
\]
where $z_{k} := z_{k,-1} = (-1)^{n-1}\left(\sum_{i=1}^{l-1}x_{i}y_{k-i} - \sum_{i=1}^{l-1}x_{i}y_{-i}\right)$.
When $l=3$, we have
\begin{align*}
z_{1}
=(-1)^{n-1} (x_2y_2 - (x_1y_2 + x_2y_1)),\\
z_{2}
=(-1)^{n-1} (x_1y_1 - (x_1y_2 + x_2y_1)).
\end{align*}
When $l=5$, we have
\begin{align*}
z_{1}
&=(-1)^{n-1}( x_2y_4 + x_3y_3 + x_4y_2
-(x_1y_4 + x_2y_3 + x_3y_2 + x_4y_1 )),\\
z_{2}
&=(-1)^{n-1}(x_1y_1 + x_3y_4 + x_4y_3
-(x_1y_4 + x_2y_3 + x_3y_2 + x_4y_1 )),\\
z_{3}
&=(-1)^{n-1}(x_1y_2 + x_2y_1 + x_4y_4
-(x_1y_4 + x_2y_3 + x_3y_2 + x_4y_1 )),\\
z_{4}
&=(-1)^{n-1}(x_1y_3 + x_2y_2 + x_3y_1
-(x_1y_4 + x_2y_3 + x_3y_2 + x_4y_1 )).
\end{align*}
When $l=7$, we have
\begin{align*}
z_{1}
&=(-1)^{n-1}( x_2y_6 + x_3y_5 + x_4y_4 + x_5y_3 + x_6y_2
-(x_1y_6 + x_2y_5 + x_3y_4 + x_4y_3 + x_5y_2 + x_6y_1 )),\\
z_{2}
&=(-1)^{n-1}( x_1y_1 +  x_3y_6 + x_4y_5 + x_5y_4 + x_6y_3
-(x_1y_6 + x_2y_5 + x_3y_4 + x_4y_3 + x_5y_2 + x_6y_1 )),\\
z_{3}
&=(-1)^{n-1}( x_1y_2 + x_2y_1 +  x_4y_6 + x_5y_5 + x_6y_4
-(x_1y_6 + x_2y_5 + x_3y_4 + x_4y_3 + x_5y_2 + x_6y_1 )),\\
z_{4}
&=(-1)^{n-1}( x_1y_3 + x_2y_2 + x_3y_1 + x_5y_6 + x_6y_5
-(x_1y_6 + x_2y_5 + x_3y_4 + x_4y_3 + x_5y_2 + x_6y_1 )),\\
z_{5}
&=(-1)^{n-1}( x_1y_4 + x_2y_3 + x_3y_2 + x_4y_1 + x_6y_6
-(x_1y_6 + x_2y_5 + x_3y_4 + x_4y_3 + x_5y_2 + x_6y_1 )),\\
z_{6}
&=(-1)^{n-1}( x_1y_5 + x_2y_4 + x_3y_3 + x_4y_2 + x_5y_1 
-(x_1y_6 + x_2y_5 + x_3y_4 + x_4y_3 + x_5y_2 + x_6y_1 )).\\ 
\end{align*}
\end{example}

Next, we relate the representation matrix $M_d$ to the field norm and derive an explicit formula for the inverse element.
In what follows, when matrix notation is used, we regard vectors in $K^{l-1}$ as row vectors.

\begin{lem}\label{lem:det-Md}
Let $l$ be an odd prime, let $d\in (\Z/l\Z)^\times$, and let ${\bf x}=(x_1,\dots,x_{l-1})\in K^{l-1}$.
Let $M_d$ be the $(l-1) \times (l-1)$ matrix with entries
\[
M_d[k,j]:=(-1)^{n-1}(x_{d^{-1}(k-j)}-x_{d^{-1}k}),\qquad 1\le k,j\le l-1
\]
where the indices are taken modulo $l$ with the convention $x_0=0$.
Then
\[
\det M_d = N_{L/K}(\alpha_{\bf x}).
\]
\end{lem}

\begin{proof}
We introduce an auxiliary $l \times l$ matrix $\widehat{M}$ indexed by $0 \le k, j \le l-1$, defined by the same formula as $M_d$.
By definition (with indices taken modulo l and the convention $x_0 = 0$),
the $0$-th column of $\widehat{M}$ is identically zero: for every $k$ we have
\[
\widehat{M}[k,0]
= (-1)^{n-1}(x_{d^{-1}(k-0)}-x_{d^{-1}k})
= (-1)^{n-1}(x_{d^{-1}k}-x_{d^{-1}k})=0.
\]
Let $U=\{v\in \overline{K}^{l}\mid v_0=0\}$. Then $U$ is $\widehat M$-stable, and the induced endomorphism on the quotient $\overline{K}^{l}/U$ is zero. 
Moreover, the restriction $\widehat M|_U$ is represented by the $(l-1)\times(l-1)$ minor $M_d$. 
Hence
\[
\det(tI_{l}-\widehat M)=t\,\det(tI_{l-1}-M_d),
\]
so the eigenvalues of $\widehat M$ consist of $0$ together with the eigenvalues of $M_d$ (with multiplicity).
Let $\zeta_l$ be a primitive $l$-th root of unity. 
For each $a \in \{0, \dots, l-1\}$, consider the vector ${\bf v}_a = (1, \zeta_l^a, \zeta_l^{2a}, \dots, \zeta_l^{(l-1)a})$ .
A direct computation yields
\[
{\bf v}_a\widehat{M} = \lambda_a ({\bf v}_a - {\bf v}_0) \quad \text{where } \lambda_a = (-1)^{n-1} \sum_{i=1}^{l-1} x_i \zeta_l^{adi}.
\]
Using the basis $\{{\bf v}_0, {\bf v}_1, \dots, {\bf v}_{l-1}\}$, the matrix representation of $\widehat{M}$ is triangular, and its eigenvalues, counted with multiplicity, are $0,\lambda_1,\ldots,\lambda_{l-1}$.
Therefore,
\[
\det M_d = \prod_{a=1}^{l-1} \lambda_a = (-1)^{(n-1)(l-1)} \prod_{a=1}^{l-1} \left( \sum_{i=1}^{l-1} x_i \zeta_l^{adi} \right).
\]
Since $l$ is an odd prime, $l-1$ is even, so the sign factor $(-1)^{(n-1)(l-1)}$ reduces to $1$.
Furthermore, since $d$ is coprime to $l$, the map $a \mapsto ad \pmod l$ is a permutation of $(\Z/l\Z)^\times$. 
Thus, the product runs over all Galois conjugates of $\alpha_{\bf x}$, yielding the norm $N_{L/K}(\alpha_{\bf x})$.
\end{proof}

Using the determinant formula, we obtain the explicit form of the inverse element.

\begin{prop}\label{prop:inverse}
Fix $d\in (\Z/l\Z)^\times$. Let ${\bf x}\in V(K)$ and assume that $N_{L/K}(\alpha_{\bf x})\neq 0$.
Then there exists a unique element ${\bf x}^{-1}\in V(K)$ such that
\[
{\bf x}\overset{d}{\ast}{\bf x}^{-1}={\bf 1},\qquad {\bf 1}:=(-1)^n(1,\dots,1),
\]
and it is given by
\[
{\bf x}^{-1}={\bf 1}\,\frac{\widetilde{M_d}}{N_{L/K}(\alpha_{\bf x})}
\]
where $\widetilde{M_d}$ denotes the adjugate matrix of $M_d$.
\end{prop}

\begin{proof}
The map ${\bf y}\mapsto{\bf x}\overset{d}{\ast}{\bf y}$ is $K$-linear and, in the row-vector convention above, is represented by the matrix $M_d$; hence
${\bf x}\overset{d}{\ast}{\bf x}^{-1}={\bf 1}$ is equivalent to ${\bf x}^{-1}M_d ={\bf 1}$.
By Lemma~\ref{lem:det-Md}, $\det M_d=N_{L/K}(\alpha_{\bf x})\neq 0$, so $M_d$ is invertible and
\[
{\bf x}^{-1}={\bf 1}M_d^{-1}={\bf 1}\,\frac{\widetilde{M_d}}{\det M_d}
={\bf 1}\,\frac{\widetilde{M_d}}{N_{L/K}(\alpha_{\bf x})}.
\]

To verify ${\bf x}^{-1}\in V(K)$, use the identity
\[
\alpha_{{\bf x}\overset{d}{\ast}{\bf y}}=\sigma_{-d}(\alpha_{\bf x})\,\alpha_{\bf y}
\]
where $\alpha_{\bf x}:=(-1)^{n-1}\sum_{i=1}^{l-1}x_i\zeta_l^{\,i},\ 
\alpha_{\bf y}:=(-1)^{n-1}\sum_{i=1}^{l-1}y_i\zeta_l^{\,i}$.
This follows by a direct calculation from the definition of ${\bf x}\overset{d}{\ast}{\bf y}:=\varphi_{d}({\bf x},{\bf y})$.
By substituting ${\bf y}={\bf x}^{-1}$ and ${\bf x}\overset{d}{\ast}{\bf x}^{-1}={\bf 1}$, we get
$\alpha_{\bf 1}=\sigma_{-d}(\alpha_{\bf x})\alpha_{{\bf x}^{-1}}$.
Since $\alpha_{\bf 1}=1$, we have $\alpha_{{\bf x}^{-1}}=\sigma_{-d}(\alpha_{\bf x})^{-1}$.
Therefore,
\[
N_{L/M}(\alpha_{{\bf x}^{-1}})
= N_{L/M}\bigl(\sigma_{-d}(\alpha_{\bf x})\bigr)^{-1}
= \sigma_{-d}\bigl(N_{L/M}(\alpha_{\bf x})\bigr)^{-1}\in K^\times,
\]
because ${\bf x}\in V(K)$ implies $N_{L/M}(\alpha_{\bf x})\in K^\times$ and $\sigma_{-d}$ acts trivially on $K$.
Hence ${\bf x}^{-1}\in V(K)$.
\end{proof}

\begin{example}[Case $l=5$ and $n=4$]
Consider the case where $l=5$ and $n=4$. 
We have $L=K(\zeta_5)$ and $(-1)^{n-1}=-1$. 
Let $d=1$.
For ${\bf x} = (x_1, x_2, x_3, x_4)$, the matrix $M_1$ is given by
\[
M_1=
\begin{pmatrix}
x_1 & x_1-x_4 & x_1-x_3 & x_1-x_2\\
x_2-x_1 & x_2 & x_2-x_4 & x_2-x_3\\
x_3-x_2 & x_3-x_1 & x_3 & x_3-x_4\\
x_4-x_3 & x_4-x_2 & x_4-x_1 & x_4
\end{pmatrix}
\]
By Lemma~\ref{lem:det-Md}, the determinant is simply the norm:
\[
    \det M_1 = N_{L/K}(x_1\zeta_5 + x_2\zeta_5^2 + x_3\zeta_5^3 + x_4\zeta_5^4).
\]
The inverse vector ${\bf x}^{-1}$ is the solution to ${\bf x}^{-1} M_1  = (1,1,1,1)$.
\end{example}

Note that the multiplication map $({\bf x},{\bf y})\mapsto {\bf x}\overset{-1}{\ast}{\bf y}$ is given by polynomials in the coordinates, hence is a regular morphism on $V\times V$. 
The only possible denominators occur in the inverse map, via the factor $N_{L/K}(\alpha_{\bf x})^{-1}$ in Proposition~\ref{prop:inverse}. 
Accordingly, we remove the boundary $S:=\{x\in V \mid N_{L/K}(\alpha_x)=0\}$
and we define $W:=V\setminus S$.
Then inversion is regular on $W$, and hence $(W,\overset{-1}{\ast})$ is an algebraic $K$-group.
That is, $W=V^\times=V\setminus S\quad {\text where}\quad S:=\{{\bf x}\in V \mid N_{L/K}(\alpha_{\bf x})=0\}$.
By Proposition~\ref{prop:CA}, $\overset{-1}{\ast}$ is commutative and associative.
Thus $W$ is an algebraic $K$-group with respect to $\overset{-1}{\ast}$.

We may therefore view $(W,\overset{-1}{\ast})$ as a commutative algebraic $K$-group. 
Let $T\subset R_{L/K}(\Gm)$ be the $K$-subgroup defined by the norm condition $N_{L/M}(\alpha)\in {\mathbb G}_{m,K}$. 
In the next proof we construct an explicit isomorphism of algebraic $K$-groups $W\simeq T$, and then verify that $T$ is an algebraic $K$-torus.

We now turn to the proof of Theorem~\ref{thm:structureV}.
\begin{proof}[Proof of Theorem~\ref{thm:structureV}]
\quad\par
(1) {\bf Step 1: $W\simeq T$ is an algebraic $K$-group.}\ Keep the notation of Theorem~\ref{thm:structureV}. 
Since $[L:M]=n$, we have $[L:K]=mn=l-1$.
Define a $K$-subgroup of $R_{L/K}(\Gm)$ by
\[
T:=N_{L/M}^{-1}(\mathbb{G}_{m,K})=\Bigl\{\alpha\in R_{L/K}(\Gm)\ \Bigm|\ N_{L/M}(\alpha)\in \mathbb{G}_{m,K}\Bigr\}\subset R_{L/K}(\Gm)
\]
where $\mathbb{G}_{m,K}$ is regarded as the diagonal subtorus of $R_{M/K}(\Gm)$ induced by $K \subset M$.

Let $\Psi:\A^{l-1}_K\to R_{L/K}(\A^1)$ be the $K$-linear isomorphism given on $K$-points by $\Psi({\bf x})=\alpha_{\bf x}$.
By Lemma~\ref{lem:BEW2}, the defining equations of $V$ are equivalent to the condition $N_{L/M}(\Psi({\bf x}))\in K$; in particular, for ${\bf x}\in W$ one has $N_{L/M}(\Psi({\bf x}))\in K^\times$.
Hence $\Psi$ restricts to a morphism $\Psi:W\to T$.
Moreover, by Theorem~\ref{thm:BF2} with $d=-1$, the operation $\overset{-1}{\ast}$ on $W$ is transported by $\Psi$ to multiplication in $R_{L/K}(\Gm)$, and the condition $N_{L/M}(\cdot)\in \mathbb{G}_{m,K}$ is preserved. 
Therefore $\Psi:W\to T$ is an isomorphism of algebraic $K$-groups.
\quad\par
{\bf Step 2: $T$ is an algebraic $K$-torus.}\ We have shown that $\Psi:W\to T$ is an isomorphism of algebraic $K$-groups over $K$.
To complete the proof, it remains to verify that $T$ is an algebraic $K$-torus; equivalently, after base change to an algebraic closure $\overline{K}$ of $K$, we construct an explicit isomorphism
\[
\Phi:(\mathbb{G}_{m,\overline{K}})^{\,1+m(n-1)} \xrightarrow{\ \sim\ } T_{\overline{K}}.
\]

For this, we use the standard description of Weil restriction after base change to $\overline{K}$.
We have
\[
L\otimes_K \overline{K}\simeq \prod_{\sigma\in\Hom_K(L, \overline{K})} \overline{K},\qquad
M\otimes_K \overline{K}\simeq \prod_{\rho\in\Hom_K(M, \overline{K})} \overline{K},
\]
and hence
\[
(R_{L/K}(\Gm))_{ \overline{K}}\simeq (\mathbb{G}_{m, \overline{K}})^{[L:K]}=(\mathbb{G}_{m, \overline{K}})^{mn},\qquad
(R_{M/K}(\Gm))_{ \overline{K}}\simeq (\mathbb{G}_{m, \overline{K}})^{m}.
\]
Choose an ordering of $\Hom_K(L, \overline{K})$ compatible with restriction to $M$, so that for each $\rho\in\Hom_K(M, \overline{K})$ the set of extensions $\sigma$ of $\rho$ has cardinality $n$.
Thus we may write coordinates on $(R_{L/K}(\Gm))_{ \overline{K}}\simeq (\mathbb{G}_{m, \overline{K}})^{mn}$ as $(x_{j,k})$ with $1\le j\le m$ and $1\le k\le n$.
Under these identifications, the base change of the norm map becomes
\[
(N_{L/M})_{ \overline{K}}:(\mathbb{G}_{m, \overline{K}})^{mn}\longrightarrow (\mathbb{G}_{m, \overline{K}})^m,\qquad
(x_{j,k})\longmapsto \Bigl(\prod_{k=1}^n x_{j,k}\Bigr)_{1\le j\le m},
\]
and the inclusion $\mathbb{G}_{m,K}\hookrightarrow R_{M/K}(\Gm)$ becomes the diagonal embedding
\[
\mathbb{G}_{m, \overline{K}}\hookrightarrow (\mathbb{G}_{m, \overline{K}})^m.
\]
Consequently,
\[
T_{ \overline{K}}
=\Bigl\{(x_{j,k})\in(\mathbb{G}_{m, \overline{K}})^{mn}\ \Bigm|\ \prod_{k=1}^n x_{1,k}=\cdots=\prod_{k=1}^n x_{m,k}\Bigr\}.
\]

Define a morphism
$\Phi:(\mathbb{G}_{m,\overline{K}})^{1+m(n-1)}
\to T_{\overline{K}}$
by $(t,(u_{j,k}))\mapsto(x_{j,k})$, where
\[
x_{j,k}=u_{j,k}\ (k=1,\dots,n-1),\qquad
x_{j,n}=\frac{t}{\prod_{k=1}^{n-1}u_{j,k}}.
\]
Then $\prod_{k=1}^n x_{j,k}=t$ for every $j$, and $\Phi$ takes values in $T_{ \overline{K}}$.
Conversely, given $(x_{j,k})\in T_{ \overline{K}}$, we define $t=\prod_{k=1}^n x_{1,k}$ and $u_{j,k}=x_{j,k}$ for $k\le n-1$; the defining equalities imply that each $x_{j,n}$ is recovered by the same formula.
These constructions are regular and mutually inverse, hence $\Phi$ is an isomorphism. Therefore
\[
T_{ \overline{K}}\simeq (\mathbb{G}_{m, \overline{K}})^{\,1+m(n-1)}=(\mathbb{G}_{m, \overline{K}})^{\,l-m},
\]
and $T$ is an algebraic $K$-torus. Since $W\simeq T$, the theorem follows.

(2) For ${\bf p}, {\bf q} \in K^{\times}$, let ${\bf x}\in W_{{\bf p}}(K)$ and ${\bf y}\in W_{{\bf q}}(K)$. 
By the group law on $W$ constructed in (1), we have $x\overset{-1}{\ast}y \in W(K)$.
By Theorem~\ref{mainth3},
\[
h(x\overset{-1}{\ast}y)=h(x)h(y)={\bf pq},
\] 
hence $x\overset{-1}{\ast}y\in W_{{\bf pq}}(K)$.
This proves that the displayed map 
\[
  W_{{\bf p}}(K) \times W_{{\bf q}}(K) \longrightarrow W_{{\bf pq}}(K),
  \qquad (x,y) \longmapsto x \overset{-1}{\ast} y
\]
is well-defined.
\end{proof}

The torus description of the open subset $W$ also shows that the defining equations of $V$ have the expected codimension. 
Thus $V$ is not merely presented by a system of $n$-ic equations; it is a complete intersection.

\begin{thm}\label{thm:CI}
The algebraic $K$-variety $V \subset K^{l-1}$ defined by the system of equations $h_{j}(x)=0$ for $j=1,\dots,m-1$ is a complete intersection of $m-1$ $n$-ics in $\A^{l-1}_K$.
\end{thm}

\begin{proof}
Fix an algebraic closure $\overline{K}$ of $K$. Using the norm description
and the split coordinates $(x_{j,k})$ from the proof of Theorem~\ref{thm:structureV},
put
\[
 P_j:=\prod_{k=1}^{n}x_{j,k}\qquad (1\leq j\leq m).
\]
Then $V_{\overline{K}}$ is defined by $P_1=\cdots=P_m$. On this variety,
set $t=P_1$. Since $N_{L/K}(\alpha_{{\bf x}})=t^m$, the complement of
$W_{\overline{K}}$ is, set-theoretically, the union of the coordinate
subspaces
\[
 B_{\mathbf k}:=
 \{x_{1,k_1}=x_{2,k_2}=\cdots=x_{m,k_m}=0\},
 \qquad \mathbf k\in\{1,\ldots,n\}^{m}.
\]
Each $B_{\mathbf k}$ has dimension $mn-m=l-m-1$. On the other hand,
every irreducible component of $V_{\overline{K}}$ has dimension at least
$(l-1)-(m-1)=l-m$, since it is defined by $m-1$ equations.
Thus no irreducible component is contained in the complement of
$W_{\overline{K}}$, and $W_{\overline{K}}$ is dense in $V_{\overline{K}}$.
By Theorem~\ref{thm:structureV}, $W_{\overline{K}}$ is an irreducible torus of dimension
$l-m$. Hence $V_{\overline{K}}$ is irreducible of the same dimension, and
\[
 \operatorname{codim}_{\A_K^{l-1}}(V)
 =(l-1)-(l-m)=m-1.
\]
Therefore the $m-1$ forms $h_1,\ldots,h_{m-1}$ form a regular sequence.
The corresponding quotient ring over $\overline{K}$ is Cohen--Macaulay
and is generically reduced, since $W_{\overline{K}}$ is smooth and dense.
It has no embedded associated primes, and hence is reduced. It follows
that $V$ is a complete intersection of $m-1$ $n$-ics in
$\A_K^{l-1}$.
\end{proof}

\begin{remark}\label{rem:fibers}
Since $h$ is multiplicative with respect to the group law $\overset{-1}{\ast}$
on $W(K)$, the map $h:W(K)\to K^\times$ is a group homomorphism, and
$W_{\bf p}(K)=h^{-1}({\bf p})$ for ${\bf p} \in K^{\times}$.
Under the isomorphism of Theorem~\ref{thm:structureV}, the fiber $W_1$ is the
norm one torus
\[
W_1=
\ker\bigl(N_{L/M}:R_{L/K}(\Gm)\longrightarrow R_{M/K}(\Gm)\bigr)=R_{M/K}(R_{L/M}^{(1)}(\Gm)).
\]
In particular,
\[
W_1(K)=\{\alpha\in L^\times\mid N_{L/M}(\alpha)=1\}.
\]

The group $W_1(K)$ acts on $W_{\bf p}(K)\times W_{\bf q}(K)$ by
\[
u\cdot(x,y):=
\bigl(x\overset{-1}{\ast}u,\,
y\overset{-1}{\ast}u^{-1}\bigr),
\]
where inverses are taken with respect to the group law $\overset{-1}{\ast}$ on
$W(K)$. 
This action preserves the fibers of the map in
Theorem~\ref{thm:structureV} (2), and each nonempty fiber is a principal homogeneous
space under $W_1(K)$.

Indeed, if $(x,y)$ and $(x',y')$ have the same image in $W_{{\bf p}{\bf q}}(K)$, then the
unique element carrying $(x,y)$ to $(x',y')$ is
\[
u=x^{-1}\overset{-1}{\ast}x'\in W_1(K).
\]
Moreover, if $y_0\in W_{\bf q}(K)$ is fixed, then translation by $y_0$ gives a
bijection of $K$-rational fibers
\[
W_{\bf p}(K)\xrightarrow{\ \sim\ }W_{{\bf p}{\bf q}}(K),\qquad
x\longmapsto x\overset{-1}{\ast}y_0,
\]
with inverse
\[
z\longmapsto z\overset{-1}{\ast}y_0^{-1}.
\]
Thus, whenever $W_{\bf q}(K)$ is nonempty, the map in Theorem~\ref{thm:structureV} (2) is
surjective, and its fibers are precisely the orbits of the $W_1(K)$-action
defined above.
\end{remark}

\begin{remark}[The boundary $S$ and torsion points]\label{rem:boundary-torsion}
\quad \\
\textup{(1)} Under the standing assumption $[K(\zeta_l):K]=l-1$, the map
$\iota : K^{l-1}\to L\ ;\ {\bf x}\mapsto \alpha_{\bf x}$ is a $K$-linear isomorphism.
Hence
\[
S(K)=\{x\in V\mid x\in K^{\,l-1},\ N_{L/K}(\alpha_x)=0\}=\{0\}.
\]

\par\smallskip
\noindent\textup{(2)} Fix an algebraic closure $\overline{K}$ of $K$. Since
\[
N_{L/K}(\alpha)=\prod_{\sigma\in\Hom_K(L, \overline{K})}\sigma(\alpha)
\qquad(\alpha\in L),
\]
we have
\[
\quad \quad \quad S(\overline{K})=\{x\in V( \overline{K})\mid N_{L/K}(\alpha_x)=0\}
\subset
\bigcup_{\sigma\in\Hom_K(L, \overline{K})}\{x\in K^{\,l-1}\otimes_K \overline{K}\mid \sigma(\alpha_x)=0\}.
\]
Thus, viewed inside $K^{\,l-1}\otimes_K \overline{K}\simeq  \overline{K}^{\,l-1}$, the set of
solutions of $N_{L/K}(\alpha_x)=0$ is a union of $ \overline{K}$-hyperplanes given by the
linear equations $\sigma(\alpha_x)=0$.

On the other hand, using the coordinates $(x_{j,k})$ ($1\le j\le m,\ 1\le k\le n$)
from Step~2 in the proof of Theorem~\ref{thm:structureV}, one has
\[
V( \overline{K})=
\Bigl\{(x_{j,k})\in  \overline{K}^{\,mn}\ \Big|\ \prod_{k=1}^n x_{1,k}
=\prod_{k=1}^n x_{2,k}
=\cdots
=\prod_{k=1}^n x_{m,k}\Bigr\}.
\]
If we define $t:=\prod_{k=1}^n x_{1,k}$, then on $V( \overline{K})$ we obtain
\[
N_{L/K}(\alpha_x)=\prod_{j=1}^m\prod_{k=1}^n x_{j,k}
=\prod_{j=1}^m\Bigl(\prod_{k=1}^n x_{j,k}\Bigr)=t^m.
\]
In particular,
\[
S( \overline{K})=\{(x_{j,k})\in V( \overline{K})\mid t=0\}.
\]
Equivalently, $S( \overline{K})$ is the union of the subsets
\[
B_{\mathbf k}( \overline{K}):=
\Bigl\{(x_{j,k})\in V( \overline{K})\ \Big|\ x_{1,k_1}=x_{2,k_2}=\cdots=x_{m,k_m}=0\Bigr\}
\]
where $\mathbf k=(k_1,\dots,k_m)\in\{1,\dots,n\}^m$,
so there are $n^m$ maximal boundary pieces inside $V( \overline{K})$ (each obtained by
forcing one coordinate to be $0$ in each block).

\par\smallskip
\noindent\textup{(3)} Under the identification $W_{ \overline{K}}\simeq T_{ \overline{K}}\simeq
\mathbb{G}_{m, \overline{K}}\times(\mathbb{G}_{m, \overline{K}})^{m(n-1)}$ from the proof of Theorem~\ref{thm:structureV},
an element of order $k$ in $W( \overline{K})$ has all coordinates in the group
$\mu_k( \overline{K})=\{\xi\in \overline{K}^\times\mid \xi^k=1\}$.
In particular, if $x\in W( \overline{K})$ has order $k$, then $N_{L/K}(\alpha_x)\in \mu_k( \overline{K})$.
Hence a point $x\in W( \overline{K})$ with $N_{L/K}(\alpha_x)$ not a root of unity
cannot be a torsion point of $W( \overline{K})$.

\par\smallskip
\noindent\textup{(4)} In the examples of Section~~\ref{S4-3}, the preceding criterion shows
that the classical trivial solutions need not be torsion points.  
For example, with $K=\Q$, $l=5$, $n=2$, a prime $p\equiv1\pmod5$, and $r\geq1$, the two quadratic equalities in Dickson's system \eqref{eq:DS5p}, written in coordinates ${\bf u}=(u_1,u_2,u_3,u_4)$ and normalized by $16P=u_1^2+125u_2^2+50u_3^2+50u_4^2$, have the solution ${\bf u}=(4p^r,0,0,0)$ for $P=p^{2r}$.
Under the coefficient change of variables in Section~\ref{S4-3}, these Dickson coordinates correspond to the power-basis coefficient vector ${\bf a}=(a_1,a_2,a_3,a_4)=(-p^r,-p^r,-p^r,-p^r)$.
Since $\alpha_{\bf a}=-\sum_{i=1}^4a_i\zeta_5^i=-p^r$, we have $N_{\Q(\zeta_5)/\Q}(\alpha_{\bf a})=p^{4r}$, so ${\bf a}\in W(\Q)$ is not a torsion point by (3).
This solution satisfies only the two quadratic equalities, since $u_1\equiv4\pmod5$ and $p\mid u_1^2-125u_2^2$, so the additional congruence and nondivisibility conditions in \eqref{eq:DS5p} are not satisfied.
\end{remark}

\begin{remark}[The toric structure and the Gorenstein Fano variety $X$]\label{rem:toric-fano}
The description of $V_{\overline{K}}$ in
Remark~\ref{rem:boundary-torsion}~(2) shows that
$V_{\overline{K}}$ is regular in codimension one. On the other hand,
by Theorem~\ref{thm:CI}, $V_{\overline{K}}$ is a complete intersection
in $\A_{\overline{K}}^{l-1}$, hence Cohen--Macaulay.
Therefore $V_{\overline{K}}$ is normal by Serre's criterion
(cf.~\cite[Chapter~II, Theorem~8.22A and Proposition~8.23]{Har77}).

Since the operation
$\overset{-1}{\ast}:V\times V\to V$ restricts to the group law on
$W$, it defines an action of $W$ on $V$. Hence $V$ is an affine
toric $K$-variety with dense torus $W$.

Since the defining forms $h_1,\ldots,h_{m-1}$ are homogeneous of
degree $n$, let
\[
X:=
\operatorname{Proj}
\bigl(K[x_1,\ldots,x_{l-1}]/(h_1,\ldots,h_{m-1})\bigr)
\subset\PS_K^{l-2}.
\]
The scalar subtorus ${\mathbb G}_{m,K}\subset W$ acts on $V$ by
scalar multiplication. Hence $X$ is a toric $K$-variety with dense
torus $W/{\mathbb G}_{m,K}$.

By Theorem~\ref{thm:CI}, $X$ is a complete intersection of $m-1$
hypersurfaces of degree $n$ in $\PS_K^{l-2}$. In particular,
\[
\dim X=l-m-1=m(n-1).
\]
Since $X$ is a complete intersection, it is Cohen--Macaulay, and
Hartshorne \cite[Chapter~III, Theorem~7.11]{Har77} gives
\[
\omega_X
\simeq
\mathcal O_X\bigl(-(l-1)+(m-1)n\bigr)
=
\mathcal O_X(-n).
\]
Thus $\omega_X$ is invertible. Since $\mathcal O_X(1)$ is ample,
$\omega_X^{-1}\simeq\mathcal O_X(n)$ is ample.

We call a normal projective $K$-variety $X$ a
\emph{Gorenstein Fano $K$-variety} if its dualizing sheaf
$\omega_X$ is invertible and $\omega_X^{-1}$ is ample.
Hence $X$ is a Gorenstein toric Fano $K$-variety.
\end{remark}

Theorem~\ref{thm:structureV} (2) and Remark~\ref{rem:fibers} describe the compatibility between the group law on $W$ and 
the fibers $W_{\bf p}(K)$ of the multiplicative form $h$: 
the product map sends $W_{\bf p}(K)\times W_{\bf q}(K)$ to $W_{{\bf p}{\bf q}}(K)$, 
and each nonempty fiber of this map is a principal homogeneous space under $W_1(K)$, 
where $W_1$ is the norm one torus.

For every $n\ge 2$, the dense open torus $W\subset V$ provides a group-theoretic way to produce new elements of the fibers $W_{\bf p}(K)$: they are transported by the power maps $[k]:W\to W$ $(k\geq1)$, and more generally by the group law on $W$. 
Thus the explicit multiplicative $n$-ic forms constructed here give rise to higher-degree Diophantine systems whose solutions over $K$, on the dense open part $W$, are organized by this torus structure.

In the quadratic case $n=2$, this recovers the solution lifting formulas of
Hoshi and Kanai~\cite{HK22}. 
More precisely, in the prime-power setting considered there, for instance when ${\bf p}={\bf q}=p^r$, 
iteration of the power maps gives the geometric counterpart of the lifting procedure obtained
from Davenport--Hasse's lifting theorem for Jacobi sums and from the corresponding $k$-fold products of multiplication matrices.

\begin{acknowledgment}
Before getting a Ph.D., the first-named author stayed at the
University of Regensburg for one year: 2004--2005. 
He thanks his teacher Manfred Knebusch 
who led him to the extensive world of quadratic forms. 
\end{acknowledgment}


\end{document}